\documentclass[11pt,twoside]{amsart}

\usepackage{latexsym}
\usepackage{amssymb}
\usepackage{vmargin}
\usepackage{amscd}
\usepackage{stmaryrd}
\usepackage{euscript}
\usepackage{mathrsfs} 
\usepackage[all]{xy}

\setmargins{32mm}{20mm}{14.6cm}{22cm}{1cm}{1cm}{1cm}{1cm}

\usepackage{xr-hyper}
\usepackage{hyperref}

\newcommand\da{\!\downarrow\!}

\newcommand\la{\leftarrow}

\newcommand\id{\mathrm{id}}

\newcommand\ten{\otimes}

\newcommand\vareps{\varepsilon}

\newcommand\CC{\mathrm{C}}

\newcommand\CCC{\mathrm{CC}}
\newcommand\GT{\mathrm{GT}}
\newcommand\Levi{\mathrm{Levi}}

\renewcommand\H{\mathrm{H}}
\newcommand\z{\mathrm{Z}}
\renewcommand\b{\mathrm{B}}

\newcommand\HH{\mathrm{HH}}

\newcommand\Z{\mathbb{Z}}
\newcommand\Q{\mathbb{Q}}

\newcommand\R{\mathbb{R}}

\newcommand\bG{\mathbb{G}}

\newcommand\C{\mathcal{C}}

\newcommand\cD{\mathcal{D}}

\newcommand\cP{\mathcal{P}}
\newcommand\cQ{\mathcal{Q}}

\newcommand\cT{\mathcal{T}}
\newcommand\cU{\mathcal{U}}

\newcommand\cW{\mathcal{W}}

\newcommand\cPer{\mathcal{P}\!\mathit{er}}

\renewcommand\O{\mathscr{O}}

\renewcommand\P{\mathscr{P}}

\newcommand\sA{\mathscr{A}}

\newcommand\sD{\mathscr{D}}

\newcommand\sO{\mathscr{O}}
\newcommand\sP{\mathscr{P}}

\newcommand\fX{\mathfrak{X}}
\newcommand\fY{\mathfrak{Y}}

\renewcommand\L{\Lambda}

\newcommand\g{\mathfrak{g}}

\newcommand\ft{\mathfrak{t}}

\renewcommand\hom{\mathscr{H}\!\mathit{om}}

\newcommand\cHom{\mathcal{H}\!\mathit{om}}
\newcommand\cDiff{\mathcal{D}\!\mathit{iff}}

\newcommand\Ass{\mathrm{Ass}}
\newcommand\Br{\mathrm{Br}}

\newcommand\Hom{\mathrm{Hom}}

\newcommand\map{\mathrm{map}}

\newcommand\EFC{\mathrm{EFC}}

\newcommand\HHom{\underline{\mathrm{Hom}}}

\newcommand\DDer{\underline{\mathrm{Der}}}

\newcommand\Ext{\mathrm{Ext}}

\newcommand\Aut{\mathrm{Aut}}

\newcommand\cone{\mathrm{cone}}

\newcommand\per{\mathrm{per}}

\newcommand{\brh}{\llbracket \hbar \rrbracket}
\newcommand{\brhh}{\llbracket \hbar^2 \rrbracket}

\newcommand\Co{\mathrm{Co}}
\newcommand\CoS{\mathrm{CoS}}

\newcommand\Top{\mathrm{Top}}

\newcommand\Gp{\mathrm{Gp}}

\newcommand\Spec{\mathrm{Spec}\,}

\newcommand\Set{\mathrm{Set}}

\newcommand\Pol{\mathrm{Pol}}
\newcommand\cPol{\mathcal{P}ol}

\newcommand\poly{\mathrm{poly}}
\newcommand\Com{\mathrm{Com}}

\newcommand\ad{\mathrm{ad}}

\newcommand\Lim{\varprojlim}
\newcommand\LLim{\varinjlim}

\newcommand\ho{\mathrm{ho}\!}
\newcommand\into{\hookrightarrow}
\newcommand\onto{\twoheadrightarrow}
\newcommand\abuts{\implies}
\newcommand\xra{\xrightarrow}
\newcommand\xla{\xleftarrow}

\newcommand\bt{\bullet}
\newcommand\by{\times}

\newcommand\mc{\mathrm{MC}}
\newcommand\mmc{\underline{\mathrm{MC}}}

\newcommand\Perf{\mathrm{Perf}}

\newcommand\Rees{\mathrm{Rees}}
\newcommand\Symm{\mathrm{Symm}}

\newcommand\et{\acute{\mathrm{e}}\mathrm{t}}

\newcommand\Tate{\mathrm{Tate}}

\newcommand\Tot{\mathrm{Tot}\,}
\newcommand\diag{\mathrm{diag}\,}

\newcommand\ind{\mathrm{ind}}
\newcommand\pro{\mathrm{pro}}

\newcommand\pd{\partial}

\newcommand\half{\frac{1}{2}}

\newcommand\gr{\mathrm{gr}}

\newcommand\Fil{\mathrm{Fil}}

\newcommand\Lie{\mathrm{Lie}}

\newcommand\op{\mathrm{opp}}

\newcommand\co{\colon\thinspace}

\newcommand\oR{\mathbf{R}}

\newcommand\oL{\mathbf{L}}

\newcommand\on{\mathbf{n}}
\newcommand\om{\mathbf{m}}

\newcommand\uleft\underleftarrow
\newcommand\uline\underline
\newcommand\uright\underrightarrow

\newcommand{\tps}{\texorpdfstring}

\DeclareMathOperator{\hatHHom}{\uline{\mathrm{H}\widehat{\mathrm{o}}\mathrm{m}}}
\DeclareMathOperator{\hatTot}{\mathrm{T}\widehat{\mathrm{o}}\mathrm{t}}

\DeclareMathOperator{\hatDDer}{\uline{\mathbb{D}{\widehat{\mathrm{e}}\mathrm{r}}}}

\newtheorem{theorem}{Theorem}[section]
\newtheorem{proposition}[theorem]{Proposition}
\newtheorem{corollary}[theorem]{Corollary}

\newtheorem{lemma}[theorem]{Lemma}
\newtheorem*{theorem*}{Theorem}
\newtheorem*{proposition*}{Proposition}
\newtheorem*{corollary*}{Corollary}
\newtheorem*{lemma*}{Lemma}
\newtheorem*{conjecture*}{Conjecture}

\theoremstyle{definition}
\newtheorem{definition}[theorem]{Definition}
\newtheorem*{definition*}{Definition}

\newtheorem*{notation*}{Notation}

\theoremstyle{remark}

\newtheorem{remark}[theorem]{Remark}
\newtheorem{remarks}[theorem]{Remarks}

\newtheorem*{example*}{Example}
\newtheorem*{examples*}{Examples}
\newtheorem*{remark*}{Remark}
\newtheorem*{remarks*}{Remarks}
\newtheorem*{exercise*}{Exercise}
\newtheorem*{property*}{Property}
\newtheorem*{properties*}{Properties}

\begin{document}

\begin{abstract}
We prove that every $0$-shifted Poisson structure on a derived Artin $n$-stack admits a curved $A_{\infty}$ deformation quantisation whenever the stack has perfect cotangent complex; in particular, this applies to  LCI schemes, where it gives a DQ algebroid quantisation. Whereas the Kontsevich--Tamarkin approach to quantisation for smooth varieties hinges on invariance of the Hochschild complex under affine transformations, we instead exploit the observation that  the Hochschild complex carries an anti-involution, and that such anti-involutive deformations of the complex of polyvectors are essentially unique. We also establish analogous statements for deformation quantisations in $\C^{\infty}$ and analytic settings. 
\end{abstract}

\title[Quantisation of derived Poisson  structures]{Quantisation of derived Poisson  structures}

\author{J.P.Pridham}


\maketitle

\section*{Introduction}

A deformation quantisation of a Poisson structure on a geometric object $Y$
is a non-commutative deformation, parametrised by power series in $\hbar$, of the
functions $\sO_Y$ on $Y$, such that the classical limit $\hbar \to 0$ recovers the Poisson structure.
Classically, this means looking at associative deformations $\star_{\hbar}$ of the
multiplication on $\sO_Y$, with the Poisson bracket then given by $\{a,b\}:=\lim_{\hbar \to 0} \frac{a\star_{\hbar}  b -b \star_{\hbar}a}{\hbar}$. In algebro-geometric settings, the local deformation quantisations tend not to glue strictly, leading to the notion of DQ algebroid quantisations as deformations of the algebra regarded as a  linear category on one object, so the cocycle condition only holds up to an inner automorphism satisfying further coherence relations. In derived geometry, $\sO_Y$ is homologically enriched so  quantisations have to be formulated in terms of more exotic algebraic structures.

By \cite{kontsevichPoisson}, every Poisson manifold admits a deformation quantisation.
For smooth algebraic varieties in characteristic $0$, Kontsevich and Yekutieli proved an analogous statement, showing in \cite{kontsevichDQAlgVar,yekutieliDQAG} that all Poisson structures admit DQ algebroid quantisations. Via local choices of connections, the question reduced 
to constructing quantisations of  affine space. These could then be handled as in \cite{tamarkinOperadicKontsevichFormality,KontsevichOperads,yekutieliTwistedDQ,vdBerghGlobalDQ}: formality of the $E_2$ operad associates  to the Hochschild complex a deformation of the $P_2$-algebra of multiderivations, and invariance under affine transformations ensures that it is the unique deformation. The same argument  extends to  graded manifolds \cite{CattaneoFelder}, but requires the differential to be trivial.\footnote{Shortly after this manuscript first appeared,  \cite{LiaoStienonXuFormalityDGMfd} established a Kontsevich--Duflo formality quasi-isomorphism for  finite-dimensional dg $\C^{\infty}$-manifolds, indirectly induced from formality for graded manifolds. A similar argument should apply to derived affine schemes with perfect cotangent complexes.}

We now consider generalisations of this question to singular schemes and more generally to derived stacks in characteristic $0$, 
considering quantisations of $0$-shifted Poisson structures in the sense of \cite{poisson,CPTVV} and their analytic and $\C^{\infty}$ analogues 
\cite{DStein,DQDG}. 
%
For positively shifted structures, the analogous question is a formality, following from the equivalence $E_{n+1}\simeq P_{n+1}$ of operads.  Quantisations  for non-degenerate  $0$-shifted Poisson structures were established in \cite{DQnonneg}, and we now consider degenerate quantisations as well, addressing the remaining unsolved case of \cite[Conjecture 5.3]{toenICM}\footnote{Although the statement for derived DM stacks was specifically claimed in \cite[Theorem 5.4]{toenICM}, it is not even stated in the reference provided, whose proof is only relevant to strictly positive shifts,  and it lacks the constraints on the cotangent complex needed to rule out known affine counterexamples.}, long regarded as the hardest\footnote{See for instance the survey \cite[\S 3]{PantevVezzosiUtah}, but beware 
that deformation quantisations for   negative shifts, starting with the BV quantisations of \cite{DQvanish}, take a far more subtle form than proposed there, and that the declared  aim of \cite[\S 2]{PantevVezzosiUtah} is somewhat moot given the uncited \cite{poisson}.}.

The construction of non-degenerate quantisations in \cite{DQnonneg,DQLag} only relied on the fact that the Hochschild complex is an anti-involutive  deformation of the complex of multiderivations. Our strategy in this paper is closer to \cite{tamarkinOperadicKontsevichFormality,KontsevichOperads} in that we  establish  an equivalence between the two complexes. 
As in \cite{DQnonneg,DQLag}, the key observation is still that the Hochschild complex of a differential graded-commutative algebra (CDGA) carries an anti-involution corresponding to the endofunctor on deformations sending an algebra to its opposite. Via a  formality quasi-isomorphism for the $E_2$ operad corresponding to an even associator, 
the Hochschild complex becomes an anti-involutive deformation of the $P_2$-algebra of multiderivations. 

We  show (Corollary \ref{fildefcor1})
that such deformations  are essentially unique whenever the complexes of polyvectors and of multiderivations are quasi-isomorphic. 
This condition is satisfied when the CDGA has perfect cotangent 
complex,  so   gives the following:
\begin{theorem*}[\ref{fildefhochthm1}] 
Let $A$ be a  CDGA, $\C^{\infty}$-DGA or DGA with entire functional calculus. Whenever  
$A$ has  perfect cotangent complex in the relevant theory, any choice of even $1$-associator yields a natural quasi-isomorphism
\begin{align*}
 D^{\poly}_{\oplus}(A)_{[-1]}&\simeq \bigoplus_{p \ge 0}  \oR\HHom_A(\oL\Omega^p_A ,A)_{[p-1]}
 \end{align*}
of  differential graded Lie algebras (DGLAs) between 
the relevant complex of polydifferential operators   and the algebra of derived polyvectors, compatible with canonical filtrations and involutions. 
In the  CDGA case, we can replace  $D^{\poly}_{\oplus}(A)$ with the cohomological Hochschild complex $\CCC_{R,\oplus}(A)$.
\end{theorem*}

The following underived consequence is a special case of   Corollary \ref{affquantcor} (applied to a cofibrant replacement and taking $\H_0$ of the output). It was previously known only for non-singular spaces.
\begin{corollary*}
 Take the ring $A$ of functions on an affine scheme, affine $\C^{\infty}$-space, or affine analytic space, with only local complete intersection (LCI) singularities, together with a Poisson structure $\{-,-\}$ on $A$. There then  exist associative deformations $\star_{\hbar}$ of the product on $A\brh$ satisfying $\{a,b\} \equiv \frac{a\star_{\hbar}  b -b \star_{\hbar}a}{\hbar} \mod \hbar$. The coefficients of $\star_{\hbar}$ are given by differential operators, and the deformation can be chosen to be anti-involutive in the sense that $b\star_{\hbar} a = a\star_{-\hbar} b$.
\end{corollary*}

Theorem \ref{fildefhochthm1} enjoys sufficient
 functoriality to yield the following global consequence (a special case of Corollary \ref{DMquantcor}, which also considers derived objects and gives a complete parametrisation).
\begin{corollary*}
Given an algebraic, $\C^{\infty}$, or analytic space or Deligne--Mumford stack  $\fX$ with only LCI singularities, every Poisson structure on the sheaf $\sO_{\fX}$ of functions admits   DQ algebroid deformations $\sA_{\hbar}$  of $\sO_{\fX}\brh$. These can be chosen to be self-dual in the sense that  $\sA_{-\hbar}\simeq \sA_{\hbar}^{\op}$.
\end{corollary*}


Deformation quantisations of $1$-shifted co-isotropic structures also follow as an immediate consequence (Corollary \ref{coisocor1}) of Theorem \ref{fildefhochthm1}.  A $1$-shifted co-isotropic structure on a morphism $\fX \to \fY$ is a notion corresponding to  a $1$-shifted Poisson structure $\varpi$ on $\fY$ acting on a $0$-shifted Poisson structure $\pi$ on $\fX$.  The deformation quantisations established in Corollary \ref{coisocor1} when $\fX$ has perfect cotangent complex consist of an almost commutative $E_2$-algebra deformation of $(\sO_{\fY},\varpi)$  acting appropriately on a deformation quantisation of $(\sO_{\fX},\pi)$. 

In Section \ref{Artinsn}, we extend these results to derived Artin $n$-stacks. This follows by essentially the same argument, but is much more technically complicated because of the subtleties in formulating polyvectors and Hochschild complexes for Artin stacks.
Locally, these are defined as Tate (i.e. sum-product) realisations $\hatTot$ of double complexes arising from  formally \'etale charts by stacky thickenings of derived affine schemes (Lie algebroids, broadly speaking).   Those total complexes do not satisfy the conditions of Corollary \ref{fildefcor1}, so we introduce an intermediate category through which $\hatTot$ factorises, and in which the $P_2$-algebra of polyvectors has no non-trivial involutive deformations.
%

For stacky thickenings of derived affine schemes, this leads to an equivalence (Theorem \ref{fildefhochthm2}) between polyvectors and polydifferential operators or the Hochschild complex, generalising Theorem \ref{fildefhochthm1} above. 
Via a form of \'etale functoriality, that yields the following corollary, where $E_1$ quantisations of $\fX$ are almost commutative curved $A_{\infty}$-algebra deformations of the rings of functions on formally \'etale stacky derived affine charts.
\begin{corollary*}[\ref{Artinquantcor}] 
Given a  
derived Artin $n$-stack $\fX$ (algebraic, $\C^{\infty}$ or analytic) with  perfect cotangent complex, 
any even associator yields a map
from the space of $0$-shifted Poisson structures on $\fX$ to 
the space of self-dual  $E_1$ quantisations of $\fX$.

These quantisations give rise to curved $A_{\infty}$ deformations of the dg category  of perfect $\O_{\fX}$-complexes,  $\hbar$-semilinearly anti-involutive with respect to the dg endofunctor $\hom_{\sO_{\fX}}(-,\sO_{\fX})$, whenever $\fX$ is strongly quasi-compact (i.e. quasi-compact, quasi-separated \ldots).
\end{corollary*}

Any $1$-shifted Lagrangian structure on a morphism gives rise to a $1$-shifted co-isotropic structure by \cite[Theorem 4.22]{MelaniSafronovII}. Examples  of these from  \cite{calaqueLagrangian} include quotient stacks $[Y/G]$ when $Y$ is equipped with a Hamiltonian or quasi-Hamiltonian structure (Lagrangian over $[\g^*/G]$ and $[G/G^{\ad}]$, respectively), and the derived moduli stack $\Perf(\Sigma)$ of perfect complexes on a del Pezzo surface $\Sigma$ (Lagrangian over $\Perf(E)$ for an elliptic curve $ E \subset \Sigma$). A $1$-shifted co-isotropic structure on a morphism gives  a $0$-shifted Poisson structure on its source,  to which Corollary \ref{Artinquantcor} can then be applied. We can however say more, with Corollary \ref{coisocor2} showing that,  when the source has perfect cotangent complex, $1$-shifted co-isotropic structures admit deformation quantisations in the form of an $E_2$ quantisation of the target acting on an $E_1$ quantisation of the source. 

%

I would like to thank Ted Voronov for alerting me to a missing symmetrisation in the definition of Poisson $L_{\infty}$-morphisms, Brent Pym and Travis Schedler for letting me know of the examples in \cite{mathieu}, and the anonymous referee for suggesting several improvements.


\subsection*{Notation and terminology}

We write CDGAs (commutative differential graded algebras) and DGAAs (differential graded associative algebras) as chain complexes (homological grading), and denote the differential on a chain complex by $\delta$. 

The graded vector space underlying a chain (resp. cochain) complex $V$ is denoted by $V_{\#}$ (resp. $V^{\#}$). Since we often have to work with chain and cochain structures separately, we denote shifts as subscripts and superscripts, respectively, so $(V_{[i]})_n:= V_{i+n}$ and $(V^{[i]})^n:= V^{i+n}$.

Given a DGAA $A$, and $A$-modules $M,N$ in chain complexes, we write $\HHom_A(M,N)$ for the chain complex given by
\[
 \HHom_A(M,N)_i= \Hom_{A_{\#}}(M_{\#},N_{\#[i]}),
\]
with differential $\delta f= \delta_N \circ f \pm f \circ \delta_M$,
where $V_{\#}$ denotes the graded vector space underlying a chain complex $V$.

 When we need to compare chain and cochain complexes, we  make use of the equivalence  $u$ from chain complexes to cochain complexes given by $(uV)^i := V_{-i}$, and refer to this as rewriting the chain complex as a cochain complex (or vice versa). On suspensions, this has the effect that $u(V_{[n]}) = (uV)^{[-n]}$.

\tableofcontents

\section{Involutively filtered deformations of  Poisson algebras}\label{fildefpoisssn}


We will assume that all filtrations are increasing and exhaustive, unless stated otherwise. An action of the algebraic group $\bG_m$ on a vector space $V$ is equivalent to a decomposition $V=\bigoplus_{i} \cW_iV$; given a $\bG_m$-equivariant vector space $V$, we then 
define a weight filtration $W$ on $V$ by setting $W_nV:=\bigoplus_{i \le n} \cW_iV$.  For $\bG_m$-equivariant complexes $U,V$, we write $\cW_i\HHom(U,V)$ for the complex $\prod_j \HHom(\cW_jU, \cW_{i+j}V)$ of homomorphisms of weight $i$, with similar conventions for complexes of derivations etc.; beware that the inclusion  $\bigoplus_i \cW_i\HHom(U,V) \to  \HHom(U,V)$ is not surjective in general.

\subsection{Involutively filtered deformations of \tps{$\cP$}{P}-algebras}

We now adapt an idea developed in \cite[\S 14.3.2]{mhs2}, using Rees constructions to interpret the problem of recovering a filtered algebra from its associated graded algebra as a deformation problem.

\begin{definition}\label{involutivelyfiltereddef}
 We say that a vector space $V$ is involutively filtered if it is equipped with a filtration $W$ and an involution $e$ which preserves $W$ and acts on $\gr^W_iV$ as multiplication by $(-1)^i$.

We say that a chain complex $V_{\bt}$ is quasi-involutively filtered if it is equipped with a  filtration $W$ by subcomplexes and an involution $e$ which preserves $W$ and acts on $\H_*(\gr^W_iV)$ as multiplication by $(-1)^i$.
\end{definition}
Observe that if $V$ is involutively filtered, then the involution gives an eigenspace  decomposition $V= V^{e=1}\oplus V^{e=-1}$. Because  $(\gr^W_{i-1}V)^{e=(-1)^i}=0$, we have $W_{2j+1}V^{e=1}=W_{2j}V^{e=1}$ and $W_{2j}V^{e=-1}=W_{2j-1}V^{e=-1}$. If $V_{\bt}$ is quasi-involutively filtered, we similarly have that $W_{2j}V^{e=1}\subset W_{2j+1}V^{e=1}$ and $W_{2j-1}V^{e=-1}\subset W_{2j}V^{e=-1} $ are quasi-isomorphic subcomplexes. 

\begin{definition}\label{Wedef}
 Given a quasi-involutively filtered chain complex $(V_{\bt},W,e)$, define the  involutively filtered chain complex  $(V_{\bt},W^e,e)$   by setting  
 \[
 W^e_nV_{\bt}:=  (W_nV_{\bt})^{e=(-1)^n} \oplus (W_{n-1}V_{\bt})^{e=(-1)^{n-1}}.
 \]
\end{definition}

\begin{lemma}\label{quasiinvlemma}
 After localisation at  filtered quasi-isomorphisms, the inclusion functor from    involutively filtered  chain complexes to  quasi-involutively filtered  chain complexes gives an equivalence of $\infty$-categories.
\end{lemma}
\begin{proof}
We can identify the involutively filtered objects as  those   $(V_{\bt},W,e)$ for which  the natural transformation $\vareps_V \co (V_{\bt},W^e,e) \to  (V_{\bt},W,e) $ is an isomorphism, so the functor $ (V_{\bt},W,e) \mapsto (V_{\bt},W^e,e)$ is right adjoint to the inclusion functor, with $\vareps$ the co-unit of the adjunction. 
 
It remains to show that $\vareps_V$ is a filtered quasi-isomorphism for all quasi-involutively filtered  chain complexes. Observe that $W_{n-1}V_{\bt}$ is contained in $ W_n^eV_{\bt}$, with quotient isomorphic to $\gr^W_nV_{\bt}^{e=(-1)^n}$, which implies that the inclusion $W_n^eV_{\bt} \into W_nV_{\bt}$ is indeed  a quasi-isomorphism by quasi-involutivity of $W$.
 \end{proof}


\begin{definition}\label{hbardef}
 Define the $\bG_m$-equivariant algebra $\Q[\hbar]$ by setting $\hbar$ to have  weight $1$ 
for the $\bG_m$-action (and implicitly homological degree $0$). 
\end{definition}


\begin{definition}\label{reesdef} 
 For a quasi-involutively filtered chain complex $(V, W_iV,e)$ over $\Q$, we define
 \[
 \Rees^e_W(V):=\bigoplus_n  \hbar^n \{ v \in  W_nV~:~e(v)=(-1)^nv\}.
 \]
This is a $\bG_m$-equivariant chain complex of  $\Q[\hbar^2]$-modules, where the $\bG_m$-action is induced by the action on $\hbar$.
 \end{definition}
 

\begin{lemma}\label{reeslemma}
 The involutive Rees functor of Definition \ref{reesdef} gives an equivalence between the category of 
 involutively filtered $\Q$-vector spaces   and the category of flat $\bG_m$-equivariant  $\Q[\hbar^2]$-modules.
 
 For quasi-involutively filtered complexes $(V,W,e)$, it also satisfies $\Rees^e_{W^e}(V)\cong \Rees^e_{W}(V)$, so factorises the $\infty$-equivalence of Lemma \ref{quasiinvlemma}.
 
  \end{lemma}
 \begin{proof}
 We adapt the equivalence between exhaustively  filtered vector spaces and flat $\bG_m$-equivariant $\Q[\hbar]$-modules from for instance \cite[Lemma 2.1]{mhs2}.

 %
 The functor $\Rees^e$ has a left adjoint, which  sends a $\bG_m$-equivariant  $\Q[\hbar^2]$-module $M$ in complexes to  
the complex  $M/(\hbar^2-1) $ equipped with involution $-1 \in \bG_m$ and  filtration 
 \begin{align*}
 W_n( M/(\hbar^2-1))&:= (\bigoplus_{i \le n}\cW_iM)/(\hbar^2-1) (\bigoplus_{i \le n-2}\cW_iM)\\
 &\cong \cW_nM \oplus \cW_{n-1}M;
 \end{align*}
flatness of $M$ ensures that the map $W_n( M/(\hbar^2-1))\to M/(\hbar^2-1)$ is indeed an inclusion. 
 
 Evaluation at $\hbar=1$  gives an morphism  $\Rees^e_W(V)/(\hbar^2-1) \to V   $, the co-unit of the adjunction. When $V$ is involutively filtered, this map is a filtered isomorphism. Even when $V$ is not involutively filtered, it follows  easily  that $W_n (\Rees^e_W(V)/(\hbar^2-1))\cong W_n^eV$. Thus the involutive  Rees functor factorises the $\infty$-equivalence of Lemma \ref{quasiinvlemma}, and the co-unit of the adjunction is a natural isomorphism when restricted to involutively filtered objects. 

Meanwhile, for any $\bG_m$-equivariant  $\Q[\hbar^2]$-module $M$  we have 
\[
\cW_n\Rees^e_W(M/(\hbar^2-1)) \cong  (\cW_n \oplus \cW_{n-1}M)^{e=(-1)^n}= \cW_nM,
 \]
 so the unit of the adjunction is  a natural isomorphism.
  \end{proof}

 \begin{remark} 
  Although we are assuming that filtrations are exhaustive, beware that we are not assuming they are Hausdorff. For instance, the $\bG_m$-equivariant  $\Q[\hbar^2]$-module $V[\hbar^2,\hbar^{-2}]$ corresponds under Lemma \ref{reeslemma} to the vector space $V$ with  trivial involution and filtration $W_iV=V$ for all $i$. 
  
  However, in applications we only ever work with filtrations $W$ on algebras satisfying the stronger condition $W_{-1}V=0$, and our filtrations on operads will always be complete.
  \end{remark}

 For the remainder of this subsection, we fix  a $\bG_m$-equivariant operad $\cP$ in chain complexes over $\Q$.
 
\begin{definition}\label{filcatdef} 
Define the category of quasi-involutively filtered  $(\cP,W,e)$-algebras  to consist of  quasi-involutively filtered chain complexes $(V,W_iV,e)$ of $\Q$-vector spaces equipped with a $\cP$-algebra structure which is compatible with the involution $-1 \in \bG_m$ on $\cP$ and compatible with the filtration in the sense that the structure maps $\cP(n) \ten V^{\ten n} \to V$ restrict to maps
\[
 \cW_i\cP(n)\ten W_r(V^{\ten n}) \to W_{r+i}V
\]
for all $r$ and $i$, where $W_r(V^{\ten n}):= \sum_{r_1+\ldots +r_n=r} (W_{r_1}V)\ten \ldots \ten (W_{r_n}V)$.
\end{definition}

The involutive Rees functor  is clearly lax monoidal, i.e.   $ \Rees^e_W(V)\ten_{\Q[\hbar^2]}\Rees^e_W(V) \to \Rees^e_W(U\ten_{\Q}V) $, as is its left adjoint $M \mapsto M/(\hbar^2-1)$,  so Lemma \ref{reeslemma} has the following immediate consequence:
\begin{lemma}\label{reeslemmaP} 
The functor of Definition \ref{reesdef} gives an equivalence of $\infty$-categories from the category of 
  quasi-involutively filtered  $(\cP,W,e)$-algebras
  localised at filtered quasi-isomorphisms to the $\infty$-category of  $\bG_m$-equivariant  $\cP$-algebras in chain complexes of flat  $\Q[\hbar^2]$-modules,   localised at quasi-isomorphisms. 
  \end{lemma}



Note that there is a projective model structure on $\bG_m$-equivariant  $\cP$-algebras in chain complexes, for which fibrations are surjections and weak equivalences are quasi-isomorphisms. Existence of this cofibrantly generated model structure follows from  \cite[Theorem 11.3.2]{Hirschhorn} applied to the forgetful functor to $\bG_m$-equivariant chain complexes of $\Q$-vector spaces.

\begin{definition}\label{DDerdef}
Given 
a $\bG_m$-equivariant  $\cP$-algebra $A$, we write $\DDer_{\cP,\bG_m}(A,M)$ for the complex of $\bG_m$-equivariant $\cP$-derivations from $A$ to $M$. We also write $\oR\DDer_{\cP,\bG_m}(A,M) $ for the complex of derived derivations, given as in \cite{Q} by $\DDer_{\cP,\bG_m}(\tilde{A},M) $ for any cofibrant replacement $\tilde{A}$ of $A$ in the projective model structure. 

Here, $M$ is a Beck $A$-module, meaning that $A \oplus M$ is a  $\bG_m$-equivariant  $\cP$-algebra for which the projection map $A \oplus M \to A$ and the addition map $(A\oplus M)\by_A(A \oplus M) \to A \oplus M$ are both $\cP$-algebra homomorphisms. Explicitly, 
\[
 \DDer_{\cP,\bG_m}(A,M)_n:= \Hom(A, A \oplus \cone(M)_{[n+1]} )\by_{\Hom(A, A)}\{\id\},
\]
where $\Hom$s on the right-hand side are in the category of $\cP$-algebras in $\bG_m$-equivariant chain complexes.
The differential $\DDer_{\cP,\bG_m}(A,M)_n\to \DDer_{\cP,\bG_m}(A,M)_{n-1}$ is then induced by the obvious map $\cone(M)_{[n+1]} \to \cone(M)_{[n]} $.
\end{definition}

\begin{proposition}\label{defPprop} 
 Take quasi-involutively filtered  $(\cP,W,e)$-algebras $A,B$ such that the filtration on $B$ is complete (satisfied in particular if $W_{-1}B=0$), together with a morphism $f \co \gr^W A \to \gr^W B$ of the associated $\bG_m$-equivariant $\cP$-algebras.  
 
 If 
 \[
\H_i \oR\DDer_{\cP, \bG_m}(\gr^W A, \hbar^{2(n+1)}\gr^WB )\cong 0
 \]
 for all $n\ge 0$ and all $i \ge -1$, then  the  homotopy fibre over $f$ of the   associated graded functor  
 \[
 \gr^W\co \oR\map_{(\cP,W,e)}(A,B) \to \oR\map_{\cP,\bG_m}(\gr^WA,\gr^WB) 
 \]
 on mapping spaces in the respective $\infty$-categories is contractible.  
\end{proposition}
\begin{proof}
By Lemma \ref{reeslemmaP},  $\oR\map_{(\cP,W,e)}(A,B) \simeq \oR\map_{\Q[\hbar^2]\ten \cP,\bG_m}(\Rees^e_WA,\Rees^e_WB)$. Since the filtration on $B$ is complete, we have $\Rees^e_WB\cong \Lim_n (\Rees^e_WB)/\hbar^{2n}$, the limit being taken in the $\bG_m$-equivariant category (i.e. separately in each weight). Since the maps $(\Rees^e_WB)/\hbar^{2n+2} \to (\Rees^e_WB)/\hbar^{2n} $ are all surjective, they are fibrations in the projective model structure, and thus the limit is a homotopy limit, so
\[
 \oR\map_{\cP[\hbar^2],\bG_m}(\Rees^e_WA,\Rees^e_WB) \simeq \ho\Lim_n \oR\map_{\cP[\hbar^2],\bG_m}(\Rees^e_WA,(\Rees^e_WB)/\hbar^{2n}),
\]
where we write $\cP[\hbar^2]$ for the operad $\Q[\hbar^2]\ten \cP$

Since the natural map $ (\Rees^e_WA)/\hbar^2 \to \gr^WA$ is a quasi-isomorphism, and similarly for $B$, we have
\[
 \oR\map_{\cP[\hbar^2],\bG_m}(\Rees^e_WA,(\Rees^e_WB)/\hbar^{2})\simeq \oR\map_{\cP,\bG_m}(\gr^WA,\gr^WB),
\]
so it suffices to show that the maps
\[
 \oR\map_{\cP[\hbar^2],\bG_m}(\Rees^e_WA,(\Rees^e_WB)/\hbar^{2n+2})\to  \oR\map_{\cP[\hbar^2],\bG_m}(\Rees^e_WA,(\Rees^e_WB)/\hbar^{2n})
\]
are all equivalences. 

We now invoke a standard obstruction theory argument. Since $\hbar^{2n}\Q \subset  \Q[\hbar^2]/\hbar^{2n+2}$ is an ideal, 
there is an obvious CDGA structure on the chain complex $C':= \cone(\hbar^{2n}\Q \to  \Q[\hbar^2]/\hbar^{2n+2})$,  with the  quotient map   $C' \to \Q[\hbar^2]/\hbar^{2n} $ being a quasi-isomorphism. 
Moreover, since $(\hbar^{2n})\cdot (\hbar^2)=0$ in  $ \Q[\hbar^2]/\hbar^{2n+2}$, the natural surjection $ C' \to \cone(\hbar^{2n}\Q \xra{0}  \Q)$ is also a CDGA map, with 
\[
 \Q[\hbar^2]/\hbar^{2n+2}\cong C'\by_{\cone(\hbar^{2n}\Q \xra{0}  \Q)}\Q.
\]

Tensoring with $\Rees^e_W(B)$ over $\Q[\hbar^2]$, this  gives rise to a quasi-isomorphism 
\[
C:=\cone(\hbar^{2n}\co  \Rees^e_W(B)/\hbar^2 \to   \Rees^e_W(B)/\hbar^{2n+2}) \xra{\alpha} \Rees^e_W(B)/\hbar^{2n} 
\]
of $\bG_m$-equivariant $\cP[\hbar^2]$-algebras, together with an isomorphism
\[
 \Rees^e_W(B)/\hbar^{2n+2} \cong C\by_{ \hbar^{2n}(\Rees^e_W(B)/\hbar^2)_{[-1]} \oplus \Rees^e_W(B)/\hbar^2 }\Rees^e_W(B)/\hbar^2.
\]

That fibre product is a homotopy fibre product because $\alpha$ is surjective, so we have a homotopy pullback square
\[
\begin{CD}
 \oR\map_{\cP[\hbar^2],\bG_m}(\Rees^e_W(A),\Rees^e_W(B)/\hbar^{2n+2}) @>>> \oR\map_{\cP,\bG_m}(\gr^WA,\gr^WB)\\
@VVV @VVV \\ 
 \oR\map_{\cP[\hbar^2],\bG_m}(\Rees^e_W(A),\Rees^e_W(B)/\hbar^{2n})@>>> \oR\map_{\cP,\bG_m}(\gr^WA,\gr^WB\oplus\hbar^{2n} \gr^WB_{[-1]}).
\end{CD}
\]
Taking homotopy fibres over $f \in \oR\map_{\cP,\bG_m}(\gr^WA,\gr^WB)$ gives a homotopy fibre sequence
\begin{align*}
 &\oR\map_{\cP[\hbar^2],\bG_m}(\Rees^e_W(A),\Rees^e_W(B)/\hbar^{2n+2})_f\\
 &\quad \to \oR\map_{\cP[\hbar^2],\bG_m}(\Rees^e_W(A),\Rees^e_W(B)/\hbar^{2n})_f\\
 &\quad \quad \phantom{\to }\, \to \oR\map_{\cP,\bG_m}(\gr^WA ,\gr^WB\oplus\hbar^{2n} \gr^WB_{[-1]})_f,
\end{align*}
but 
\[
\pi_j \oR\map_{\cP,\bG_m}(\gr^WA ,\gr^WB\oplus\hbar^{2n} \gr^WB_{[-1]})_f \cong \H_{j-1}\oR\DDer_{\cP, \bG_m}(\gr^W A, \hbar^{2n}\gr^WB )
\]
which is $0$ by hypothesis for all $j>0$. Thus  the base of the fibration is contractible, which gives the desired equivalence.
 \end{proof}
 
 \begin{corollary}\label{defPcora}
  If $A$ is a quasi-involutively filtered  $(\cP,W,e)$-algebra such that the filtration  is complete and
 $
\H_i \oR\DDer_{\cP, \bG_m}(\gr^W A, \hbar^{2n}\gr^WA)\cong 0
 $
 for all $n>0$ and all $i \ge -1$, then $A$ is quasi-isomorphic to $\gr^WA$ as a $(\cP,W,e)$-algebra.
 \end{corollary}
\begin{proof}
Proposition \ref{defPprop} implies that the $\infty$-category of  quasi-involutively filtered  $(\cP,W,e)$-algebras $B$ with fixed quasi-isomorphism $\gr^WB\simeq \gr^WA$  is contractible (by taking $f$ to be  the identity map on $\gr^WA$). Since $A$ and $\gr^WA$ both lie in this category, it follows that they are quasi-isomorphic.
\end{proof}

  \begin{corollary}\label{defPcor}
   The associated graded functor from the $\infty$-category of quasi-involutively filtered  $(\cP,W,e)$-algebras  to the category of $\bG_m$-equivariant $\cP$-algebras becomes an equivalence when restricted to objects $A$ satisfying the conditions:
   \begin{enumerate}
    \item\label{poswgt} $\gr^W_iA \simeq 0$ for all $i<0$, and
    \item\label{wgtsunder2}  $\oR\DDer_{\cP, \bG_m}(\gr^W A, M) \simeq 0$ for all Beck $\gr^WA$-modules $M$ with $\cW_iM \simeq 0$ for all $i<2$. 
   \end{enumerate}
   The quasi-inverse functor is given by sending a $\bG_m$-equivariant $\cP$-algebra to the underlying filtered algebra.
  \end{corollary}
  In particular, note that these conditions are satisfied if $\gr^WA$ has a cofibrant replacement generated in weights $\le 1$.
  \begin{proof}
   Given $A$ satisfying condition (\ref{wgtsunder2}), $B$ satisfying condition (\ref{poswgt}), and a   morphism $f \co \gr^W A \to \gr^W B$, observe that the Beck $A$-module $\hbar^{2(n+1)}\gr^WB$ is acyclic in weights below $2n+2$, so for all $n\ge 0$ we have
   \[
    \oR\DDer_{\cP, \bG_m}(\gr^W A, \hbar^{2(n+1)}\gr^WB )\simeq 0
   \]
   by hypothesis on $A$.
   Thus Proposition \ref{defPprop} ensures that the restricted functor is full and faithful. The functor sending a grading to its underlying increasing filtration is clearly a right inverse to $\gr^W$, and hence a quasi-inverse.
  \end{proof}

\begin{remark}\label{filteredoperadrmk1}
If we take a quasi-involutively filtered dg operad $\cQ$ with an involutive equivalence $\gr^W\cQ \simeq \cP$, then the conditions of 
 Proposition \ref{defPprop} give the same conclusion when $A$ and $B$ are quasi-involutively filtered  $(\cQ,W,e)$-algebras, with the same proof, because the associated graded pieces are the same. Since $\gr^W\cQ$-algebras are not then canonically $\cQ$-algebras, the analogous statement to Corollary \ref{defPcor} is just that the  associated graded functor from quasi-involutively filtered  $(\cQ,W,e)$-algebras  to the category of $\bG_m$-equivariant $\cP$-algebras is full and faithful when restricted to objects satisfying those conditions.  
 
 Essential surjectivity also  follows if we assume $\H_{<0}\cP=0$ and restrict to  $\cQ$-algebras $A$ satisfying $\H_{<0}(\gr^WA)=0$, because if the condition $\H_{-2} \oR\DDer_{\cP, \bG_m}(B, \hbar^{2(n+1)}B)=0$ holds for all $n\ge 0$,  then there must exist a quasi-involutively filtered  $(\cQ,W,e)$-algebra $A$ with $\gr^WA \simeq B$. This follows by the following deformation argument similar to \cite{KS,hinichDefsHtpyAlg}. By taking a cofibrant model for $B$ concentrated in non-negative degrees and lifting  generators, we have a quasi-free $\Rees^e_W(\cQ)$-algebra $\tilde{B}$ in graded vector spaces, on which we wish to define a closed derivation $\tilde{\delta}$ lifting $\delta$; the quasi-freeness and concentration conditions will ensure the dg algebra $(\tilde{B},\tilde{\delta})$ is cofibrant. We construct $\tilde{\delta}$ on generators inductively in $\hbar^2$, with the potential obstruction to lifting from $\tilde{B}/\hbar^{2n}$ to $\tilde{B}/\hbar^{2(n+1)}$ being the class $[\tilde{\delta}\circ \tilde{\delta}]$ in the space of derivations above. 
 \end{remark}


\subsection{Almost commutative Poisson algebras}\label{polyvectorsn} 

We now consider non-unital  $P_k$-algebras  (i.e. $(k-1)$-shifted Poisson algebras); these are non-unital CDGAs equipped with a Lie bracket of chain degree $k-1$ acting as a biderivation. They are governed by an operad $P_k$ which can be written as $\Com \circ (s^{1-k}\Lie)$ via a distributive law (cf. \cite[\S 8.6]{lodayvalletteoperads}), for the  operads $\Com,\Lie$ governing non-unital commutative algebras and Lie algebras, where the shift $s\cP$ of a dg operad $\cP$ is given by $(s\cP)(n):= \cP(n)_{[n-1]}$.

\begin{definition}\label{involutivePacdef}
Define the $\bG_m$-equivariant dg  operad $P_k^{ac}$ to be the dg operad $ \Com \circ s^{1-k}\hbar^{-1}\Lie$, where $(\hbar^j \cP)(i):= \hbar^{j(i-1)}\cP(i)$ for any operad $\cP$, and as in Definition \ref{hbardef}, $\hbar$ has degree $0$ and weight $1$ 
for the $\bG_m$-action. 



 Define a  quasi-involutive a.c. $P_k$-algebra over a CDGA $R$ to be an  $(R \ten P_k^{ac},W,e)$-algebra $A$ in 
quasi-involutively  filtered  chain complexes. Here, $R \ten \cP$ is the operad $(R\circ \cP)(n):=R\ten\cP(n) $ with operad structure 
coming from the distributive law (as in \cite[\S 8.6]{lodayvalletteoperads}) $\cP(n)\ten R^{\ten n} \to R \ten \cP(n)$ given by the multiplication on $R$. 
\end{definition}
Thus a  $\bG_m$-equivariant $P_k^{ac}$-algebra is a $P_k$-algebra equipped with a $\bG_m$-action for which multiplication has weight $0$ and the Lie bracket has weight $-1$. A quasi-involutive a.c. $P_k$-algebra over $R$ is a $P_k$-algebra in $R$-modules, equipped with an increasing 
filtration $W$ 
 satisfying $W_i\cdot W_j \subset W_{i+j}$ and $[W_i,W_j] \subset W_{i+j-1}$, together with an involution $*$ preserving the filtration, satisfying $(a\cdot b)^*=a^*\cdot b^*$ and $[a,b]^*=-[a^*,b^*]$, and acting as $(-1)^i$ on $\H_*\gr^W_i$. 
 
 Note that  the chain degree of a non-zero element of the operad $P_k^{ac}$ is always $(1-k)$ times its weight.


\begin{lemma}
 The Koszul dual of the $\bG_m$-equivariant dg operad $P_k^{ac}$ is given by $s^{k-1}\hbar P_k^{ac}$.
\end{lemma}
\begin{proof}
The proof of \cite[8.6.11]{lodayvalletteoperads} adapts to this shifted and graded setting (cf. \cite[Appendix C]{fresseOperadsGT2}) to give
\begin{align*}
(P_k^{ac})^! &= (\Com \circ s^{1-k}\hbar^{-1}\Lie)^!\\
 &=   (s^{1-k}\hbar^{-1}\Lie)^!\circ \Com^!\\
&\cong (s^{k-1}\hbar\Com)\circ \Lie\\
&= s^{k-1}\hbar(\Com \circ (s^{1-k}\hbar^{-1}\Lie))\\
&=s^{k-1}\hbar P_k^{ac}. \qedhere
\end{align*}
\end{proof}

We now fix a CDGA $R$ over $\Q$. There is a  model structure on the category of $\bG_m$-equivariant $R$-chain complexes (the projective model structure), in which fibrations are surjections.

Before we state the next proposition, recall that the cotangent complex $\oL\Omega^1$ for CDGAs is defined as the left-derived functor of K\"ahler differentials $\Omega^1$. 
For a morphism $f \co A \to B$ of CDGAs, that means  $\oL\Omega^1_{B/A}$ (defined up to $B$-linear quasi-isomorphism) is given by the $B$-module $\Omega^1_{\tilde{B}/A}\ten_{\tilde{B}}B$ for any factorisation $A \to \tilde{B} \to B$ of $f$ as a cofibration followed by a weak equivalence. 

For a morphism $A \to B$ of non-unital $R$-CDGAs, we then write $\oL\Omega^1_{B/A} := \oL\Omega^1_{(R \oplus B)/(R \oplus A)}$, using the equivalence between $B$-modules and modules for the unital ring $R \oplus B$. Note that this agrees with the definition above for  morphisms of unital CDGAs because then there are isomorphisms $R \oplus A \cong R \by A$ and $R \oplus B \cong R  \by B$ sending $(r,x)$ to $(r, r+x)$, with $\oL\Omega^1_{(R \by  B)/(R \by A)} \simeq \oL\Omega^1_{B/A}$ since cotangent complexes satisfy \'etale descent  and $ \oL\Omega^1_{R/R}\simeq 0$.

The following  lemma runs on  similar lines to \cite{melaniPoisson}:
\begin{lemma}\label{polyvectorlemma} 
 Given a  $\bG_m$-equivariant  $P_k^{ac}$-algebra  $A$ in $R$-chain complexes which is  cofibrant as a $\bG_m$-equivariant $R$-CDGA (unital or non-unital), together with a Beck $A$-module $M$, there is a model for $\oR\DDer_{P_k^{ac}\ten R, \bG_m}(A, M)$ which has a
 complete decreasing filtration $F^1 \supset F^2 \supset \ldots $, 
with associated graded complexes quasi-isomorphic to
 \[
  \HHom_A(\CoS_A^p((\Omega^1_{A/R})_{[-k]}),\hbar^{p-1}M)^{\bG_m}_{[-k]} \quad \text{ or }\quad \HHom_{R \oplus A}(\CoS_{R \oplus A}^p((\Omega^1_{A/R})_{[-k]}),\hbar^{p-1}M)^{\bG_m}_{[-k]}
 \]
in the unital and non-unital cases respectively,
 where $\Omega^1$ denotes the  complex of K\"ahler differentials of the underlying CDGA 
and  $\CoS_A^p(N) :=\Co\Symm^p_A(N)= (N^{\ten_A p})^{\Sigma_p}$ the cosymmetric powers.
\end{lemma}
\begin{proof}
Adapting the formulae of \cite[\S 11.2]{lodayvalletteoperads} to the $R$-linear setting,  the Koszul duality map  $\alpha \co  (s (P_k^{ac})^!)^{\vee} \to P_k^{ac}$ gives rise to a cobar-bar adjunction  $\Omega_{\alpha} \dashv \b_{\alpha} $ between $P_k^{ac}$-algebras in $R$-chain complexes and $s(P_k^{ac})^!$-coalgebras  in $R$-chain complexes. Moreover,  $\Omega_{\alpha}\b_{\alpha}A$ is a cofibrant resolution of $A$ \cite[Theorem 11.4.7]{lodayvalletteoperads},  so
 \[
  \oR\DDer_{P_k^{ac}\ten R, \bG_m}(A, M)\simeq \DDer_{P_k^{ac}\ten R, \bG_m}(\Omega_{\alpha}\b_{\alpha}A,M).
 \]

 Since $\Omega_{\alpha}C $ is freely generated by $C_{[1]}$, the latter complex can be rewritten as 
\[
(\HHom_R(\b_{\alpha} A,M)_{[-1]}^{\bG_m},  \delta +[\omega_A \circ \alpha,-]),
 \]
 where $\omega_A$ encodes the data defining the $P_k^{ac}$-algebra structure on $A$ and the
 bracket is  defined via the convolution product of \cite[6.4.4]{lodayvalletteoperads} (which is of weight $0$ for the $\bG_m$-action).
 
Substituting for  $\b_{\alpha}$, we can rewrite this complex as
\begin{align*}
& (\prod_n(\hbar^{n-1}P_k^{ac}(n)_{[k(n-1)]}\ten^{\Sigma_n}\HHom_R(A^{\ten n},M))^{\bG_m}, \delta +[\omega_A\circ \alpha,-]),\\
 &\cong  (\HHom_{R}( \Co\Symm^+(\hbar^{-1}\Co\Lie(A_{[-1]})_{[1-k]} )_{[k]}\hbar,M)^{\bG_m}, \delta +[\omega_A \circ \alpha,-]).
 \end{align*}
Since the differential on $ \b_{\alpha}A$ is non-increasing on cosymmetric powers, it follows that  this dual complex has a complete decreasing filtration $F$, with $F^p$ given by terms involving $\Co\Symm^i$ for $i\ge p$. 
 The associated graded pieces are then given by
\[
 \gr_F^p  \cong  (\HHom_{R}( \Co\Symm^p(\hbar^{-1}\Co\Lie(A_{[-1]})_{[1-k]} )_{[k]}\hbar,M)^{\bG_m}, \delta +[\omega_A \circ \alpha,-]).
\]

In order to proceed further, we use the Koszul duality map
$
 \beta \co (s\Lie)^{\vee} \to \Com,
$
which
as in  \cite[\S 11]{lodayvalletteoperads} 
 gives us a cofibrant resolution $A':=\Omega_{\Com}\b_{\Com} A$ of the non-unital CDGA underlying $A$, on generators $(\b_{\Com} A)_{[1]}$. Repeating the argument above for this instance of Koszul duality gives
 
 \begin{align*}
 & (\HHom_{R}( \Co\Symm^p(\hbar^{-1}\Co\Lie(A_{[-1]})_{[1-k]} )_{[k]}\hbar,M)^{\bG_m}, \delta +[\omega_A \circ \alpha,-])\\
 &\cong \HHom_{R \oplus A'}(\CoS_{R \oplus A'}^p((\hbar^{-1}\Omega^1_{A'/R})_{[-k]})\hbar,M)^{\bG_m}_{[-k]}, \\
&\simeq \HHom_{R \oplus A}(\CoS_{R \oplus A}^p((\hbar^{-1}\Omega^1_{A/R})_{[-k]})\hbar,M)^{\bG_m}_{[-k]},\\
\end{align*}
where the final line follows because  $A$ is  cofibrant and $A' \to A$ a quasi-isomorphism. When $A$ is unital,  the first factor of $ R \by A \cong R\oplus A$  acts trivially on $\Omega^1_{A/R}$ and on the $A$-module $M$, so we may replace $R \oplus A$ with $A$ throughout.
\end{proof}


\begin{proposition}\label{weightprop} 
 If $B$ is a  non-negatively weighted $\bG_m$-equivariant $P_k^{ac}$-algebra over a CDGA $R$   for which the map  $(\cW_1\oL\Omega^1_{B/\cW_0B})\ten_{(R \oplus \cW_0B)}^{\oL}(R \oplus  B) \to \oL\Omega^1_{B/\cW_0B}$ 
 is a quasi-isomorphism, then
\[
 \oR\DDer_{P_k^{ac}\ten R, \bG_m}(B, M) \simeq 0
 \]
 for all Beck $B$-modules $M$ with $\cW_iM \simeq 0$ for all $i<2$.
\end{proposition}
\begin{proof}
 Without loss of generality, we may assume that $B$ is cofibrant. For  the unital algebra $B':=R \oplus B$, with $R$ in weight $0$, we have  $\Omega^1_{B'}\cong \Omega^1_B$ and $\Omega^1_{\cW_0B'}\cong \Omega^1_{\cW_0B}$, hence  an exact triangle
\[
\Omega^1_{\cW_0B}\ten_{\cW_0B'}B' \to \Omega^1_{B} \to \Omega^1_{B/\cW_0B} \to\Omega^1_{\cW_0B}\ten_{\cW_0B'}B'_{[-1]},
\]
which by hypothesis simplifies to
\[
 \Omega^1_{\cW_0B}\ten_{\cW_0B'}B' \to \Omega^1_{B} \to (\cW_1\Omega^1_{B/\cW_0B})\ten_{\cW_0B'}B' \to\Omega^1_{\cW_0B}\ten_{\cW_0B'}B'_{[-1]}.
\]

Since $\Omega^1_{B}$ is quasi-isomorphic to a $B$-module (equivalently, $B'$-module) 
freely generated   by terms of  weights $0,1$,  we thus have that 
$\oR\HHom_{B'}(\CoS_{B'}^p((\Omega^1_{B/R})_{[-k]}),N)^{\bG_m}$ is acyclic  for any Beck module $N$ with weights $> p$, or more generally with $\cW_iN$ acyclic  for $i \le p$. Setting $N:=\hbar^{p-1}M$,  the statement now follows from the convergent spectral sequence 
\[
 \Ext^j_{B'}(\CoS_{B'}^p((\Omega^1_{B/R})_{[-k]}),\hbar^{p-1}M)^{\bG_m} \abuts \H_{k-j}\oR\DDer_{P_k^{ac}\ten R, \bG_m}(B, M)
\]
arising from the complete filtration of Lemma \ref{polyvectorlemma}.
\end{proof}

\begin{corollary}\label{fildefcor1} 
 The associated graded functor from the $\infty$-category of quasi-involutive a.c. $P_k$-algebras to $\bG_m$-equivariant $P_k^{ac}$-algebras becomes an equivalence when restricted to those objects $B$ with $W_{-1}B=0$ for which the map  
 \[
 (\gr^W_1\oL\Omega^1_{B/W_0B})\ten_{(R \oplus W_0B)}^{\oL}(R \oplus \gr^WB) \to \oL\Omega^1_{\gr^WB/W_0B} 
 \]
of commutative cotangent complexes is a quasi-isomorphism.

The quasi-inverse functor is given by sending a $\bG_m$-equivariant $P_k^{ac}$-algebra to the underlying filtered algebra.
  \end{corollary}
\begin{proof}
 Substitute Proposition \ref{weightprop} into Corollary \ref{defPcor}.
\end{proof}


\begin{remark}\label{noninvolutivermk}
 If we were to consider non-involutive deformations instead, then the analogue of Corollary \ref{fildefcor1} would not hold. The non-involutive analogue of Corollary \ref{defPcor} involves homology of  $\oR\DDer_{\cP, \bG_m}(\gr^W B, M)$ for $M$ concentrated in weights $\ge 1$, giving a first term which is seldom acyclic. This is similar to phenomena arising in \cite{DQvanish, DQnonneg,DQLag}, where the only obstruction to quantisation is first-order, and can be eliminated by restricting to involutive quantisations.
\end{remark}

\begin{remark}\label{filteredoperadrmk2} 
Following Remark \ref{filteredoperadrmk1},  
if we take an involutively filtered dg operad $(\cP,W)$ with an involutive equivalence $\gr^W\cP \simeq P_k^{ac}$, then the conditions of Corollary \ref{fildefcor1} also ensure (via the argument of \cite{KS,hinichDefsHtpyAlg}) that for $B$ a non-negatively weighted $\bG_m$-equivariant $P_k^{ac}$-algebra concentrated in non-negative homological degrees, the space of quasi-involutive a.c. derived $\cP$-algebras $(B',W)$  with $\gr^WB'\simeq B$ is contractible. 

In particular, when $k=1$ we can take $\cP$ to be the $BD_1$ operad, given by the PBW filtration on the associative operad
as in  \cite[\S 3.5.1]{CPTVV} (named by analogy with the Beilinson--Drinfeld ($BD$ or $BD_0$) algebras of \cite[\S 2.2]{CostelloGwilliamVol1}). 
The argument above then gives an essentially unique filtered associative dg algebra $(B',W)$ equipped with a filtered involution $(B')^{\op} \cong B'$ and an equivalence $\gr^WB'\simeq B$. When $B$ is an algebra of polyvectors, $B'$ will thus be given by the ring of differential operators $\sD(\omega^{\half})$ (in the sense of \cite[\S 2.2]{ginzburgLectDmods}) on a square root $\omega^{\half}$ of the dualising bundle whenever this exists, and $B'$ gives a ring of twisted differential operators generalising $\sD(\omega^{\half})$ even when the dualising complex is not 
a line bundle, or has no square root. 
\end{remark}

\section{Quantisations on derived Deligne--Mumford stacks} 


\subsection{Hochschild complexes}\label{HHsn}


\begin{definition}\label{HHdef0}
 For a  cofibrant CDGA $A$ over $R$, we define the filtered  chain complex
\[
 (\CCC_{R,\oplus}(A)_{\bt}, \tau^{\HH})
\]
 to be the direct sum  total complex (written as a chain complex) of the double complex $\uline{\CCC}^{\bt}_R(A)$ given by  
\[
 \uline{\CCC}^n_R(A)= \HHom_R(A^{\ten_R n}, A),
\]
with Hochschild differential $b \co \uline{\CCC}^{n-1} \to \uline{\CCC}^{n}$ given by
\begin{align*}
 (b f)(a_1, \ldots , a_n) = &a_1 f(a_2, \ldots, a_n)\\
 &+ \sum_{i=1}^{n-1}(-1)^i f(a_1, \ldots, a_{i-1}, a_ia_{i+1}, a_{i+2}, \ldots, a_n)\\
&+ (-1)^n f(a_1, \ldots, a_{n-1})a_n.
\end{align*} 

The filtration $\tau^{\HH}$   
on $\uline{\CCC}^{\bt}_R(A)$ and $\CCC_{R,\oplus}(A)_{\bt}$ is given by good truncation in the Hochschild direction, so $\tau^{\HH}_p \CCC_{R,\oplus}(A)_{\bt} \subset \CCC_{R,\oplus}(A)_{\bt}$
is the subspace
\[
 \prod_{i=0}^{p-1} \HHom( A^{\ten_R i},A)_{[i]} \by \ker(b \co \HHom( A^{\ten_R p},A) \to \HHom( A^{\ten_R (p+1)},A) )_{[p]}.
\]
\end{definition}

\begin{remarks}
 Beware that $\CCC_{R,\oplus}(A)_{\bt}$ is  usually a proper subcomplex of the cohomological Hochschild complex $\CCC_{R,\Pi }(A)_{\bt}$ of $A$, which is defined by taking the product, rather than sum, total complex. Since $\CCC_{R,\oplus}(A)_{\bt}=\bigcup_p  \tau^{\HH}_p\CCC_{R,\Pi }(A)_{\bt}$, our constructions will be consistent with those of \cite{DQnonneg}.
 
There is a 
filtration $\gamma$ from \cite[Definition \ref{DQLag-HHdef0}]{DQLag}
which is quasi-isomorphic to $\tau^{\HH}$ under the cofibrancy hypothesis of Definition \ref{HHdef0}, but which also  behaves well whenever $A$ is cofibrant as an $R$-module (rather than as an $R$-CDGA) and gives an involutive, rather than quasi-involutive, filtration analogous to Lemma \ref{involutiveHH}. For simplicity, we will just use $\tau^{\HH}$. 
\end{remarks}

Recall that a brace algebra $B$ is a chain  complex  equipped with a cup product in the form of
a chain  map
\[
 B\ten B \xra{\smile} B,
\]
 and braces in the form of  maps
\[
 \{-\}\{-,\ldots,-\}_r \co B \ten B^{\ten r}\to B_{[r]}
\]
satisfying the conditions of \cite[\S 3.2]{voronovHtpyGerstenhaber}  (where brace algebras are called homotopy $G$-algebras) with respect to the  differential. The commutator of the brace $\{-\}\{-\}_1$ is a Lie bracket, so for any brace algebra $B$, there is a natural DGLA structure on $B_{[-1]}$.

There is a natural brace algebra structure on $\CCC_{R,\oplus}(A)$ by \cite[Theorem 3.1]{voronovHtpyGerstenhaber}.

The following are taken from \cite[\S \ref{DQLag-bracesn}]{DQLag}:
\begin{definition}\label{braceopdef}
 Given a brace algebra $B$, define the opposite brace algebra $B^{\op}$ to have the same elements as $B$, but multiplication $b^{\op}\smile c^{\op} := (-1)^{\deg b\deg c} (c\smile b)^{\op}$ and brace operations
given by the multiplication $(\b B^{\op}) \ten (\b B^{\op})\to \b B^{\op}$ on the bar construction induced by the isomorphism $(\b B^{\op})\cong (\b B)^{\op}$, the opposite coalgebra. Explicitly,
\[
 \{b^{\op}\}\{c_1^{\op}, \ldots, c_m^{\op}\}:=  \pm\{b\}\{c_m, \ldots, c_1\}^{\op},
\]
where $\pm= (-1)^{m(m+1)/2 + (\deg b-m)(\sum_i \deg c_i -m) +  \sum_{i<j}\deg c_i\deg c_j}$.
\end{definition}
Observe that when a filtered brace algebra $B$ is almost commutative, then so is $B^{\op}$.

\begin{definition}\label{acbracedef}
 Define a filtration $\tau$ on the brace operad $\Br$ of \cite{voronovHtpyGerstenhaber} by good truncation aritywise, so $(\tau_{-p}\Br)(n):= \tau_{\ge p}
 \Br(n)$  

We refer to (brace, $\tau$)-algebras in filtered complexes as almost commutative brace algebras.
We define a  quasi-involutive a.c. brace algebra to be an almost commutative brace algebra $(B,F)$ equipped with an involution $(B,F)\cong (B^{\op},F) $ of  (brace, $\tau$)-algebras which acts on $\H_*(\gr^F_iB)$ as multiplication by $(-1)^i$. 
\end{definition}

\begin{lemma}\label{acHHlemma}
 The filtered Hochschild complex $(\CCC_{R,\oplus}(A)_{\bt}, \tau^{\HH})$ naturally has the structure of an almost commutative brace algebra. 
\end{lemma}
\begin{proof}
The double complex  $\uline{\CCC}^{\bt}_R(A)$ is a brace algebra in the cochain direction, as a consequence of \cite[Theorem 3.1]{voronovHtpyGerstenhaber}. Compatibility of tensor products with good truncation then makes $(\uline{\CCC}^{\bt}_R(A), \tau^{\HH})$ a $(\Br,\tau)$-algebra in double complexes, and passing to total complexes gives the required result. 
\end{proof}

\begin{definition}\label{poldef} 
Define $\Pol(A/R,0)$ to be the $\bG_m$-equivariant $P_2^{ac}$-algebra $\bigoplus_{p \ge 0}  \HHom_A(\Omega^p_{A/R} ,A)_{[p]}$ with the   obvious graded-commutative multiplication and the Schouten--Nijenhuis Lie bracket, given by interpreting $\HHom_A(\Omega^p_{A/R} ,A) $ as the complex of antisymmetric $p$-derivations $A^p \to A$.
 \end{definition}

\begin{lemma}\label{grtaulemmaHH}
For a cofibrant $R$-CDGA $A$, the HKR isomorphism gives a quasi-isomorphism
 $\mathrm{HKR} \co  \gr^{\tau^{\HH}}\CCC_{R,\oplus}(A)_{\bt} \to \Pol(A/R,0)$ 
 of  $\bG_m$-equivariant $\gr^{\tau}\Br$-algebras, where the $\gr^{\tau}\Br$-algebra structure on  $\Pol(A/R,0)$ comes from the quasi-isomorphism $\gr^{\tau}\Br \to H_*\Br \cong P_2 $ of dg operads from \cite{McClureSmithDeligneConj}. 
 \end{lemma}
 \begin{proof}
 By turning the structural differential $\delta$ on $A$ off and on again, the HKR isomorphism gives us a quasi-isomorphism $(\HH^p_R(A_{\#}),\delta) \simeq \HHom_A(\Omega^p_A,A)$, since $A$ is cofibrant\footnote{In fact, it suffices for  $A$ to be quasi-smooth in the sense of \cite{Kon} (not to be confused with the clashing sense of \cite{toenseattle}) or even ind-quasi-smooth, i.e. for $A_0$ to be ind-smooth and the graded algebra $A_{\#}$ to be freely, or even projectively, generated over it. The cofibrant hypothesis on $A$ can be accordingly relaxed throughout this paper.}.
Thus 
$\gr_p^{\tau^{\HH}} \uline{\CCC}^{\bt}_R(A)$ maps  quasi-isomorphically to $\HHom_A(\Omega^p_A,A)^{[-p]}$.

Now, the action of $\gr^{\tau}\Br$ on $\HH^*_R(A_{\#})$ factors through its quotient $\H_*\Br \cong P_2$. This operad is generated by the commutative product and Lie bracket, and a simple check shows that the cup product and Gerstenhaber bracket on Hochschild cohomology act in the prescribed fashion under the HKR isomorphism. 
 \end{proof}

\begin{lemma}\label{involutiveHH} 
Given a cofibrant  CDGA $A$ over $R$, there is an anti-involution 
\[
 -i \co \CCC_{R,\oplus}(A)^{\op} \to \CCC_{R,\oplus}(A)
\]
given by
\[
- i(f)(a_1, \ldots, a_m) = (-1)^{\sum_{i<j}  \deg a_i \deg a_j} (-1)^{m(m+1)/2}f(a_m, \ldots , a_1),
\]
 making $(\CCC_{R,\oplus}(A),\tau^{\HH})$ into  a quasi-involutive 
 a.c. brace algebra.
\end{lemma}
\begin{proof}
The (brace, $\tau$)-algebra structure is given by Lemma \ref{acHHlemma}.
The involution is given in  \cite[\S 2.1]{braunInvolutive}, and comes from the antipode on the cofree tensor coalgebra $\bigoplus (A_{[-1]})^{\ten n}$ regarded as the universal enveloping coalgebra of a cofree graded Lie coalgebra. Quasi-involutivity is then a simple calculation as in \cite[Lemma \ref{DQnonneg-involutiveHH}]{DQnonneg}, and compatibility with the brace structure follows as in  \cite[Remark \ref{DQnonneg-oddcoeffsrmk}]{DQnonneg}.
\end{proof}

\subsection{Polydifferential operators}

There is a variant of the Hochschild complex defined using polydifferential operators instead of $R$-linear maps. In the algebraic setting it is quasi-isomorphic to the Hochschild complex, so leads to the same theory, but in smooth and analytic settings it gives a much better behaved object.


The following adapts \cite{CarchediRoytenbergHomological} as in \cite{DStein,DQDG}:
\begin{definition}\label{CinftyEFCdef}
Define a  $\C^{\infty}$-DGA (over a base $R=\R$), resp. an EFC-DGA  (over a base $R=K$, a complete valued field), to be an $R$-CDGA for which:
\begin{itemize}
 \item $A_0$ is equipped with a $\C^{\infty}$-ring structure, resp. entire functional calculus, enhancing its commutative $R$-algebra structure, and
 \item the derivation $\delta \co A_0 \to A_{-1}$ is a $\C^{\infty}$-derivation, resp. EFC-derivation.
\end{itemize}
A morphism $A \to B$ of $\C^{\infty}$-DGAs, resp. EFC-DGAs, is then a morphism of $R$-CDGAs preserving the $\C^{\infty}$ (resp. EFC) structure.
\end{definition}
A $\C^{\infty}$ (resp. EFC) structure on $A_0$ is  a product-preserving set-valued  functor $\R^n \mapsto (A_0)^n$ (resp. $K^n \mapsto (A_0)^n$) on the category with objects $\{\R^n\}_{n \ge 0}$ (resp. $\{K^n\}_{n \ge 0} $) and morphisms consisting of $\C^{\infty}$ (resp. $K$-analytic) maps. Explicitly, the set $A_0$ is equipped, for every $\C^{\infty}$ (resp. analytic) function $f$ in $n$ variables  with an operation $\Phi_f \co (A_0)^n \to A_0$. The condition on $\delta \co A_0 \to A_{-1}$ then says that the map $(\id, \delta) \co A_0\to A_0 \by A_{-1}$ is a $\C^{\infty}$ (resp. EFC) morphism.

As in  \cite[Theorem 6.10]{CarchediRoytenbergHomological}, the corresponding categories of algebras carry cofibrantly generated model structures  obtained by adjunction from the model category of chain complexes, so weak equivalences are quasi-isomorphisms and fibrations are surjections.  For most applications (though not \S \ref{intBCsn}), it suffices to know that examples of cofibrant $\C^{\infty}$-DGAs and EFC-DGAs $A$ are given by CDGAs $A=A_{\ge 0}$ which are freely generated as graded-commutative $A_0$-algebras, with $A_0$ respectively a ring of smooth functions on $\R^n$ or a ring of analytic functions on $K^n$.\footnote{In the remainder of the paper, cofibrancy is only used to give the correct cotangent complexes, so the condition on $A$ can be relaxed by asking only that $A_0$ consist respectively of functions on a smooth manifold or 
holomorphic functions on  a Stein manifold, and  that $A=A_{\ge 0}$ be projectively generated as a graded-commutative $A_0$-algebra.} Morphisms are CDGA morphisms with additional restrictions in degree $0$ corresponding to $\C^{\infty}$ (resp. holomorphic) morphisms between manifolds. 

Both $\C^{\infty}$-DGAs and EFC-DGAs similarly have notions of derivations which are more restrictive than those of the underlying CDGAs, with resulting modules of differential forms; for rings of functions on manifolds, these correspond respectively to smooth and analytic forms on the manifold. This leads to a notion of differential operators, which have the expected form in terms of co-ordinates as in \cite[Definition \ref{DQDG-diffopsdef}]{DQDG}; for intrinsic definitions, see \cite[Remark \ref{DQDG-diffopsrmk}]{DQDG}.

\begin{definition}\label{Dpolydef0}
 For a  cofibrant $R$-CDGA (resp. $\C^{\infty}$-DGA, resp. EFC-DGA) $A$, we define the filtered  chain complex
\[
 (D^{\poly}_{\oplus}(A)_{\bt}, \tau^{\HH})
\]
of polydifferential operators as follows.
First let 
\[
 \uline{D}^{\poly,n}(A):= \cDiff(A^{\coprod n}, A), 
\]
the complex of $R$-linear algebraic (resp. $\C^{\infty}$, resp. EFC) differential operators from the $n$-fold coproduct  $A^{\coprod n} $ (i.e. $A^{\ten_R n}$ in the algebraic case) to the $A^{\coprod n} $-module  $A$. This carries a Hochschild differential $b \co \uline{D}^{\poly,n-1} \to \uline{D}^{\poly,n}$ given by the same formula as in Definition \ref{HHdef0}, leading to a  double complex $\uline{D}^{\poly,\bt}(A)$. We then define  $D^{\poly}_{\oplus}(A)_{\bt}$  to be the direct sum  total complex of  $\uline{D}^{\poly,\bt}(A)$. 

The filtration $\tau^{\HH}$   
on $\uline{D}^{\poly,\bt}(A)$ and $\cD_{\oplus}^{\poly}(A)_{\bt}$ is given by good truncation in the Hochschild direction.
\end{definition}

The subcomplex $D^{\poly}_{\oplus}(A)_{\bt} \subset  \CCC_{R,\oplus}(A)$ is clearly closed under the brace operations, where $R$ is taken to be $\R$ (resp. $K$) in the $\C^{\infty}$ (resp. EFC) setting, so $D^{\poly}_{\oplus}(A)_{\bt}$ is a brace algebra over $R$.

\begin{definition}\label{poldefEFC} 
Given a $\C^{\infty}$-DGA $A$,  resp. EFC-DGA $A$ over $K$,  define $\Pol(A,0)$ 
to be the $\bG_m$-equivariant $P_2^{ac}$-algebra of $\C^{\infty}$, resp. EFC, polyvectors of $A$, defined by replacing the $\Omega^1_{A/R}$ in Definition \ref{poldef} with the complexes $\Omega^1_{A, \C^{\infty}}$ and  $\Omega^1_{A/K, \EFC}$ of $\C^{\infty}$ and EFC differentials, respectively.
\end{definition}

\begin{lemma}\label{grtaulemmaDpoly}
For any cofibrant $R$-CDGA (resp. $\C^{\infty}$-DGA, resp. EFC-DGA) $A$, we have a natural quasi-isomorphism
 $\mathrm{HKR} \co \gr^{\tau^{\HH}}D^{\poly}_{\oplus}(A)_{\bt} \to \Pol(A/R,0)$ 
 of  $\bG_m$-equivariant $\gr^{\tau}\Br$-algebras.
 \end{lemma}
 \begin{proof}
 Observe that for the order filtration $F$ on differential operators, we have $\gr^F_j \uline{D}^{\poly,n}(A)\cong \HHom_{A}(\Symm^j_A((\Omega^1_A)^{\oplus n}),A)$. Thus $\gr^F\uline{D}^{\poly,n}(A)$ is the $A$-linear dual of the free  $A$-algebra generated by $(\Omega^1_A)^{\oplus n}$, and the Hochschild differential on these copies of $\Omega^1_A$ is easily seen to act as the total differential of the simplicial nerve of $\Omega^1_A$. Thus the Eilenberg--Zilber theorem gives quasi-isomorphisms
 \[
  \gr^F_pD^{\poly}_{\oplus}(A)_{\bt} \to  \HHom_A(\Symm^p_A((\Omega^1_A)_{[-1]}),A) \cong \HHom_A(\Omega^p_A,A)_{[p]}. 
 \]
These are compatible with $\tau^{\HH}$, so the restriction to $D^{\poly}_{\oplus}(A)_{\bt}$ of the HKR map is  a quasi-isomorphism.
The rest of the proof follows as in Lemma \ref{grtaulemmaHH}.
 \end{proof}

\begin{lemma}\label{involutiveDpoly} 
Given a cofibrant $R$-CDGA (resp. $\C^{\infty}$-DGA, resp. EFC-DGA) $A$, there is an anti-involution 
\[
 -i \co D^{\poly}_{\oplus}(A)_{\bt}^{\op} \to D^{\poly}_{\oplus}(A)_{\bt}
\]
 making $(D^{\poly}_{\oplus}(A)_{\bt}^{\op}(A),\tau^{\HH})$ into  a quasi-involutive 
 a.c. brace algebra.
\end{lemma}
\begin{proof}
 Most of the properties follow immediately from Lemma \ref{involutiveHH}, since $D^{\poly}_{\oplus}(A)_{\bt} \subset  \CCC_{R,\oplus}(A)$ is a brace subalgebra. It only remains to show that the involution acts as $(-1)^j$ on $\H_*\gr^{\tau^{\HH}}_jD^{\poly}_{\oplus}(A)_{\bt}^{\op}$. By Lemma \ref{grtaulemmaDpoly}, these groups are $\H_{*+j}\HHom_A(\Omega^j_A),A)$, on which the induced involution acts as 
 $(-1)^j$ by \cite[Lemma \ref{DQnonneg-involutiveHH}]{DQnonneg}, giving
quasi-involutivity.
\end{proof}

\subsection{Involutions from the Grothendieck--Teichm\"uller group}\label{GTsn} 

In order to exploit the structure provided by the anti-involutions on brace algebras,  we need a formality isomorphism for the brace operad  which intertwines the functors on brace algebras and Gerstenhaber algebras sending algebras to their opposites. Drinfeld associators which are even are known to provide  equivalences for Lie bialgebras satisfying an analogous property \cite{EnriquezHalbout}, and we now deduce similar results for brace algebras. 

Heuristically, the idea is that the involution $\op_{\Br}$ of the brace operad which sends a brace algebra to its opposite corresponds under Dunn additivity to the involution $\id \ast \op_{E_1}$  of $E_2 \simeq E_1 \ast E_1$ which sends $E_1$-algebras in the second factor to their opposites. The group space of homotopy automorphisms of $E_2$ is homotopy equivalent to the orthogonal group $\mathrm{O}(2,\R)$ acting on discs in the obvious way, with both this involution and Drinfeld's involution corresponding to reflections, and thus being homotopic. Since the interaction of Dunn additivity with solutions of the Deligne conjecture is not fully set out, we will establish equivalence by a  less direct argument.

By \cite{drinfeldQuasiTriangular,barnatan} as interpreted in \cite{fresseOperadsGT2}, the Grothendieck--Teichm\"uller group $\GT(\Q)$, defined as certain automorphisms of the pro-unipotent completion of the free group $F_2$ on two generators, acts naturally on a model for the operad  $\CC_{\bt}(E_2,\Q)$ of chains of $E_2$ as an operad in cocommutative dg coalgebras (equivalently, as a   dg Hopf operad). $\GT$ is a pro-algebraic group with reductive quotient $\bG_m$.

\begin{definition}
 As in \cite[\S 4.1]{barnatan}, define $P \in \GT(\Q)$ to be the  automorphism sending each generator of $F_2$  to its inverse. This lies over $-1 \in \bG_m$ and by \cite[Proposition 4.1]{drinfeldQuasiTriangular} is the unique non-trivial object-preserving automorphism of the operad (in groupoids) of parenthesised braids.
\end{definition}

\begin{lemma}\label{GTinvlemma} 
 The brace operad $\Br$ equipped with its involution $\op_{\Br}$ is naturally quasi-isomorphic as a $C_2$-equivariant $\Q$-linear dg operad to the operad $\CC_{\bt}(E_2,\Q)$ of chains of the little discs operad equipped with its involution $P \in \GT$. 
\end{lemma}
\begin{proof}
 By \cite[Theorem 1.1]{McClureSmithDeligneConj}, the brace operad $\Br$ is naturally  $\Q$-linearly quasi-isomorphic to $\CC_{\bt}(E_2,\Q)$. As discussed in \cite[\S\S 3.2 and 4.5]{kaufmannSpinelessCactiDeligne}, there are models $E_2'$ for $E_2$ for which the involution $\op_{\Br}$ is induced by an involution $t$. In other words, there exist  a topological operad $E_2'$, a quasi-isomorphism $\CC_{\bt}(E_2',\Q) \to \Br$ intertwining $t$ and $\op_{\Br}$, and a zigzag of weak equivalences between $E_2'$ and $E_2$.
 
 Now, by \cite[Theorem 8.5]{horelProfiniteOperad} the group space $\Aut^h(E_2)$ of homotopy automorphisms of $E_2$ is weakly equivalent to $\mathrm{O}(2,\R) \simeq \{\pm 1\} \ltimes B\Z$, with the non-trivial class in $\pi_0\Aut^h(E_2)$ being $[P]$. We know that $t$ cannot be homotopic to the identity because $\op_{\Br}$ acts non-trivially on $\gr^{\tau}\Br \simeq P_2$. Moreover, 
 since the second cohomology group of $C_2$ with coefficients in the non-trivial representation on $\Z$ is zero, it follows that the homotopy fibre of the map $\oR\map_{\Top\Gp}(C_2, \Aut^h(E_2)) \to \Hom_{\Gp}(C_2, \pi_0\Aut^h(E_2))$ over the non-trivial homomorphism $C_2 \to \{\pm 1\}$ is connected (and in fact equivalent to $B\Z$, corresponding to the space of reflections inside $\mathrm{O}(2,\R)$).  Thus after identification via the quasi-isomorphism $\CC_{\bt}(E_2,\Q)\simeq \Br$ above,   $P$ and $t$ must induce homotopic involutions $C_2 \to \Aut^h(E_2)$.  
\end{proof}

\begin{definition}
Denote the pro-unipotent radical of the  pro-algebraic group  $\GT$ (defined over $\Q$) by $\GT^1$. Write $\Levi_{\GT}$ for the set of Levi decompositions $\GT \cong \bG_m \ltimes \GT^1$; equivalently, this is the set of sections of the natural map $\GT \to \bG_m$.
We then define $\Levi_{\GT}^P \subset \Levi_{\GT}$ to be the set of sections $w$ 
satisfying $w(-1)=P$. 
\end{definition}
Taking base change to arbitrary commutative $\Q$-algebras $A$, we can extend this definition  to give sets $\Levi_{\GT}(A), \Levi_{\GT}^P(A)$ of decompositions of the pro-algebraic group  $\GT \by \Spec A$ over $A$.
  
Every Levi decomposition $\GT \cong \GT^1 \rtimes \bG_m$ of $\GT$ automatically equips $\GT^1$ with a grading. Since the natural $\bG_m$-action  on the abelianisation of $\GT^1$ has weight $1$, each such grading splits the lower central series filtration on $\GT^1$, giving an isomorphism between $\GT^1$ and the graded Grothendieck--Teichm\"uller group $\mathrm{GRT}^1$. Thus $\Levi_{\GT}$ is naturally isomorphic to the space of Drinfeld $1$-associators.  The subset  $\Levi_{\GT}^P \subset\Levi_{\GT} $ then corresponds to  $1$-associators which are even in the sense of   \cite[\S 4.1]{barnatan}.
  
\begin{lemma}\label{leviGTlemma}
 The functor  $\Levi_{\GT}^P$ is an affine scheme over $\Q$ equipped with the structure of a trivial torsor for the subgroup scheme $(\GT^1)^P\subset \GT^1$ given by the centraliser  of  $P$.
\end{lemma}
\begin{proof}
We expand the argument from \cite[Remark \ref{DQnonneg-oddcoeffsrmk}]{DQnonneg}.
 By the general theory \cite{Levi} of pro-algebraic groups in characteristic $0$, the set $\Levi_{\GT}(\Q)$ is non-empty, and for all commutative $\Q$-algebras $A$, the group $\GT^1(A)$ acts transitively on $\Levi_{\GT}(A)$ via the adjoint action. Because the graded quotients of the lower central series of $\GT^1$ have non-zero weight for the adjoint $\bG_m$-action, the centralisers of this action are trivial and $\Levi_{\GT}(A)$ is a torsor for $\GT^1(A)$.

Now, choose any Levi decomposition $w_0 \in \Levi_{\GT}(\Q)$ and let $w_0(-1)=Pu$ for $u  \in \GT^1(\Q)$. 
Since $P$ and $w_0(-1)$ are both of order $2$, we have $u= \ad_P(u^{-1})$. Writing $u = \exp(v)$ for $v$ in the pro-nilpotent Lie algebra $\g\ft^1$, and setting $u^{\half}:= \exp(\frac{v}{2})$, we have  $u^{\half}=\ad_P(u^{-\half})$, giving $w:=\ad_{u^{\half}}\circ w_0 \in \Levi_{\GT}^P(\Q)$. Thus $\Levi_{\GT}^P$ is non-empty, so $\Levi_{\GT}^P \subset \Levi_{\GT}$ is a torsor for the subgroup $(\GT^1)^P \subset \GT^1$ fixing $P$ under the adjoint action.
\end{proof}


 \begin{proposition}\label{Brformalprop} 
Every $1$-associator $w \in \Levi_{\GT}$ induces a zigzag $\theta_w$ of 
filtered quasi-isomorphisms between $(\Br,\tau)$ and $(P_2,\tau)$.
If the $1$-associator is even (i.e. $w \in \Levi_{\GT}^P $), then the quasi-isomorphisms are $C_2$-equivariant, preserving the involutions given by $-1 \in \bG_m$ on $P_2^{ac}$  and  by Definition \ref{braceopdef} on $\Br$.

The quasi-isomorphism $\theta_w$ is compatible with the natural maps $s^{-1}\Lie \to \P_2$ and $s^{-1}\Lie \to \Br$ from the shifted Lie operad $s^{-1}\Lie$.
 \end{proposition}
\begin{proof}
As in Lemma \ref{GTinvlemma}, we have a zigzag of $C_2$-equivariant  quasi-isomorphisms between $(\Br,\op_{\Br})$ and $(\CC_{\bt}(B\mathrm{PaB}, \Q),P)$, the dg operad of chains on the operad of parenthesised braids, equipped with their respective involutions.

As explained succinctly in \cite{petersenGTformality}, formality of the operad $\CC_{\bt}(B\mathrm{PaB}, \Q)$ is a consequence of its $\GT$-action and the  observation that the   Grothendieck--Teichm\"uller group $\GT$ is a pro-unipotent extension of $\bG_m$.  
The pro-unipotent radical $\GT^1 =\ker(\GT \to \bG_m)$ acts trivially on homology $\H_*(B\mathrm{PaB},\Q)  \cong P_2$, inducing a $\bG_m$-action on $P_2$. 

Since the $\bG_m$-action  has weight $0$  on $\H_0\Br(2)$ and  weight $-1$ on $\H_1\Br(2)$, and since these generate the whole operad, it follows that the $\bG_m$-action on $\H_j\Br(k)$ has weight $-j$ for all $j,k$. Thus we have
a $\bG_m$-equivariant isomorphism
\[
 \H_*(B\mathrm{PaB}, \Q)\cong P_2^{ac} 
\]
with the $\bG_m$-equivariant operad $P_2^{ac}=  \Com \circ s^{-1}\hbar^{-1}\Lie$ of Definition \ref{involutivePacdef}. The element $P \in \GT(\Q)$ lies over $-1 \in \bG_m(\Q)$, so the involution  $P$ 
on $\Br$ induces the action  of $-1 \in \bG_m$ on  $P_2^{ac}$ under the isomorphism above.

Any Levi decomposition $w \co \bG_m \to \GT$ gives a $\bG_m$-action on $\CC_{\bt}(B\mathrm{PaB}, \Q)$, i.e. a weight decomposition. As in \cite{petersenGTformality}, the  weight decomposition associated to $w$ thus induces a zigzag of $\bG_m$-equivariant quasi-isomorphisms between $\CC_{\bt}(B\mathrm{PaB}, \Q)$ and $P_2^{ac}$. The involution $-1 \in \bG_m$ of the latter then necessarily corresponds to the involution $w(-1) \in \GT$ of the former. 
 
These quasi-isomorphisms of operads combine to give $\theta_w$, and the quasi-isomorphisms respect the involutions when $w(-1)=P$.
The  good truncation  filtration $\tau$ of Definition \ref{acbracedef} is defined similarly on any dg operad in place of $\Br$, and in particular there is a good truncation  on the quasi-isomorphic $\Q$-linear dg  operads $\CC_{\bt}(B\mathrm{PaB}, \Q)$ and  $\CC_{\bt}(E_2,\Q)$ of chains on the topological operad $E_2$.\footnote{The latter is presumably what is meant by the filtration given by the Postnikov tower in \cite[\S 3.5.1]{CPTVV}, since the Postnikov tower itself consists of quotients rather than subspaces, with kernels given by good truncation.} Any chain map automatically preserves the good truncation filtration,  so $\theta_w$ is a filtered quasi-isomorphism.

Finally, the natural morphism from the Lie operad to  $\Br$  is given in each arity by inclusion of the top weight term for the decreasing filtration, 
 i.e. 
 \[
 \{\Br(n)\}_n \la \{\tau_{\ge n-1}\Br(n)\}_n \xra{\sim} \{\H_{n-1}\Br(n)_{[1-n]}\}_n \cong \{\Lie(n)_{[1-n]}\}_n= \{s^{-1}\Lie(n)\}_n.
 \]
 The same construction applied to each dg operad in the zigzag ensures that our map $(P_2, \tau) \to (\Br,\tau)$ in the homotopy category of filtered dg operads (or $C_2$-equivariant filtered dg operads for $w$ even) respects the natural maps from the shifted Lie operad on each side.
 \end{proof}

\begin{definition}\label{pwinvdef}
 Given a  Levi decomposition  $w \in \Levi_{\GT}(\Q)$ of the pro-algebraic group $\GT$ over $\Q$, we denote by $p_w$ the $\infty$-functor  from almost commutative brace algebras to almost commutative $P_2$-algebras coming from  the map  $ \theta_w \co (P_2, \tau)\to (\Br,\tau) $. This preserves the underlying filtered $L_{\infty}$-algebras up to equivalence.
 
 When $w \in \Levi_{\GT}^P(\Q)$, this induces an $\infty$-functor from  quasi-involutive a.c. brace algebras to quasi-involutive a.c. $P_2$-algebras, which we also denote by $p_w$.
 \end{definition}

Explicitly, Proposition \ref{Brformalprop} gives us a filtered quasi-isomorphism $\tilde{\theta}_w \co \tilde{P}_2^{ac} \to (\Br, \tau)$ from any  $\bG_m$-equivariant cofibrant replacement $\tilde{P}_2^{ac}$ of $\tilde{P}_2$, and $\tilde{\theta}_w$ is involutive when $w$ is even. This automatically gives any  involutive a.c. brace algebra the structure of an involutive a.c.
$\tilde{P}_2^{ac}$-algebra, and we can then obtain $p_w$ by applying the derived left adjoint of the forgetful functor from  $\Rees^{-1 \in \bG_m}_W(P_2^{ac})$-algebras to $\Rees^{-1 \in \bG_m}_W(\tilde{P}_2^{ac})$-algebras induced by the morphism $\tilde{P}_2^{ac} \to P_2^{ac}$.

\subsection{Existence of deformation quantisations}


\subsubsection{Formality}

\begin{lemma}\label{Polcotlemma}
Take a cofibrant $R$-CDGA (resp. $\C^{\infty}$-DGA, resp. $K$-EFC-DGA) $A$  with  perfect cotangent complex $\Omega^1_{A/R}$ (resp. $\Omega^1_{A, \C^{\infty}}$, resp.  $\Omega^1_{A/K, \EFC}$). For the non-negatively weighted  $\bG_m$-equivariant $P_2^{ac}$-algebra $\Pol(A,0)$ 
of polyvectors (of Definition \ref{poldef}, resp. \ref{poldefEFC}), the map
 \[
 (\cW_1\oL\Omega^1_{\Pol(A,0) /A})\ten_{A}^{\oL}\Pol(A,0)  \to \oL\Omega^1_{\Pol(A,0)/A} 
 \]
of commutative cotangent complexes 
is a quasi-isomorphism.
\end{lemma}
\begin{proof}
Write $\Omega^p_{A,\boxempty}$ for $\Omega^p_{A/R} $ (resp.  $\Omega^p_{A, \C^{\infty}}$, resp.  $\Omega^p_{A/K, \EFC}$). Since $\Omega^1_{A,\boxempty}$ is perfect, we have a $\bG_m$-equivariant  quasi-isomorphism 
\begin{align*}
\bigoplus_{p\ge 0} \oL \Symm^p_A(\HHom_{A}(\Omega^1_{A,\boxempty},A)_{[1]})  &\to \bigoplus_{p\ge 0} \HHom_{A}(\Omega^p_{A, \boxempty},A)_{[p]}\\
 \oL\Symm_A(\cW_1\Pol(A,0) )  &\to \Pol(A,0)
\end{align*}
of CDGAs, from which the statement follows immediately. 
 \end{proof}
 
\begin{theorem}\label{fildefhochthm1} 
Given a cofibrant $R$-CDGA (resp. $\C^{\infty}$-DGA or  $K$-EFC-DGA)  $A$ with  perfect cotangent complex  $\Omega^1_{A/R}$ (resp. $\Omega^1_{A,\C^{\infty}}$ or $\Omega^1_{A/K,\EFC}$), the quasi-involutively  filtered DGLA underlying the  complex of polydifferential operators  
$
D^{\poly}_{\oplus}(A)_{[-1]}
$  
with the increasing filtration $\tau^{\HH}$ of Definition \ref{Dpolydef0} is filtered quasi-isomorphic to the graded DGLA  $\Pol(A,0)_{[-1]}$ from Definition \ref{poldef} (resp. Definition  \ref{poldefEFC}).   

This quasi-isomorphism depends only on  a choice of even $1$-associator $w \in \Levi_{\GT}^P$, and  is natural with respect to
all morphisms $(\cD^{\poly}_{\oplus}(A), \tau^{\HH}) \to (\cD^{\poly}_{\oplus}(A'), \tau^{\HH})$ in the $\infty$-category of quasi-involutive a.c. brace algebras.
%

When  $A$ is an $R$-CDGA with  perfect cotangent complex, the same statements hold for the Hochschild complex 
$
\CCC_{R,\oplus}(A)
$
in place of $D^{\poly}_{\oplus}(A)$.
\end{theorem}
\begin{proof}
Lemma \ref{involutiveDpoly} shows that $(D^{\poly}_{\oplus}(A), \tau^{\HH})$ is a quasi-involutive a.c. brace algebra. 
We have a quasi-isomorphism $\mathrm{HKR} \co \gr^{\tau^{\HH}}D^{\poly}_{\oplus}(A) \xra{\sim} \Pol(A,0)$ of graded $\gr^{\tau}\Br$-algebras  by Lemma \ref{grtaulemmaDpoly}. 

Lemmas \ref{involutiveHH} and  \ref{grtaulemmaHH} give the corresponding statements for 
$\CCC_{R,\oplus}(A)$ in the algebraic setting; since $D^{\poly}_{\oplus}(A) \to \CCC_{R,\oplus}(A)$ is thus a filtered quasi-isomorphism, it suffices to prove the statements for polydifferential operators in place of the Hochschild complex.

Applying the $\infty$-functor $p_w$ of Definition \ref{pwinvdef} for some even associator $w \in \Levi_{\GT}^P(\Q)$ (or even a point in the space $\Levi_{\GT}^P(R)$) gives a quasi-involutive a.c. $P_2$-algebra $p_wD^{\poly}_{ \oplus}(A)$ together with   a zigzag of quasi-isomorphisms of $\bG_m$-equivariant $P_2^{ac}$-algebras between its
associated graded algebra 
and $\Pol(A,0)$. 

By Lemma \ref{Polcotlemma},  
$\oL\Omega^1_{\Pol(A,0)/A}$ 
satisfies the  conditions of  Corollary \ref{fildefcor1}, giving an essentially unique equivalence
$
\alpha_{w,A}\co p_wD^{\poly}_{\oplus}(A) \simeq  p_w\gr^{\tau^{\HH}}D^{\poly}_{\oplus}(A)   
$
of involutive a.c. $P_2$-algebras, natural with respect to quasi-involutive a.c. brace algebra morphisms of $D^{\poly}_{\oplus}(A)$. Composition with the HKR quasi-isomorphism then gives us the required equivalence
\[
 \mathrm{HKR} \circ \alpha_{w,A}\co p_wD^{\poly}_{\oplus}(A) \simeq \Pol(A,0). \qedhere
\]
\end{proof}

\begin{remark}\label{cfKTrmk}
When applied to polynomial rings in the algebraic setting and to $\C^{\infty}(\R^n)$ in the smooth setting, the statement of Theorem \ref{fildefhochthm1}  recovers \cite[Theorem 4]{KontsevichOperads} and \cite[\S 3]{tamarkinOperadicKontsevichFormality}. For more general smooth  varieties it recovers   \cite[Theorem 1.1]{vdBerghGlobalDQ}. 
The preliminary steps are the same, but the arguments for eliminating the potential first-order deformation are very different, as we consider anti-involutive deformations while Tamarkin  and Kontsevich looked   at invariance under affine transformations, which do not exist even locally for our more general rings. 

Complete intersection singularities  have perfect cotangent complexes, so Theorem \ref{fildefhochthm1} also promotes Kontsevich's  quasi-isomorphism from \cite[Appendix, Proposition 1]{FronsdalKontsevich} to an $L_{\infty}$ quasi-isomorphism.
\end{remark}

\subsubsection{Affine quantisations}

\begin{definition}\label{mcPLdef}
 Given a   differential graded Lie algebra (DGLA) $L$ with homological grading, define the the Maurer--Cartan set by 
\[
\mc(L):= \{\omega \in  L_{-1}\ \,|\, \delta\omega + \half[\omega,\omega]=0 \in   L_{-2}\}.
\]

Following \cite{hinstack}, define the Maurer--Cartan space $\mmc(L)$ (a simplicial set) of a nilpotent  DGLA $L$ by
\[
 \mmc(L)_n:= \mc(L\ten_{\Q} \Omega(\Delta^n)_{\bt}),
\]
where 
\[
\Omega(\Delta^n)_{\bt}=\Q[t_0, t_1, \ldots, t_n,\delta t_0, \delta t_1, \ldots, \delta t_n ]/(\sum t_i -1, \sum \delta t_i)
\]
is the commutative dg algebra of de Rham polynomial forms on the $n$-simplex, with the $t_i$ of degree $0$ and $\delta t_i$ of chain degree $-1$, so $\Omega(\Delta^n)_{-m}= \Omega^m(\Delta^n)$, the space of $m$-forms.

Given an inverse system $L=\{L_{\alpha}\}_{\alpha}$ of nilpotent DGLAs, define
\[
 \mc(L):= \Lim_{\alpha} \mc(L_{\alpha}) \quad  \mmc(L):= \Lim_{\alpha} \mmc(L_{\alpha}).
\]
Note that  $\mc(L)= \mc(\Lim_{\alpha}L_{\alpha})$, but $\mmc(L)\ne \mmc(\Lim_{\alpha}L_{\alpha}) $. 
\end{definition}

\begin{definition}\label{hatpoldef}
 Given a cofibrant $R$-CDGA $A$ (resp. $\C^{\infty}$-DGA or EFC-DGA), as in \cite[Definition \ref{poisson-poldef}]{poisson} we complete the $P_2$-algebra $\Pol(A,0)$ of Definition \ref{poldef} (resp. Definition \ref{poldefEFC}) to give a $P_2$-algebra 
 \[
  \widehat{\Pol}(A,0):= \prod_{p \ge 0} \cW_p \Pol(A,0) 
 \]
 of polyvectors. This is $\prod_{p \ge 0}  \HHom_A(\Omega^p_{A/R} ,A)_{[p]}$ in the algebraic setting and corresponding expressions using $\C^{\infty}$ or EFC differentials in the other settings.
 
This is not graded, but does have a decreasing filtration 
\[
F^i\widehat{\Pol}(A,0):= \prod_{p \ge i}   \cW_p \Pol(A,0) 
\]

The space $\cP(A,0)$ of $0$-shifted Poisson structures on $A$ is then defined in \cite[Definition \ref{poisson-poissdef}]{poisson}  to be 
\[
\mmc(F^2 \widehat{\Pol}(A,n)_{[-1]});
\]
this is equivalent to the space of $P_1$-algebras with underlying CDGA  quasi-isomorphic to $A$.

\end{definition}

\begin{definition}\label{qpoldef}
 Given a cofibrant $R$-CDGA, $\C^{\infty}$-DGA or  $K$-EFC-DGA $A$, adapting \cite[Definition \ref{DQnonneg-qpoldef}]{DQnonneg}  as in \cite{DQDG,DStein}, we define the DGLA $ Q\widehat{\Pol}(A,0)_{[-1]}$ of quantised polyvectors by setting
 \[
  Q\widehat{\Pol}(A,0):= \prod_{p \ge 0} \tau^{\HH}_p D^{\poly}_{\oplus}(A)_{\bt}\hbar^{p-1};
 \]
 observe that because the Gerstenhaber bracket satisfies $[\tau^{\HH}_p,\tau^{\HH}_q] \subset \tau^{\HH}_{p+q-1}$, this is indeed a DGLA.
 
 Define a decreasing filtration $\tilde{F}$ on  $Q\widehat{\Pol}(A,0)$  by the subcomplexes 
\begin{align*}
 \tilde{F}^iQ\widehat{\Pol}(A,0)&:= \prod_{j \ge i} \tau^{\HH}_jD^{\poly}_{\oplus}(A)\hbar^{j-1}.
\end{align*}
This filtration is complete and Hausdorff,   with $[\tilde{F}^i,\tilde{F}^j] \subset \tilde{F}^{i+j-1}$.
In particular,  this makes $\tilde{F}^2Q\widehat{\Pol}(A,0)_{[-1]}$  into a pro-nilpotent filtered DGLA.

 The space $Q\cP(A,0)$ of  $0$-shifted quantisations of $A$ is then defined (adapting \cite[Definition \ref{DQnonneg-Qpoissdef}]{DQnonneg}) to be 
\[
\mmc(\tilde{F}^2 Q\widehat{\Pol}(A,n)_{[-1]}).
\]
The subspace $Q\cP(A,0)^{sd} \subset Q\cP(A,0)$ of self-dual quantisations then consists of fixed points for the involution $(-)^* $ given by $\Delta^*(\hbar):= i(\Delta)(-\hbar)$, for the involution  $i$ of Lemma \ref{involutiveHH}. 
 \end{definition}

 \begin{remark}\label{BD1rmk}
  Since for an $R$-CDGA $A$, the inclusion $ D^{\poly}_{\oplus}(A)_{\bt} \subset \CCC_{R,\oplus}(A)_{\bt}$ is a filtered quasi-isomorphism, replacing polyvectors with the Hochschild complex gives an equivalent construction in the algebraic setting. Similarly the  filtration $\gamma$ from \cite[Definition \ref{DQLag-HHdef0}]{DQLag} is  quasi-isomorphic to $\tau^{\HH}$, so the definitions are also equivalent to those of \cite{DQLag}.
  
  As in \cite[Remark \ref{DQLag-curvedrmk}]{DQLag}, $Q\cP(A,0)$  is thus equivalent to the space of curved $BD_1$-algebras (almost commutative $\hbar$-adic associative algebras) deforming $A$; 
when $A$ lies non-negative chain degrees no curvature is possible on objects,  but curvature still leads to additional  higher morphisms coming from   inner automorphisms. The objects in the $\C^{\infty}$ and EFC settings admit a similar interpretation, but with  additional constraints from the restriction to polydifferential operators.

As in \cite[Remark \ref{DQLag-curvedrmk}]{DQLag}, $Q\cP(A,0)^{sd}$ consists of curved $BD_1$-algebras $\tilde{A}$ deforming $A$ and equipped with an anti-involution $\tilde{A}^{\op} \cong \tilde{A}$ which is semilinear   under the transformation $\hbar \mapsto -\hbar$.
 \end{remark}

 \begin{corollary}\label{affquantcor} 
 Given a cofibrant $R$-CDGA (resp. $\C^{\infty}$-DGA or  $K$-EFC-DGA)  $A$ with  perfect cotangent complex  $\Omega^1_{A/R}$ (resp. $\Omega^1_{A,\C^{\infty}}$ or $\Omega^1_{A/K,\EFC}$),  the space $Q\cP(A,0)$ of $0$-shifted quantisations of $A$ is  equivalent to the Maurer--Cartan space
\begin{align*}
  &\mmc((A_{[-1]}\hbar \by \HHom_A(\Omega^1_{A, \boxempty} ,A)\hbar \by  \prod_{p \ge 2} \HHom_A(\Omega^p_{A,\boxempty} ,A)_{[p-1]}\hbar^{p-1})\brh)\\ 
  &\cong 
  \mmc( F^2\widehat{\Pol}(A,0)_{[-1]} \by\hbar F^1\widehat{\Pol}(A,0)_{[-1]} \by \hbar^2\widehat{\Pol}(A,0)_{[-1]}\brh)
 \end{align*}
 where the DGLA structure comes from the Schouten--Nijenhuis bracket, and $\Omega^p_{A,\boxempty}$ is $\Omega^p_{A/R}$ (resp. $\Omega^p_{A,\C^{\infty}}$, resp. $\Omega^p_{A/K,\EFC} $). 
 
The subspace  $Q\cP(A,0)^{sd} \subset Q\cP(A,0)$ of self-dual quantisations is then  equivalent to the Maurer--Cartan space
\begin{align*}
 &\mmc((A_{[-1]}\hbar \by \HHom_A(\Omega^1_{A, \boxempty} ,A)\hbar^2 \by  \prod_{p \ge 2} \HHom_A(\Omega^p_{A,\boxempty} ,A)_{[p-1]}\hbar^{p-1})\brhh)\\ 
&\cong \mmc(F^2\widehat{\Pol}(A,0)_{[-1]} \by \hbar^2\widehat{\Pol}(A,0)_{[-1]}\brhh).
 \end{align*}
 
 In particular, there
exist self-dual   associative
 quantisations for every Poisson structure on $A$. 
\end{corollary}
\begin{proof}
 The quasi-isomorphism of  Theorem \ref{fildefhochthm1}  gives us a quasi-isomorphism
 \begin{align*}
  \tilde{F}^2Q\widehat{\Pol}(A,0)_{[-1]} &\simeq \prod_{p \ge 2, i \le p} \cW_i \Pol(A,0)_{[-1]} \hbar^{p-1}= \prod_{i\ge 0}  \hbar^{\max(i-1,1)}\cW_i \Pol(A,0)_{[-1]}\brh 
   \end{align*}
of pro-nilpotent filtered DGLAs, and hence an equivalence of the respective Maurer--Cartan spaces. We then use the  isomorphism 
\begin{align*}
  &F^2\widehat{\Pol}(A,0) \by\hbar F^1\widehat{\Pol}(A,0) \by \hbar^2\widehat{\Pol}(A,0)\brh 
  \xra{\simeq}    \prod_i  \hbar^{\max(i-1,1)}\cW_i \Pol(A,0)\brh
\end{align*}
 given by multiplying 
$\cW_i\Pol(A,0)$ by $\hbar^{i-1}$; this gives a DGLA morphism because the Schouten--Nijenhuis bracket satisfies $[\cW_i,\cW_j]\subset \cW_{i+j-1}$.

Now, the involution $*$ acts on $\H_*(\gr^{\tau^{\HH}}_i\CCC_R(A)) \hbar^j$ and  on $ \cW_i \Pol(A,0) \hbar^j$ as $(-1)^{i+j-1}$. Under the isomorphism above given by  division by $\hbar^{i-1}$, this corresponds to the involution $\hbar \mapsto -\hbar$ of $F^2\widehat{\Pol}(A,0) \by\hbar F^1\widehat{\Pol}(A,0) \by \hbar^2\widehat{\Pol}(A,0)\brh$, so the fixed points are $F^2\widehat{\Pol}(A,0)  \by \hbar^2\widehat{\Pol}(A,0)\brhh $.

Existence of quantisations then follows from the inclusion  of the first term in 
$F^2\widehat{\Pol}(A,0) \by \hbar^2\widehat{\Pol}(\sO,0)\brhh$,
 giving a morphism $\cP(A,0) \to Q\cP(A,0)^{sd} \subset Q\cP(A,0)$.
 
 \end{proof}
 
\begin{remark}\label{explicitquantrmk}
 In principle, the proof of  Corollary \ref{affquantcor} could be used to construct deformation quantisations. It would be similar to explicit application of  \cite{tamarkinOperadicKontsevichFormality}  to smooth varieties, for which the author is unaware of any examples.  After choosing an even associator, the first step would be to describe the resulting strong homotopy $P_2$-algebra operations on Hochschild cochains as combinations of brace operations. One would then have to solve iteratively for the essentially unique involutive filtered equivalence  between that s.h. $P_2$-algebra and the algebra of polyvectors, with the resulting $L_{\infty}$ equivalence generating quantisations from Poisson structures. 
 \end{remark}

 \subsubsection{\'Etale functoriality revisited}\label{intBCsn} 

Functoriality for Hochschild complexes and polyvectors is subtle, but exists with respect to homotopy \'etale morphisms as in \cite[\S \ref{DQnonneg-Artindiagsn}]{DQnonneg} and  \cite[\S \ref{poisson-descentsn}]{poisson}. We now revisit and generalise the results there.

\begin{definition}\label{intBCdef}
 Given a category $\C$, let $\int B\C$ be the Grothendieck construction of the nerve of $\C$. Explicitly, objects of $\int B\C$ are pairs $(m,\, A \co [m] \to \C)= (m,\, A(0) \to A(1) \to \ldots \to A(m)) $ for $m \ge 0$, and a morphism $u \co (m,\, A \co [m] \to \C) \to (n,\, B \co [n] \to \C)  $ is a morphism $u \co \on \to \om$ in the simplex category such that $B= A\circ u$. 
\end{definition}

\begin{lemma}\label{intBClemma}
 The category $\C$ is equivalent to the simplicial localisation of $ \int B\C$ at the class $\cW_0$ of morphisms $u \co (m,\, A \co [m] \to \C)\to  (n,\, A \circ u\co [n] \to \C)$ with $u(0) =0 \in \om$. 
\end{lemma}
\begin{proof}
 The functor $\rho \co \int B\C \to \C$ is given by $\rho(m,A) = A(0)$ on objects and by $ \rho(u)=A(\gamma) \co A(0) \to A(u(0))$ on morphisms $u \co (m, A \co [m] \to \C)\to  (n, A \circ u \co [n] \to \C)$,  for the unique morphism $\gamma \co 0 \to u(0)$ in $\om$.

 In order to show that this induces an equivalence $L^{\cW_0}(\int B\C) \simeq \C$ of simplicial categories, by   \cite[Theorem 2.2]{DKEquivsHtpyDiagrams} it suffices to show that the functor $\rho^* \co (s\Set)^{\C} \to (s\Set)^{\int B\C, \cW_0}$ from the model category of simplicial set-valued functors on $\C$ to the model category of $\cW_0$-restricted $\int B\C$-diagrams 
 is a right Quillen equivalence.
 
 Now, giving a simplicial set-valued functor $F$ on $\int B\C$ is equivalent to giving a bisimplicial set $\int F$ over the nerve $B\C$, with  $(\int F)_n:= \coprod_{ x \in B_m \C} F(x)$. Thus the category $(s\Set)^{\int B\C, \cW_0}$ is equivalent to $ss\Set \da B\C$. Moreover, $F$ sends morphisms in $\cW_0$ to weak equivalences if and only if the maps $F(n,\,A(0) \to \ldots \to A(n))) \to F(0,\,A(0))$ are weak equivalences, which amounts to saying that the maps 
 \[
  (\int F)_n \xra{(\pd_1)^n} (\int F)_{0} \by_{B_{0}\C,(\pd_1)^n}B_n\C
 \]
are  weak equivalences. When $\int F$ is Reedy fibrant, this is equivalent to saying that $\int F$ is a left fibration over $B\C$ in the sense of \cite{debritoSegalGrothendieck}. 

The equivalence $\int^{-1} \co s\Set \da B\C \to  (s\Set)^{\int B\C, \cW_0}$ of categories is a right Quillen functor for the left fibration model structure on $s\Set \da B\C $, and the observation above implies that it is a Quillen equivalence. Composing this with the right Quillen equivalence of \cite[Theorem A]{debritoSegalGrothendieck} (in its variant for left fibrations) then gives the required result.
\end{proof}

If $\C$ is the category of $R$-CDGAs, EFC-DGAs or $\C^{\infty}$-DGAs, then for any diagram $D \co I \to \C$, we can define $D^{\poly}_{\oplus}(D)$ and $\Pol(D,0)$ by substituting $\HHom_B(M,N)$ with 
the equaliser of the obvious diagram
\[
\prod_{i\in I} \HHom_{B(i)}(M(i),N(i)) \implies \prod_{f\co i \to j \text{ in } I}   \HHom_{B(i)}(M(i),f_*N(j)).
\]
throughout the corresponding definitions for algebras.
These constructions behave well for diagrams $D \co [m] \to \C_{c, \onto}$ to the subcategory $\C_{c, \onto} \subset \C$ of cofibrant objects and surjective morphisms.
Since a morphism $u \co [n] \to [m]$ naturally induces a morphism $u^{\#} \co  D^{\poly}_{\oplus}(D)\to D^{\poly}_{\oplus}(u^*D) $ by restriction, the constructions $D \mapsto (D^{\poly}_{\oplus}(D), \tau^{\HH})$ and $D \mapsto \Pol(D,0)$ define  functors on the Grothendieck construction $\int B \C_{c, \onto}$.

Since all objects of $\C$ are fibrant, the functor $\C_{c,\onto} \to \C$ induces a weak equivalence of simplicial categories on localisation at weak equivalences. We are unaware of a reference for this fact, but \cite[Proposition 5.2]{DKfunction} gives a closely related statement with a very similar proof. Since the category admits a right homotopy calculus of fractions,  the key to the argument 
is that for any morphism $f \co A \to B$ in $\C_c$, factorising the graph of $f$ leads to a weakly equivalent span $(A \xla{\sim}  P_fA \to B)$ in $\C_{c,\onto}$ with the first map a trivial fibration. The homotopy equivalence  is given by the span $(A \xla{\sim} P_fA \to P_{\id_B}B)$ (which lies in $\C_{c,\onto}$ whenever $f$ is a fibration) and its canonical levelwise trivial fibrations to the span above and to the span $(A = A \to B)$.  

When $u(0)=0$ and the morphisms in $D$ are all homotopy \'etale, the morphism $u^{\#}$ is moreover a filtered quasi-isomorphism, so the restriction of $D^{\poly}_{\oplus}$ to the subcategory $\int B (\C_{c, \onto,\et})$ of homotopy \'etale morphisms sends morphisms in $\cW_0$ to filtered quasi-isomorphisms. It then follows from Lemma \ref{intBClemma} and quasi-isomorphism invariance that $(D^{\poly}_{\oplus},\tau^{\HH})$ and $\Pol(-,0)$  define   $\infty$-functors on the  $\infty$-category  $\oL \C^{\et}$ of $R$-CDGAs, EFC-DGAs or $\C^{\infty}$-DGAs, and homotopy \'etale morphisms; the HKR quasi-isomorphism is then a natural transformation between them. Taking Maurer--Cartan elements as in Definition \ref{qpoldef} then  gives an  $\infty$-functor $Q\cP(-,0)$  from    $\oL \C^{\et}$ to the  $\infty$-category of simplicial sets.  

\subsubsection{Global quantisations}


Incorporating the homotopy \'etale functoriality of \S \ref{intBCsn} into the functoriality of Theorem \ref{fildefhochthm1}
immediately leads to the following generalisation of Corollary \ref{affquantcor} on taking derived global sections on the \'etale site:
\begin{corollary}\label{DMquantcor}
Given a derived DM $n$-stack $\fX$ over $R$  with  perfect cotangent complex $\oL\Omega^1_{\fX/R}$, 
the space  $Q\cP(\fX,0):= \oR\Gamma(\fX_{\et}, Q\cP(\sO,0)) $   of $0$-shifted quantisations of $\fX$ from \cite[Definitions \ref{DQnonneg-Qpoissdef}, \ref{DQnonneg-QpoissdefX}]{DQnonneg} and its subspace $Q\cP(\fX,0)^{sd}$ of self-dual (or involutive) quantisations from  \cite[Definition 1.33]{DQnonneg} are respectively equivalent to the Maurer--Cartan spaces
\begin{align*}
  &\oR\Gamma(\fX_{\et},\mmc(  F^2\widehat{\Pol}(\sO,0)_{[-1]} \by\hbar F^1\widehat{\Pol}(\sO,0)_{[-1]} \by \hbar^2\widehat{\Pol}(\sO,0)_{[-1]}\brh)),\\
   &\oR\Gamma(\fX_{\et},\mmc(  F^2\widehat{\Pol}(\sO,0)_{[-1]} \by \hbar^2\widehat{\Pol}(\sO,0)_{[-1]}\brhh)).
   \end{align*}

In particular, every Poisson structure $\pi \in \cP(\fX,0) = \oR\Gamma(\fX_{\et}, \mmc( F^2\widehat{\Pol}(\sO,0)_{[-1]}  )) $
admits self-dual quantisations in the form of almost commutative curved $A_{\infty}$-deformations $\sA_{\hbar}$  of $\sO_{\fX}$ with $\sA_{-\hbar}\simeq \sA_{\hbar}^{\op}$.

The analogous statements for derived $\C^{\infty}$ and derived analytic DM $n$-stacks (in the sense of \cite{DStein}) with perfect cotangent complexes also hold.
\end{corollary}

\begin{remarks}
The hypotheses of 
Corollary \ref{DMquantcor} 
are satisfied by any derived Deligne--Mumford stack locally of finite presentation over the CDGA $R$. When $R=\H_0R$, this includes underived  schemes $X$ which are local complete intersections over $R$, in which case the cotangent complex $\oL\Omega^1_{X/R}$ is concentrated  in homological degrees $[0,1]$. For such underived schemes, a quantisation in the sense of the corollary reduces to the usual notion, namely a DQ algebroid deformation of $\sO_X$ over $R\brh$. For more details, see \cite[Remarks \ref{DQLag-curvedrmk}, \ref{DQLag-algdex} and \ref{DQLag-curvedsdrmk}]{DQLag}.

In analytic settings, Corollary \ref{DMquantcor} similarly gives DQ algebroid quantisations for local complete intersections. Since $\C^{\infty}$ spaces tend to embed in affine space under mild finiteness hypotheses, Corollary \ref{affquantcor} also gives strict quantisations for LCI $\C^{\infty}$ spaces. 

Our hypothesis that $\fX$ have perfect cotangent complex cannot be  removed, since Mathieu's example \cite{mathieu} gives a non-quantisable Poisson structure on a non-LCI scheme. 
\end{remarks}

\begin{remark}\label{strictquantnrmk}
The space $Q\cP^+(\fX,0)$ of strict deformation quantisations, i.e. quantisations by $BD_1$-algebras (associative, almost commutative algebras) rather than algebroids, is given by replacing $ Q\widehat{\Pol}(\sO_{\fX},0)$ with the kernel $ \tilde{F}^2Q\widehat{\Pol}^+(\sO_{\fX},0)$ of the natural map $ \tilde{F}^2Q\widehat{\Pol}(\sO_{\fX},0) \to \hbar\sO_{\fX}\brh $, and similarly by $ \tilde{F}^2Q\widehat{\Pol}^+(\sO_{\fX},0)^{sd}:= \ker(  \tilde{F}^2Q\widehat{\Pol}(\sO_{\fX},0) \to \hbar\sO_{\fX}\brhh)$ for $Q\cP^+(\fX,0)^{sd}$. These maps are not in general compatible with the formality quasi-isomorphism of Corollary \ref{DMquantcor}, but they are pro-nilpotent surjections of DGLAs, so lead to   homotopy fibre sequences
\begin{align*}
 Q\cP^+(\fX,0) \to Q\cP(\fX,0) \to \oR\Gamma(\fX, \hbar \sO_{\fX}\brh)_{[-2]},\\
Q\cP^+(\fX,0)^{sd} \to Q\cP(\fX,0)^{sd} \to \oR\Gamma(\fX, \hbar \sO_{\fX}\brhh)_{[-2]}.
 \end{align*}
In particular this means that if $\H^2(\fX, \sO_{\fX})=0$ then every (self-dual) quantisation in  Corollary \ref{DMquantcor} comes from a strict (self-dual) quantisation. 
\end{remark}



\subsubsection{Quantisation of $1$-shifted co-isotropic structures}\label{coisosn1} 


In \cite[\S 5.3]{MelaniSafronovII}, existence of quantisations for $n$-shifted co-isotropic structures is established as a direct consequence of the formality of the $E_n$ operad, for $n>1$. Theorem \ref{fildefhochthm1} allows us to establish the corresponding result for  $1$-shifted co-isotropic structures. 
 
Before defining quantised co-isotropic structures, we need a few preliminaries. We will use the term $A_{\infty}$-brace algebra to refer to  algebras $B$ for the 
brace operad $\Br$ of \cite[\S 3.4]{willwacherHtpyBracesFormality}, corresponding to the restricted class of $B_{\infty}$-algebras from \cite[\S 5.2]{GetzlerJones}; this enlarges the class of brace algebras from \S \ref{HHsn} by having an $A_{\infty}$-algebra structure rather than just an associative cup product, and can be encoded as a bialgebra structure on  the dg coalgebra $\b_{\Ass}B$ given by the bar construction, with restrictions.

\begin{definition}\label{TQpoldef}
Given a $0$-shifted quantisation $\Delta \in  \mc(\tilde{F}^2Q\widehat{\Pol}(A,0)[1]) $, define the centre of $(A,\Delta)$ by
\[
 T_{\Delta}Q\widehat{\Pol}(A,0):= \prod_{p \ge 0}F_p\CCC^{\bt}_R(A)\hbar^{p},
\]
with derivation $\delta \pm b + [\Delta,-]$ (necessarily square-zero by the Maurer--Cartan conditions). The  $A_{\infty}$ operations $\smile +\{\Delta\}\{\}_2$ and  $\{\Delta\}\{\}_i$ for $i \ge 3$ and unchanged brace operations $\{\}\{\}_j$ for $j \ge 1$ make  this an $A_{\infty}$ brace algebra. Explicitly, this structure corresponds to twisting the differential on the bialgebra $\b_{\Ass}T_0Q\widehat{\Pol}(A,0)$ 
 by the commutator of $\Delta$ with respect to the product induced by the brace operations. 

Similarly, given a $0$-shifted Poisson structure $\pi \in  \mc(F^2\widehat{\Pol}(A,0)$, define the centre of $(A,\pi)$ by 
\[
 T_{\pi}\widehat{\Pol}(A,0):=  (\widehat{\Pol}(A,0),\delta_{\widehat{\Pol}} +[\pi,-]), 
\]
with the same commutative product and shifted Lie bracket as $\widehat{\Pol}(A,0)$ making this a $P_2$-algebra.
\end{definition}

We now take the model for $BD_2$ given by the completed Rees construction of the good truncation filtration on the $A_{\infty}$ brace operad, so  a $BD_2$-algebra is an almost commutative $A_{\infty}$ brace algebra in  complete $R\brh$-modules  equipped with the $\hbar$-adic filtration.

Note that the following  is slightly weaker than the $n=1$ case of  \cite[Definition 5.13]{MelaniSafronovII}, since we cannot require the quantisation of $\fX$ to be strict.
\begin{definition}\label{qcoisodef}
 Define a $1$-shifted  quantised co-isotropic structure on a morphism $f \co \fX \to \fY$ of  (algebraic) derived DM $n$-stacks  to consist of:
 \begin{enumerate}
  \item a quantised $0$-shifted Poisson structure $\Delta$ on $\fX$,
  \item  a quantised $1$-shifted Poisson structure on $\fY$ in the form of a
 $BD_2$-algebra $\tilde{\sO}_{\fY}$  deforming the CDGA $\sO_{\fY}$ in the sense that it is equipped with a CDGA quasi-isomorphism $\tilde{\sO}_{\fY}\ten_{R\brh}^{\oL}R \to \sO_{\fY}$, and
 \item   a $BD_2$-algebra morphism 
 \[
F \co   f^{-1}\tilde{\sO}_{\fY} \to T_{\Delta}Q\widehat{\Pol}(\sO_{\fX},0)
 \]
 of \'etale hypersheaves on $\fX$,  where the target is defined using the functoriality of \S \ref{intBCsn}.
 \end{enumerate}

 We say that this quantisation is self-dual if $\Delta$ is self-dual, the $BD_2$-algebra  $\tilde{\sO}_{\fY}$ is given a quasi-involutive structure and $F$ intertwines the involutions.
  \end{definition}
 
Note that on reducing modulo $\hbar$ and applying the HKR quasi-isomorphism, this gives a $0$-shifted Poisson structure $\pi$ on $\fX$, a  strong homotopy  $P_2$-algebra structure $\varpi$ on the CDGA $\sO_{\fY}$ (i.e. a $1$-shifted Poisson structure) and a zigzag   $f^{-1}\sO_{\fY} \xla{\sim} f^{-1}\tilde{\sO}_{\fY}/\hbar  \to T_{\pi}\widehat{\Pol}(\sO_{\fX},0)$ of strong homotopy $P_2$-algebra morphisms, with the first a quasi-isomorphism. This last is equivalent to giving a morphism $\rho \co f^{-1}\sO_{\fY}  \to T_{\pi}\widehat{\Pol}(\sO_{\fX},0)$ in the $\infty$-category of $P_2$-algebras, so $(\pi,\varpi,\rho)$ are precisely the data of a $1$-shifted co-isotropic structure on $f$ in the sense of \cite{MelaniSafronovII}.


\begin{corollary}\label{coisocor1}
Given a morphism $f \co \fX \to \fY$ of (algebraic) derived DM $n$-stacks  such that $\fX$ has perfect cotangent complex,
 every $1$-shifted co-isotropic structure on $f$   admits self-dual quantisations in the sense of Definition \ref{qcoisodef}.
 \end{corollary}
 \begin{proof}
 The twisting in Definition \ref{qcoisodef} commutes with formality, by \cite[Theorem 1.1]{DolgushevWillwacher}. Alternatively, we can interpret  $(F,\Delta)$ as a curved brace morphism, with any formality equivalence for the unital brace operad as a Hopf operad compatibly inducing formality for the curved brace operad by Koszul duality.  
 
 Applying the $\infty$-equivalence  $p_w$  for any even  associator $w$, and then composing with the filtered quasi-isomorphism $\mathrm{HKR} \circ \alpha_{w,B}$ from  the proof of Theorem \ref{fildefhochthm1}, we thus see that a  quantisation is equivalent to the data of Theorem \ref{fildefhochthm1} together with  a formal deformation  $p_w\tilde{\sO}_{\fY}$ of $\sO_{\fY}$ over $R\brh$ equipped with an $R\brh$-linear  $1$-shifted Poisson structure, and 
 a morphism 
 \[
 f^{-1}p_w\tilde{\sO}_{\fY} \to ( \prod_{p \ge 0}  \oL\Lambda^p_{\sO_{\fX}}(T_{\fX/R})_{[p]} \hbar^{p}\brh, \delta + \{\alpha_w\Delta, -\})
 \]
 of strong homotopy $P_2$-algebras over $R\brh$.  In the notation of \cite{MelaniSafronovII}, this is  a formal $P_{[2,1]}\brh$-algebra deformation  of a $P_{[2,1]}$-algebra (i.e. $1$-shifted co-isotropic) structure on $f$, except that we allow our $P_{[2,1]}\brh$-algebra to incorporate $\hbar^2$-curvature in the underlying $P_1\brh$-algebra deformation of $\sO_{\fX}$. When the quantisation is self-dual, the deformation is invariant under the involution $\hbar \mapsto -\hbar$, or equivalently induced by a deformation over $R\brhh$.
 
 In particular, every $P_{[2,1]}$-algebra structure $\sP$   on $(\sO_{\fY}, f^{-1}\sO_{\fY} \to \sO_{\fX})$ admits  a self-dual quantisation corresponding under the equivalence above to  the trivial $P_{[2,1]}\brh$-algebra deformation $\sP\brh$  of $\sP$.
 \end{proof}

 \begin{remark}[Analytic and $\C^{\infty}$ analogues]\label{EFCcoisormk}
  For structures defined as in Definition \ref{qcoisodef}, the proof of Corollary \ref{coisocor1} is still valid for analytic and $\C^{\infty}$ derived DM $n$-stacks. However, such structures are too weak to quantise the correct notion of co-isotropic structures in those cases, since the s.h. $P_2$-algebra $\tilde{\sO}_{\fY}/\hbar$ and s.h. $P_2$-algebra morphism $F/\hbar$ should  be defined in terms of EFC or $\C^{\infty}$ multiderivations. 
  
Instead of constructing $BD_2$-algebras and morphisms in the category of complexes over $R\brh$, the solution is to work in the pre-triangulated dg category generated by $\sO_{\fY}\brh$-modules and differential operators. Although this dg category is not monoidal, we can still define algebras by endowing it with the dg multicategory (i.e. coloured dg operad) structure in which multi-operations $(M_1, \ldots, M_r) \to N$ are given by (EFC or $\C^{\infty}$) polydifferential operators. Because an EFC or $\C^{\infty}$ polydifferential operator which is an algebraic derivation is automatically an EFC or $\C^{\infty}$ multiderivation, $P_k$-algebras in that dg multicategory do correspond to $(k-1)$-shifted Poisson structures.

Working  in that dg multicategory and recasting Definition \ref{qcoisodef} accordingly,
the proof of Corollary \ref{coisocor1} then gives genuine deformation quantisations of  $1$-shifted co-isotropic structures in   analytic and $\C^{\infty}$ settings. 
\end{remark}

\section{Quantisations on derived Artin stacks} \label{Artinsn}

When applied to (derived) Artin stacks, the definition of  Poisson structures in \cite[\S \ref{poisson-Artinsn}]{poisson} or \cite{CPTVV} and their
quantisations in \cite[Definitions \ref{DQnonneg-Qpoissdef}, \ref{DQnonneg-QpoissdefX}]{DQnonneg} is more subtle than for derived DM stacks, since polyvectors and the Hochschild complex are not functorial with respect to smooth morphisms. 

This is resolved in \cite[\S \ref{poisson-stackyCDGAsn}]{poisson} by observing that the formal completion of a  derived Artin stack $\fX$ along an affine atlas $f\co U \to \fX$ with $f$ smooth can be recovered from a 
%
commutative algebra in double complexes. 

These lead to sufficiently functorial constructions of polyvectors and Hochschild complexes, but only after passing to sum-product total complexes $\hatTot$ (Tate realisations in the terminology of \cite{CPTVV}).  
Generalising quantisation results to derived Artin stacks is thus far from straightforward, and to establish them we  introduce an intermediate category  $\cU_{P_k^{ac}}[\hbar^2]$ in \S \ref{LinftyTate} which delicately balances the requirements of functoriality and deformations.   

\begin{remark}
 While the results of \S \ref{fildefpoisssn} easily adapt  to  $P_k^{ac}$-algebras in double complexes, the $P_{n+2}$-algebra of $n$-shifted  polyvectors only satisfies the analogue of  Proposition \ref{weightprop} before applying $\hatTot$. In terms of our notation below, $\cPol(-,n)$ (roughly corresponding to $\Pol^{\mathrm{int}}(-,n)$ in \cite{CPTVV}) is not sufficiently functorial, while the algebra $\hatTot_{\bG_m} \cPol(-,n)$  (roughly corresponding to $\Pol^t(-,n)$ in \cite{CPTVV}) cannot simply be described as derived symmetric powers of the tangent complex, so we cannot constrain its deformations. 
\end{remark}


\subsection{Double complexes and stacky Hochschild complexes}\label{stackyHHsn}

 \begin{definition} \label{stackyCDGAdef}
A stacky CDGA is  a chain cochain complex (i.e. a double complex) 
\[
 A^{\bt}_{\bt} = (A^0_{\bt} \xra{\pd} A^1_{\bt}\xra{\pd} A^2_{\bt}\xra{\pd}\ldots).
\]
equipped with a commutative product $A\ten A \to A$ and unit $\Q \to A$.  Given a chain CDGA $R$, a stacky CDGA over $R$ is then a morphism $R \to A$ of stacky CDGAs, where we silently regard $R$ as a stacky CDGA concentrated in cochain degree $0$.

A stacky $\C^{\infty}$-DGA (resp. EFC-DGA over $K$) is a stacky CDGA over $\R$ (resp. $K$) $A^{\bt}_{\bt}$ equipped with a $\C^{\infty}$-DGA (resp. EFC-DGA) structure on $A^0_{\bt}$ and such that $\pd \co A^0_0 \to A^1_0$ is a  $\C^{\infty}$-derivation (resp. EFC-derivation) in the sense of \cite[Definition \ref{DStein-oplusdef}]{DStein}. 
\end{definition}
As explained in \cite[Remark \ref{poisson-cfCPTVV}]{poisson}, these correspond to the ``graded mixed cdgas'' of \cite{CPTVV} (but beware that the latter are something of a misnomer, not having mixed differentials). The structure in the chain direction encodes derived information, while the cochain direction encodes stacky structure.

For general derived Artin $n$-stacks, these formal completions are constructed in \cite[\S \ref{poisson-stackyCDGAsn}]{poisson} by forming affine hypercovers as in \cite{stacks2}, and then applying the functor  $D^*$ (left adjoint to denormalisation) to obtain a stacky CDGA. When $\fX$ is a   derived Artin $1$-stack, 
the formal completion of an affine atlas $U \to \fX$ is simply  given by the relative de Rham complex
\[
 O(U) \xra{\pd} \Omega^1_{U/\fX} \xra{\pd} \Omega^2_{U/\fX}\xra{\pd}\ldots,
\]
which arises by applying the functor $D^*$  
to the \v Cech nerve of $U$ over $\fX$.

\begin{lemma}\label{modellemma}
 There is a cofibrantly generated model structure on  the category of cochain chain complexes in which 
 fibrations are surjections and weak equivalences are levelwise quasi-isomorphisms in the chain direction. For any  chain operad $\cP$, this induces  a cofibrantly generated model structure on   $\cP$-algebras in cochain chain complexes, in which 
 fibrations  and weak equivalences are those of the underlying cochain chain complexes.
 
 Similarly, for any $\bG_m$-equivariant chain operad $\cP$, there is a   cofibrantly generated model structure on  the category of $\bG_m$-equivariant  $\cP$-algebras in cochain chain complexes with the same cofibrations, fibrations  and weak equivalences.
\end{lemma}
\begin{proof}
This follows as in  the proof of \cite[Lemma \ref{poisson-bicdgamodel}]{poisson}.  Each generating (trivial) cofibration in the non-$\bG_m$-equivariant setting gives rise to a $\Z$-indexed family of  generating (trivial) cofibrations in the $\bG_m$-equivariant setting, corresponding to choices of weight for the generators.
\end{proof}
 Note that since $\cP$ is assumed to be  a chain operad, the operations for the $\cP$-algebras $A$ in Lemma \ref{modellemma} are  maps 
\[
\cP(n)_i\ten A_{j_1}^{m_1}\ten \ldots \ten  A_{j_n}^{m_n} \to A_{i+(j_1+\ldots +j_n)}^{m_1+\ldots +m_n};
\]
we refer to such  $\cP$-algebras in cochain chain complexes as stacky $\cP$-algebras; beware that this conflicts slightly with the terminology of Definition \ref{stackyCDGAdef} because  we allow stacky $\cP$-algebras to have terms of negative cochain degree, while our  stacky CDGAs are concentrated in non-negative cochain degrees.

We also have a more subtle variant:
\begin{lemma}\label{modellemmanonneg} 
Given a chain  CDGA $R$,
there is a cofibrantly generated model structure on the category of $\bG_m$-equivariant stacky $P_k^{ac}$-algebras $A$ over $R$ of non-negative weights (i.e. $\cW_iA=0$ for all $i<0$), in which 
 fibrations are surjections and weak equivalences are levelwise quasi-isomorphisms.
 
 Moreover, the forgetful functor from this category to the category of  $\bG_m$-equivariant stacky $R$-CDGAs preserves cofibrant objects.
\end{lemma}
\begin{proof}
 We apply \cite[Theorem 11.3.2]{Hirschhorn}  to the forgetful functor mapping to non-negatively weighted $\bG_m$-equivariant double complexes of $\Q$-vector spaces. The only non-trivial condition to check for the first statement is that the left adjoint $F$ sends trivial cofibrations to levelwise quasi-isomorphisms. For the second statement it suffices to show that for the class $C$ of cofibrations of double complexes, pushouts of morphisms in $FC$ over $FC$-cells (i.e. transfinite iterated pushouts of 
 morphisms in $FC$) are cofibrations of $\bG_m$-equivariant stacky $R$-CDGAs.
 
 Since the forgetful functor factors through the category of  non-negatively weighted $\bG_m$-equivariant stacky $s^{1-k}\hbar^{-1}\Lie$-algebras (by forgetting the commutative $R$-algebra structure but keeping the Lie bracket), the left adjoint $F$ factors as $(R\ten_{\Q}\Symm_{\Q})\circ L_+$, where $L_+$ is the left adjoint  to the forgetful functor from that intermediate category. For the Lie bracket of chain degree $k-1$ and $\bG_m$-weight $-1$, the functor $L_+$ sends a $\bG_m$-equivariant double complex $V$ to the quotient $L(V)/(\cW_{<0}L(V))$ of the free graded Lie algebra $L(V)$ (with Koszul signs) by the Lie ideal generated by terms of negative weight. 
 
 Now, a cofibration of double complexes is an injective map   $U \into V$ for which the quotient $V/U$ is acyclic in the cochain direction (i.e. $\H^i(V_j/U_j)=0$ for all $i,j$); it is trivial if it is moreover acyclic in the chain direction. On $V/U$ there is thus a contracting cochain homotopy $h \co (V/U)_{\#}^{\#} \to V/U^{\#[-1]}_{\#}$  such that the graded commutator $[\pd,h]$ is the identity, and if the cofibration is trivial, there is also a contracting chain homotopy $h'   \co (V/U)_{\#}^{\#} \to (V/U)^{\#}_{\#[1]}$ such that $[\delta,h']$ is the identity.  If we give $V$ an increasing filtration with $\Fil_{-1}V=0$, $\Fil_0V=U$ and $\Fil_1V=V$, then the Lie bracket induces an exhaustive increasing filtration $\Fil$ on $L_+(V)$, with $\gr^{\Fil}L_+(V)\cong L_+(U \oplus (V/U))$. 
 
 We can then define a derivation $H \co \gr^{\Fil}L_+(V)_{\#}^{\#}\to  \gr^{\Fil}L_+(V)^{\#[-1]}_{\#}$ (resp. $H' \co \gr^{\Fil}L_+(V)_{\#}^{\#}\to  \gr^{\Fil}L_+(V)^{\#}_{\#[1]}$), given on generators by $0$ on $U$ and $h$ (resp. $h'$) on $V/U$.  It follows that the commutator $[\pd,H]$ (resp. $[\delta,H]$) is a derivation acting as $0$ on generators $U$ and as the identity on generators $V/U$, so it must act as multiplication by $p$ on $\gr^F_pL_+(V)$. 
 
 The maps $\Fil_pL_+V \to \Fil_{p+1}L_+V$ are thus cofibrations (resp. trivial cofibrations) of double complexes for all $p \ge 0$, since $h/(p+1)$ and $H/(p+1)$ provide the necessary homotopies on the quotient, so $L_+U \to L_+V$ is a cofibration (resp. trivial cofibration) of double complexes. Passing to symmetric powers over $R$ then gives that $F(U) \to F(V)$ is a cofibration (resp. trivial cofibration) of $\bG_m$-equivariant CDGAs, proving the first statement.
 
 For the second statement, observe that cochain acyclicity of the quotient allows us to split any cofibration of double complexes if we turn off the chain differential. Any $FC$-cell $A$ thus has the form $(F(U_{\#}),\delta)$ for some differential $\delta$, and a pushout $B$ of $A$ along a map in $FC$ takes the form $ (F(U_{\#}\oplus V_{\#}),\delta)$ with $\delta(U \oplus V) \subset A$ and $V$ acyclic in the cochain direction. Taking the increasing filtration $\Fil$ of $L_+(U \oplus V)$ by powers of $V$, the quotients $\gr^{\Fil}_p L_+(U_{\#} \oplus V_{\#})$ are cochain acyclic with $\delta(\Fil_{p+1} ) \subset \Symm(\Fil_p)$, so taking symmetric powers gives a sequence of cofibrations $\Symm(\Fil_p) \to \Symm(\Fil_{p+1})$ 
 of stacky sub-CDGAs of $B$; composing them all implies that the morphism $A \to B$ is a cofibration of $\bG_m$-equivariant stacky CDGAs. 
  \end{proof}

Hochschild complexes and multiderivations on stacky CDGAs are then defined in terms of the semi-infinite total complex:
\begin{definition}\label{hatTotdef}
 The subcomplex
 $\hatTot  V \subset \Tot^{\Pi}V$ is given by
\[
(\hatTot  V)_m := (\bigoplus_{i < 0} V^i_{i+m}) \oplus (\prod_{i\ge 0}   V^i_{i+m})
\] 
with differential $\pd \pm \delta$. This functor has the properties that it maps levelwise quasi-isomorphisms in the chain direction to quasi-isomorphisms, and that it  is lax monoidal. 

Applied to the internal $\Hom$ functor $\cHom$,  this construction gives chain complexes
\[
 \hatHHom_R(M,N):= \hatTot \cHom_R(M,N)
\]
 for cochain complexes $M,N$ of  $R$-modules in chain complexes, and hence a dg enhancement of the monoidal category of $R$-cochain chain complexes.
\end{definition}

\begin{definition}\label{HHdefa}
 For a stacky DGAA $A$ (i.e. an associative algebra in double complexes) over a chain CDGA $R$, 
 we define the internal cohomological  Hochschild complex $\C\C_{R,\oplus}(A)$ 
 by replacing $\HHom$ with $\cHom$ in Definition \ref{HHdef0} (as in \cite[Definition \ref{DQLag-HHdefa}]{DQLag})  to give a  chain cochain complex. Note that for the conventions we have chosen, this means that the Hochschild differential is acting in the chain direction.

We similarly define $\cD^{\poly}_{\oplus}(A) \subset \C\C_{R,\oplus}(A)$ to be the double subcomplex of polydifferential operators (including the $\C^{\infty}$ setting with $R=\R$ and the EFC setting with $R=K$,  a complete valued field), defined exactly as in Definition \ref{Dpolydef0}, but with an additional grading coming from the cochain grading on $A$.

The filtration $\tau^{\HH}$ on $\C\C_{R,\oplus}(A)$  and $\cD_{\oplus}^{\poly}(A)$ is then given by good truncation in the Hochschild direction.
\end{definition}

\begin{definition}\label{cPoldef}
 For $A$ a cofibrant stacky $R$-CDGA (resp. stacky $\C^{\infty}$-DGA, resp. stacky EFC-DGA),  define the $\bG_m$-equivariant  stacky $P_{n+2}^{ac}$-algebra 
 $\cPol(A,n)$ by
\[
\cPol(A,n):= \bigoplus_{p \ge 0} \cHom_A(\CoS_A^p((\Omega^1_{A, \boxempty})_{[-n-1]}),A),
\]
 where $p$ is the weight for the $\bG_m$-action, the commutative multiplication and Lie bracket are defined in the usual way for  polyvectors, and $\Omega^p_{A,\boxempty}$ is $\Omega^p_{A/R}$ (resp. $\Omega^p_{A,\C^{\infty}}$, resp. $\Omega^p_{A/K,\EFC} $). 
 \end{definition}

In particular, note that $\cPol(A,0)= \bigoplus_{p \ge 0} \cHom_A(\Omega^p_{A,\boxempty},A)_{[p]}$.


 \subsection{\'Etale functoriality and \tps{$\hatHHom$}{Tate}-equivalences}\label{hatHHomsn}

We now consider conditions for a morphism $f \co P \to Q$ of $A$-modules to be a $\hatHHom_A$-homotopy equivalence, i.e. for the homology class $[f] \in \H_0\hatHHom_A(P, Q)$ to have an inverse in  $\H_0\hatHHom_A(Q, P)$. Writing $\sigma^{\le n}M:=M/\sigma^{>n}M$ for the brutal cotruncation  in the cochain direction, we have:

\begin{lemma}\label{hatHHomlemma} 
 If $A$ is a stacky CDGA concentrated in non-negative cochain degrees, and $P$ and $Q$ are cofibrant $A$-modules in double complexes, with the chain complexes $(P\ten_AA^0)^i, (Q\ten_AA^0)^i $  zero for all $i <r$ and 
 acyclic for all $i>s$, 
 then a morphism $f \co P \to Q$ 
 is  a $\hatHHom_A$-homotopy equivalence
whenever the map  $\Tot\sigma^{\le s} (P\ten_AA^0) \to \Tot \sigma^{\le s}(Q\ten_AA^0)$ is a quasi-isomorphism. 
Under this condition, the morphism $f^* \co \cHom_A(Q,A) \to \cHom_A(P,A)$ is also a  $\hatHHom_A$-homotopy equivalence. 
\end{lemma}
\begin{proof}
Since $P$ is cofibrant, brutal truncation $\{\sigma^{\ge p}M\}_p$ of any $A$-module $M$ in cochain chain complexes $M$ induces surjections $\cHom_A(P, \sigma^{\ge i}M) \to  \cHom_A(P, \sigma^{\ge i-1}M) $ with kernels $\cHom_A(P, M^i)^{[-i]}$. The second  hypothesis on $P$ implies  
that $(\cHom_A(P, M^i)^{[-i]})^m$ is  acyclic for  $i>m+s$, thus giving a quasi-isomorphism
\[
 \cHom_A(P,M)^m \cong \Lim_n \cHom_A(P,\sigma^{\le n}M)^m \xra{\sim} \cHom_A(P,\sigma^{\le m+s}M)^m.
\] 

Moreover, if $M^i=0$ for all $i<t$, then  $\cHom_{A}(P, \sigma^{\le m+s}M)^m=0$ for all $m<t-s$, 
so $\cHom_A(P,M)^m \simeq 0 $ for all  $m<t-s$. Applying $\hatTot $ then gives
\[
 \hatHHom(P, M) \simeq \Tot^{\Pi}\sigma^{\ge t-s}\cHom( P,M). 
\]

Now, we have quasi-isomorphisms $\Tot \sigma^{\le s'}(P\ten_AA^0) \to \Tot \sigma^{\le s'}(Q\ten_AA^0)$  for all $s' \ge s$, giving quasi-isomorphisms $\Tot^{\Pi}\sigma^{\ge -s'}\cHom_{A}(Q, M^i)\to \Tot^{\Pi}\sigma^{\ge -s'}\cHom_{A}(P, M^i) $ for all $i$. Writing $M$ as the limit of the tower $ \ldots \to \sigma^{\le t+1}M \to \sigma^{\le t}M= (M^t)^{[-t]} $ and setting $s'= s+i-t$, these give a quasi-isomorphism 
\[
 \Tot^{\Pi}\sigma^{\ge t-s}\cHom_A( Q,M) \to \Tot^{\Pi}\sigma^{\ge t-s}\cHom_A( P,M).
\]

For all $M$ bounded below in the cochain direction, we therefore have quasi-isomorphisms
\[
f^* \co \hatHHom_A(Q, M)\to \hatHHom_A(P, M).
\]
The first hypothesis on $P$ implies that  $P\cong \sigma^{\ge r}P$, so it is bounded below. 
We may therefore take $M=P$, giving a class $[g] \in \H_0\hatHHom_A(Q, P)$ with $f^*[g]=[\id]$. Thus $[g]$ is inverse to $[f] \in \H_0\hatHHom_A(P, Q)$, so $f$ is indeed a $\hatHHom_A$-homotopy equivalence.

Finally, observe that the contravariant functor $\cHom_A(-,A)$ on $A$-modules is a $\cHom_A$-enriched functor, having natural maps
\[
 \cHom_A(M,N)\to \cHom_A(\cHom_A(N,A),\cHom_A(M,A))
\]
for all $A$-modules $M,N$ in chain cochain complexes, compatible with composition.
Applying $\hatTot $ then gives maps
\[
 \hatHHom_A(M,N)\to \hatHHom_A(\cHom_A(N,A),\cHom_A(M,A))
\]
compatible with composition, so $[f]$ and $[g]$ give rise to mutually inverse elements of 
\[
 \H_0\hatHHom_A(\cHom_A(Q,A),\cHom_A(P,A)) \quad \text{ and }\quad \H_0\hatHHom_A(\cHom_A(P,A),\cHom_A(Q,A)).\qedhere
\]
\end{proof}

If $D$ denotes an $[m]$-diagram $(A(0)  \xra{f_1} A(1) \xra{f_2} \ldots \xra{f_m} A(m))$ of cofibrant stacky CDGAs with all  $f_i$  
surjective, then adapting
 \cite[\S \ref{DQnonneg-Artindiagsn}]{DQnonneg} (cf. \S \ref{intBCsn}), gives a stacky involutive a.c. brace algebra $\C\C_{R, \oplus}(D)$ (resp. $\cD^{\poly}_{\oplus}(D)$) 
 with restriction maps
\begin{align*}
 \C\C_{R, \oplus}(D)\to \C\C_{R, \oplus}(u^*D)\\ 
 \cD^{\poly}_{R, \oplus}(D)\to \cD^{\poly}_{R, \oplus}(u^*D)
  \end{align*}
for all maps $u \co [m'] \to [m]$. For $\C$ the relevant category of stacky CDGAs, EFC-DGAs or $\C^{\infty}$-DGAs, these combine to give a functor on the Grothendieck construction $\int B\C_{c,\onto}$ of the nerve (cf. \S \ref{intBCsn}).

Similarly, setting 
\[
 \cPol(A,B;n):= \bigoplus_{p \ge 0}\cHom_A(\CoS_A^p((\Omega^1_{A,\boxempty})_{[-n-1]}),B), 
\]
 where $\Omega^1_{A,\boxempty}$ is $\Omega^1_{A/R}$ (resp. $\Omega^1_{A,\C^{\infty}}$, resp. $\Omega^1_{A/K,\EFC} $) and
 $
\cPol(D,n)$ is 
\[
\cPol(A(0),n)\by_{\cPol(A(0),A(1);n)}\cPol(A(1),n) \ldots  \by_{\cPol(A(m-1),A(m);n)}\cPol(A(m),n)
%
\]
gives a stacky $P_{n+2}^{ac}$-algebra with restriction maps
\[
\cPol(D,n) \to \cPol(u^*D,n)
\]
for all maps $u \co [m'] \to [m]$.

As in \cite[Definition \ref{DQnonneg-hfetdef}]{DQnonneg}, we say that a map $A \to B$ of stacky CDGAs is homotopy formally \'etale if the map
\[
 \{\Tot \sigma^{\le q} (\oL\Omega_{A}^1\ten_{A}^{\oL}B^0)\}_q \to \{\Tot \sigma^{\le q}(\oL\Omega_{B}^1\ten_B^{\oL}B^0)\}_q
\]
is a pro-quasi-isomorphism, where $\sigma^{\le q}$ denotes the brutal cotruncation. When the maps  $f_i$ are all homotopy formally \'etale, the  map 
$
 \cPol(A(0) \xra{f_1} A(1) \xra{f_2} \ldots \xra{f_n} A(n),n) \to \cPol(A(0),n)
$
induced by the inclusion $[0] \to [m]$ becomes a quasi-isomorphism on applying $\hatTot $, as shown in \cite[\S \ref{DQnonneg-Artindiagsn}]{DQnonneg}. However, it will not usually be a  levelwise filtered quasi-isomorphism. Crucially for us, a slightly stronger property than $\hatTot$-quasi-isomorphism holds, as we will see in Lemma \ref{hatHHomPollemma} below.

%

\begin{definition}
Given a  stacky CDGA $A$ and $A$-modules $P,Q$ in cochain chain complexes, we define $\oR\hatHHom_A(P,Q)$ to be any of the quasi-isomorphic complexes $\hatHHom_A(\tilde{P},Q)$ given by replacing $P$ with a cofibrant  replacement $\tilde{P} \to P$  in the model structure of Lemma \ref{modellemma}; all choices are quasi-isomorphic because all objects are fibrant, so quasi-isomorphisms between cofibrant objects are homotopy equivalences in the strict model structure, inducing quasi-isomorphisms on $\cHom$ and hence on $\hatHHom$.

We say that a map $f \co P \to Q$ is an $\oR\hatHHom_A$-homotopy equivalence if the induced homology class $[f] \in \H_0 \oR\hatHHom_A(P,Q)$ has an inverse in $\H_0 \oR\hatHHom_A(Q,P)$. 
\end{definition}

\begin{lemma}\label{hatHHomPollemma} 
 Take a diagram $D= (A(0)  \xra{f_1} A(1) \xra{f_2} \ldots \xra{f_m} A(m))$ of cofibrant stacky $R$-CDGAs (resp. stacky EFC-DGAs, resp. stacky $\C^{\infty}$-DGAs) concentrated in non-negative cochain degrees with all  $f_i$  
surjective, such that 
\begin{enumerate}
 \item 
there exists $s \ge 0$ for which the  chain complexes $(\Omega^1_{A(i),\boxempty}\ten_{A(i)}A(i)^0)^r $ are acyclic for all $r >s$, and
 \item the maps $f_i$ are all homotopy formally \'etale in the sense that the maps
\[
\Tot \sigma^{\le s} (\Omega_{A(i-1),\boxempty}^1\ten_{A(i-1)}A(i)^0) \to \Tot \sigma^{\le s}(\Omega_{A(i),\boxempty}^1\ten_{A(i)}A(i)^0)
\]
are quasi-isomorphisms of $A(i)^0$-modules.
\end{enumerate}

Then  the natural morphisms 
\begin{align*}
 \cW_i\cPol(D,n) &\to \cW_i\cPol(A(0),n) 
\end{align*}
are all
 $\oR\hatHHom_{A(0)}$-homotopy equivalences.
\end{lemma}
\begin{proof}
Since each $f_i$ is surjective and taking symmetric invariants is an exact functor, the question reduces to showing that for $A=A(i-1)$ and $B=A(i)$, the maps
\[
 \cHom_B((\Omega^1_{B,\boxempty})^{\ten_B p},B)\to \cHom_A((\Omega^1_{A,\boxempty})^{\ten_A p},B)\cong \cHom_B((\Omega^1_{A,\boxempty})^{\ten_A p}\ten_AB,B)
\]
are 
$\oR\hatHHom_{A(0)}$-homotopy equivalences for all $i>0$. 

The boundedness hypotheses ensure that the chain complexes  $((\Omega^1_{A,\boxempty})^{\ten p}\ten_AA^0)^i$ and $((\Omega^1_{B,\boxempty})^{\ten p}\ten_BB^0)^i$ are acyclic for $i>sp$. Combined with the 
homotopy formally \'etale hypothesis, this   ensures that the maps $  (\Omega^1_{A,\boxempty})^{\ten p}\ten_AB\to (\Omega^1_{B,\boxempty})^{\ten p}$
satisfy the conditions of 
 Lemma \ref{hatHHomlemma}, so are  $\hatHHom_B$-homotopy equivalences, and hence $\oR\hatHHom_{A(0)}$-homotopy equivalences \emph{a fortiori}. 
\end{proof}
 
\begin{lemma}\label{hatHHomOmegalemma}
 If $A \to B$ is a morphism of cofibrant $\bG_m$-equivariant stacky CDGAs with non-negative $\bG_m$-weights, such that $\cW_0A \to \cW_0B$ is a levelwise quasi-isomorphism and the morphisms $\cW_iA \to \cW_iB$ are  $\hatHHom_{\cW_0A}$-homotopy equivalences for all $i \ge 0$, then the morphisms
 \[
 \cW_i (\Omega^1_A\ten_A\cW_0A) \to  \cW_i (\Omega^1_B\ten_B\cW_0B)
 \]
are $\hatHHom_{\cW_0A}$-homotopy equivalences for all $i$. 
\end{lemma}
\begin{proof}
 Since $\cW_0A \to \cW_0B$ is a levelwise quasi-isomorphism, so too is $\Omega^1_{\cW_0A}\to  \Omega^1_{\cW_0B}$. Using the exact sequence $0 \to \Omega^1_{\cW_0A}\ten_{\cW_0A}A \to \Omega^1_A \to \Omega^1_{A/\cW_0A} \to 0$,  we can thus replace $\Omega^1_A$ and $\Omega^1_B$ with $\Omega^1_{A/\cW_0A}$ and $\Omega^1_{B/\cW_0B}$ to give an equivalent statement.
 
 Because $A$ is cofibrant and $\cW_{>0}A $ is the augmentation ideal of the morphism $A \to \cW_0A$, we have  $\Omega^1_{A/\cW_0A}\ten_A\cW_0A\cong (\cW_{>0}A)/((\cW_{>0}A)\cdot (\cW_{>0}A))$, so the Koszul resolution (equivalently, the commutative bar construction  as in the proof of Lemma \ref{polyvectorlemma}) gives a canonical $\cW_0A$-module resolution of $  \Omega^1_{A/\cW_0A}\ten_A\cW_0A$ equipped with an increasing filtration whose graded pieces are  $\Co\Lie_n(\cW_{>0}A_{[-1]})_{[1]}$. In weight $i$, this becomes a finite filtration only involving terms with $n \le i$, so it suffices to observe that the maps  $\cW_j\Co\Lie^n_{\cW_0A}(\cW_{>0}A_{[-1]}) \to \cW_j\Co\Lie^n_{\cW_0B}(\cW_{>0}B_{[-1]})$ are all $\hatHHom_{\cW_0A}$-homotopy equivalences because they are finite tensor expressions in the terms $\cW_kA, \cW_kB$.  
 \end{proof}

\subsection{Poisson \tps{$L_{\infty}$}{L-infinity}-morphisms in the Tate category} \label{LinftyTate}

 For this section, we fix a chain CDGA $R$ over $\Q$ to act as our base ring.
 
We now introduce a variant of Definition \ref{hatTotdef} incorporating non-negatively weighted $\bG_m$-actions. 
\begin{definition}\label{cTdef}
 Define the non-negatively weighted Tate dg category $\cT^+_{R,dg}$ 
 as follows. Objects are $\bG_m$-equivariant $R$-modules in chain cochain complexes for which the $\bG_m$-weights are non-negative. Morphisms are given by the complexes
 \[
  \cT^+_{R,dg}(M,N):=\prod_{i \ge 0}\hatHHom_R(\cW_iM,\cW_iN)
 \]
with the obvious composition rule.

The non-negatively  weighted Tate category $\cT^+_{R}$ is defined to have the same objects, but with morphisms $\z_0\cT^+_{R,dg}(M,N)$.
\end{definition}
Note that the tensor product defines a bifunctor on these categories, since $\hatTot$ commutes with finite direct sums, so
\[
 \cT^+_{R,dg}(M\ten_RM',P)= \prod_{i\ge 0,j\ge 0}\hatHHom_R(\cW_iM\ten_R\cW_jM',\cW_{i+j}P),
\]
meaning that the maps 
\[
\hatHHom_R(\cW_iM,\cW_iN)\ten_R\hatHHom_R(\cW_jM',\cW_jN') \to  \hatHHom_R(\cW_iM\ten_R\cW_jM',\cW_{i+j}(N\ten N'))
\]
coming from lax monoidality of $\hatTot$ induce natural maps
\[
\cT^+_{R,dg}(M,N)\ten_R \cT^+_{R,dg}(M',N') \to \cT^+_{R,dg}(M\ten_RM',N\ten_RN').
\]
Note that our hypothesis that the $\bG_m$-weights be non-negative is essential for this to hold, since otherwise the expression $\cW_n(M\ten_RM')= \bigoplus_{i+j=n} \cW_iM\ten_R\cW_jM'$ would not be finite.

Also observe that $\cT^+_{R,dg}$ has a dg-subcategory consisting of  $\bG_m$-equivariant $R$-linear morphisms of double complexes, which we refer to as strict morphisms, given by the subcomplex  $\prod_{i \ge 0}\z^0\cHom_R(\cW_iM,\cW_iN)\subset \cT^+_{R,dg}(M,N)$. 
\begin{definition}\label{cTdefrel}
Given a stacky CDGA $A$ equipped with a $\bG_m$-action of non-negative weights, we   define  $\cT^+_{A,dg}$ to be the dg category whose objects are $\bG_m$-equivariant $A$-modules in chain cochain complexes for which the $\bG_m$-weights are non-negative and whose complexes $\cT^+_{A,dg}(M,N)$ of morphisms are given by the natural equaliser 
\[
 \cT^+_{R,dg}(M,N) \implies \cT^+_{R,dg}(A \ten_R M,N)
\]
coming from the observations above.

The  category $\cT^+_{A}$ is then defined to have the same objects, but with morphisms $\z_0\cT^+_{A,dg}(M,N)$.
\end{definition}
Since the bifunctor $\hatHHom_R$ respects finite colimits in the first argument, for any morphism $A \to B$ of non-negatively weighted $\bG_m$-equivariant CDGAs, the forgetful functor $\cT^+_{B,dg}\to  \cT^+_{A,dg}$ then has  a left adjoint given by $M \mapsto M\ten_AB$.

Since $\cT^+_{R}$ does not contain arbitrary coproducts, and in particular since infinite direct sums are not coproducts there, we cannot directly adapt \S \ref{polyvectorsn}. Instead, in order to express our obstructions in terms of polyvectors  we have to enlarge the class of morphisms as follows. 

For now, fix a  non-negatively weighted $\bG_m$-equivariant chain CDGA  $S$, which in applications will be either $R[\hbar^2]$ or $R[\hbar^2]/\hbar^{2m}$.

\begin{definition}\label{PoissonLinftydef} 
Take a symmetric monoidal pre-triangulated dg category $\C$ equipped with a tensor functor from the dg category of $\bG_m$-representations, and take $P_k^{ac}$-algebras $A,B$ in $\C$; write $\hbar^p\co \C \to \C$ for the dg functor given by tensoring with the $\bG_m$-representation of weight $p$. 

We then define  Poisson $L_{\infty}$-morphisms from $A$ to $B$ to consist of sequences $f= (f_1,f_2, \ldots)$ with $f_p \in \HHom_{\C}(\Symm^p(A_{[1-k]}),\hbar^{p-1}B)_{p-k}$, such that
\begin{enumerate}
 \item   $f$ is an $L_{\infty}$-morphism (i.e. satisfies the formulae of \cite[Proposition 10.2.13]{lodayvalletteoperads}) and  
\item \label{PoissonLinftymult} $f$ satisfies the multiplicative property
\[
f_{n+1}(ab,x_1, \ldots,x_n)=\sum_{\substack{  p+q=n\\ \sigma \in \mathrm{Sh}_{p,q}}} \pm  f_{p+1}(a,x_{\sigma(1)}, \ldots, x_{\sigma(p)})f_{q+1} (b,x_{\sigma(p+1)}, \ldots, x_{\sigma(n)}),
\]
 where $\pm$ denotes the appropriate Koszul sign (see Remark \ref{voronovrmk})  and $\mathrm{Sh}_{p,q}$ the set of $(p,q)$-shuffle permutations.
  \end{enumerate}

When $A,B$ are  weighted $\bG_m$-equivariant stacky $P_k^{ac}$-algebras over 
  $S$, we write $\Tate_{P_k^{ac},S,L_{\infty}}(A,B)$ for the set  of   Poisson $L_{\infty}$-morphisms from $A$ to $B$ in $\cT^+_{S,dg}$.
  \end{definition}

Again, observe that the category from Lemma \ref{modellemmanonneg} arises as a subcategory of $\Tate_{P_k^{ac},R,L_{\infty}}(A,B)$, consisting of morphisms $f$ with $f_1 \in \z_0\z^0\cHom_R(A,B) \subset \cT_{R,dg}^+(A,B)$ and $f_p=0$ for all $p>1$; we refer to these as strict morphisms.

\begin{remark}\label{voronovrmk}
         Adapting \cite[Theorem 1]{voronovMicroformal}, a conceptual interpretation of condition (\ref{PoissonLinftymult}) in Definition \ref{PoissonLinftydef} is that the tangent map of the Maurer--Cartan map $\omega \mapsto \sum_n f_n(\omega, \ldots, \omega)/n!$  is universally multiplicative. Explicitly, for any pro-nilpotent CDGA $C$ and any element $\omega \in (A \ten C)_{-1}$ (not necessarily satisfying the Maurer--Cartan equation), 
condition (\ref{PoissonLinftymult}) implies that the $C$-linear map $\sum_{n \ge 0} f_{n+1}(-,\omega, \ldots, \omega)/n! \co A \ten C \to B \ten C$ is multiplicative. Condition (\ref{PoissonLinftymult}) can be explicitly recovered from this property by letting $C$ contain free independent variables $t_1, \ldots, t_n$ with $\deg(t_i)=-\deg(x_i)-1$, then setting $\omega = \sum x_i t_i$. 
        \end{remark}

\begin{lemma}\label{PoissonLinftycatlemma}
   Poisson $L_{\infty}$-morphisms 
   are closed under   composition of $L_{\infty}$-morphisms.
\end{lemma}
\begin{proof}
 This can be checked by direct substitution or as in  \cite[Theorems 3 and 6]{voronovMicroformal}.
 
 Alternatively, we can form  bar constructions with respect to the $P_k$ and Lie operads (cf. Lemma \ref{polyvectorlemma})  in the ind-category 
 $\ind(\C)$,
 which does have infinite coproducts. An $L_{\infty}$-morphism $f$ corresponds to a 
 cocommutative coalgebra map  $\b_{\Lie}(f) \co   \b_{s^{1-k}\hbar^{-1}\Lie}A \to \b_{s^{1-k}\hbar^{-1}\Lie}B$ in the ind-category, with $f$ being given by restriction to cogenerators $B$. Since $\b_{s^{1-k}\hbar^{-1}\Lie}A$ is a quotient of the $P_k^{ac}$-coalgebra  $\b_{P_k^{ac}}A$, the same procedure 
 defines a map $\b_{P_k^{ac}}(f)\co \b_{P_k^{ac}}A \to \b_{P_k^{ac}}B$ of $\bG_m$-equivariant graded $P_k^{ac}$-coalgebras (i.e. a degree $0$ morphism in the dg category but not necessarily closed). The conditions  of Definition \ref{PoissonLinftydef}  are then precisely the evaluation on cogenerators $B$ of the condition that $\b_{P_k^{ac}}(f)$ be closed under the differential.
 
 Thus a Poisson $L_{\infty}$-morphism from $A$ to $B$ in $\C$ 
 corresponds to a   $P_k^{ac}$-coalgebra map  $\b_{P_k^{ac}}A \to \b_{P_k^{ac}}B$ in 
 $\ind(\C)$
 whose co-restriction to the natural coalgebra quotient $ \b_{s^{1-k}\hbar^{-1}\Lie}B$ factors through the corresponding quotient $ \b_{s^{1-k}\hbar^{-1}\Lie}A $ of $\b_{P_k^{ac}}A$. This property is manifestly closed under composition.
 \end{proof}

\begin{definition}\label{hatDDerdef} 
 Given a $\bG_m$-equivariant   stacky $P_k^{ac}$-algebra $A$ over 
 $S$ and a Beck $A$-module $M$ in the strict category of $\bG_m$-equivariant  $S$-modules in  double complexes, define the chain complex $\hatDDer_{P_k^{ac},S,L_{\infty}}(A,M)$ to consist in degree $n$ of  $\bG_m$-equivariant Poisson $L_{\infty}$-derivations $A \to \cone(M)_{[n]}$ in $\cT^+_{S}$ , with differential $\delta \co \hatDDer_{P_k^{ac},S,L_{\infty}}(A,M)_n \to \hatDDer_{P_k^{ac},S,L_{\infty}}(A,M)_{n-1}$ induced by the obvious map $\cone(M)_{[n]} \to  \cone(M)_{[n-1]}$.
 \end{definition}

 The following acts as our substitute for Lemma \ref{polyvectorlemma} in the stacky setting:
\begin{lemma}\label{polyvectorlemma2}
Given a non-negatively weighted  $\bG_m$-equivariant stacky $P_k^{ac}$-algebra $A$ over 
$R$ which is  cofibrant as a $\bG_m$-equivariant $R$-CDGA, together with a Beck $A$-module $M$, there is a complete decreasing filtration $F^1 \supset F^2 \supset \ldots $ 
 on the complex  $\hatDDer_{P_k^{ac}, R,L_{\infty}}(A, M)$ with associated graded complexes
 \[
  \cT^+_A(\Symm_A^p((\Omega^1_{A/R})_{[-k]}),\hbar^{p-1}M)_{[-k]},
 \]
 where $\Omega^1$ denotes the double  complex of K\"ahler differentials of the underlying stacky CDGA.
\end{lemma}
\begin{proof}
 Unwinding  Definition \ref{PoissonLinftydef}, a $\bG_m$-equivariant  Poisson $L_{\infty}$-derivation in $\cT^+_R$ is an $L_{\infty}$-derivation $\theta$ for which the elements $\theta_p \in \cT^+_{R,dg}(\Symm^p(A_{[1-k]}),\hbar^{p-1}M)_{p-k}$ satisfy $\theta_{p}(ab,y_2, \ldots,y_p)= a\theta_p(b, y_2, \ldots, y_p)\pm \theta_p(a,y_2, \ldots, y_p)b$, which, since the operations are symmetric,  says precisely that they are multiderivations.

 Thus if we forget differentials we have an isomorphism of graded $R$-modules
\[
 \hatDDer_{P_k^{ac}, R,L_{\infty}}(A, M)_{\#} \cong \prod_{p \ge 1} \cT^+_{A,dg}(\Symm_A^p((\Omega^1_{A/R})_{[-k]}),\hbar^{p-1}M)_{[\#-k]},
\]
and we can then set $F^j$ to be the product of the terms with $p \ge j$. Since $ \hatDDer_{P_k^{ac}, R,L_{\infty}}(A, M)$ is a subcomplex of the complex of $L_{\infty}$-derivations, the $p$th component $(\delta \theta)_p$ is an expression in terms of $\theta_j$ for $j \le p$. If $\theta_p=0$ for all $p<j$, this implies that $(\delta \theta)_p=0$ for all $p<j$, so $\delta(F^j) \subset F^j$ and we have a filtration by subcomplexes. Moreover, $(\delta \theta)_j= \delta(\theta_j)$ under those conditions, so  $\delta \theta \in \delta(\theta_j) +F^{j+1}$, giving the required description of the associated graded complex. 
 \end{proof}

\begin{definition}
 Say that a morphism $A \to B$ of stacky  $\cP$-algebras 
 is an abelian extension if it is surjective and if whenever a $\cP$-algebra operation on $A$ has more than one input in $\ker(A\to B)$, the output is zero.  Say that a morphism $A \to B$ is a nilpotent extension if it is a finite composition of abelian extensions.
 \end{definition}
 Note that the condition implies that $\ker(A \to B)$ is naturally a Beck $A$-module, and that this structure is induced by a Beck $B$-module structure. It also gives us a stacky  $\cP$-algebra isomorphism $A\by_BA \cong A\by_B(B\oplus \ker(A \to B))$. 

The following proposition is the technical key to this section:
\begin{proposition}\label{Tatedefprop} 
 Take strict morphisms $B' \xra{g} B \xra{h} \bar{B} \xla{f} A$ of non-negatively weighted $\bG_m$-equivariant stacky $P_k^{ac}$-algebras over 
 $S$, 
 such that $g$ is an abelian extension with kernel $I$, $h$ is a nilpotent extension, and the natural Beck $B$-module structure on $I$ is induced by a Beck $\bar{B}$-module structure.

 We then have a short exact sequence
 \[
  0 \to \z_0 \hatDDer_{P_k^{ac},S,L_{\infty}}(A,f_*I)
  \to \Tate_{P_k^{ac},S,L_{\infty}}(A, B')_f \xra{g_*} \Tate_{P_k^{ac}.S,L_{\infty}}(A, B)_f
 \]
of groups and sets, where $(-)_f$ denotes the fibre over $ f \in \Tate_{P_k^{ac},S,L_{\infty}}(A, \bar{B})$.

Moreover, if $A$  is cofibrant as a stacky CDGA in the model structure of Lemma \ref{modellemma}, then the short exact sequence extends to a further term $o_g \co \Tate_{P_k^{ac},S,L_{\infty}}(A, B)_f \to \H_{-1} \hatDDer_{P_k^{ac},S,L_{\infty}}(A,f_*I) $, so  $o_g(\phi)=0$ if and only if $\phi$ lies in the image of $g_*$.
 \end{proposition}
\begin{proof}
Given two elements $\phi, \psi \in \Tate_{P_k^{ac},S,L_{\infty}}(A, B')_f$ with $g \circ \phi = g \circ \psi$, consider the element $\phi -\psi \in \prod_{p\ge 1} \cT_{S,dg}^+(\Symm^p_S(A_{[1-k]}),\hbar^{p-1}B)_{p-k}$. Since $g$ is an abelian extension with $I$ a Beck $\bar{B}$-module, we have an isomorphism
\begin{align*}
 B'\by_{\bar{B}}(\bar{B} \oplus I) &\xra{\cong} B'\by_BB'\\
(b, (\bar{b},x)) &\mapsto (b, b+x)
 \end{align*}
of $P_k^{ac}$-algebras. Noting that the inclusion functor from the strict dg category of double complexes to $\cT_{S,dg}^+$ preserves finite limits, it follows that $(f, \phi -\psi)\co A \to \bar{B} \oplus I$ defines a Poisson $L_{\infty}$-morphism of $\bG_m$-equivariant $P_k^{ac}$-algebras   $\cT_{S}^+$. Expanding this out in terms of Definition \ref{PoissonLinftydef}, this says precisely that $\phi - \psi \in \z_0 \hatDDer_{P_k^{ac},S,L_{\infty}}(A,f_*I)$, which gives the first statement.

For the second statement, we start with an intermediate lemma.
\begin{lemma*}
When $A$ is cofibrant as a stacky CDGA,  every element $ \rho \in \Tate_{P_k^{ac},S,L_{\infty}}(A, B)_f$ admits a lift  $\tilde{\rho}$ in $\prod_{p\ge 1} \cT^+_{S,dg}(\Symm^p_S(A_{[1-k]}),\hbar^{p-1}B')_{p-k}$ satisfying the multiplicativity property of Definition \ref{PoissonLinftydef}. 
\end{lemma*}
\begin{proof}[Proof of lemma]
Since this lifting property is preserved on passage to a retract,  we may assume that  $A$ is freely  generated as a commutative algebra by a trigraded (i.e. $\bG_m$-equivariant bigraded) $\Q$-vector space  $V= \bigoplus_i \cW_i V^{\#}_{\#}$; forgetting the differentials  gives $A^{\#}_{\#} \cong \bigoplus_n S\ten_{\Q} \Symm_{\Q}^n(V^{\#}_{\#})$. 

If we let $\alpha$ be the restriction of $\rho$ to $V$, we can then choose a lift $\tilde{\alpha} \in \prod_{p \ge 1}\cT^+_{\Q}(\Symm^p_{\Q}(V_{[1-k]}) ,\hbar^{p-1}B')_{p-k}$, since $B' \to B$ is surjective. Repeated application of the multiplicativity property then gives us elements $\tilde{\rho}_p$ of $\Tot^{\Pi} (\cHom_S(\Symm^p_S(A_{[1-k]}),\hbar^{p-1}B'))_{p-k}$ agreeing with $\tilde{\alpha}_p$ on generators $V\subset A$, but we need to check that each term lies in $\cT^+_{dg} \subset \Tot^{\Pi}\cHom$. 
The iterated multiplicativity property gives, for $r:= \sum_{j=1}^p r_j$ and  $L=r+1-p$, 
\begin{align*}
 &\tilde{\rho}_p(\Symm^{r_1}V\ten \ldots \ten\Symm^{r_p}V)\\ 
&\subset \sum_{\substack{\sum_{l=1}^L m_l = r} } \tilde{\alpha}_{m_1}(V, \ldots,V)\cdots \tilde{\alpha}_{m_L}(V, \ldots, V),
\end{align*}
which in particular means that 
$m_l \le p$ for all $l$, so  $\tilde{\rho}_p$ depends only on $\{\tilde{\alpha}_m\}_{m \le p}$. 

Writing $\tilde{\alpha}_{p,i}$ for the component of $\tilde{\alpha}$  in $\hatHHom_{\Q}(\cW_i\Symm^p_{\Q}(V_{[1-k]}),\cW_{i-p+1}B')_{p-k}$, and $\tilde{\rho}_{p,i}$ for the restriction of $\tilde{\alpha}_p$ to $\cW_i$, we moreover have that $\tilde{\rho}_{p,i}$ depends only on the finite set $\{\tilde{\alpha}_{m,j}\}_{m \le p, j \le i}$. Since the set is finite, by definition of $\hatHHom$ there  exists an integer $N$ such that $\tilde{\alpha}_{m,j}$ lies in cochain degrees $[-N, \infty)$ for all $m\le p, j\le i$.

We can decompose $\tilde{\alpha}_{m,j}$ as $ \tilde{\alpha}_{m,j}^+ + \tilde{\alpha}_{m,j}^-$ into terms of non-negative and negative cochain degrees. 
 Since $f$ is a strict morphism, the image of $ \tilde{\alpha}_{m,j}^-$ must lie in the kernel $J'$ of the nilpotent surjection $B' \to \bar{B}$. If $n$ is the index of nilpotence of $J'$, then at most $n-1$ of the terms $ \tilde{\alpha}_{m,j}^-$ can feature in a product to give a non-zero contribution. Thus the cochain degree of $\tilde{\rho}_{p,i}$ is bounded below by $-(n-1)N$ (crucially independent of $r$), so  $\tilde{\rho}_{p,i} \in  \hatHHom_S(\cW_iSymm^p_S(A_{[1-k]}),\cW_{i+1-p} B')_{p-k}$, proving the lemma.
\end{proof}
To complete the proof of the proposition, we 
can now  proceed by a standard obstruction argument.

Given a morphism $\rho\in \Tate_{P_k^{ac},S,L_{\infty}}(A, B)_f$, we choose a lift $\tilde{\rho}$ as in the lemma. This will be an $L_{\infty}$ morphism if and only if the induced  coalgebra morphism  $\b_{\Lie}(\tilde{\rho})\co \b_{s^{1-k}\Lie}A \to \b_{s^{1-k}\Lie}B'$ as in the proof of Lemma \ref{PoissonLinftycatlemma} commutes with $\pd \pm \delta$. The resulting commutator $\kappa(\tilde{\rho}) \in \prod_p  \cT^+_{S,dg}(\Symm^p_S(A_{[1-k]}),\hbar^{p-1}B')_{p-k-1}$ is thus the obstruction to $\tilde{\rho}$ being an $L_{\infty}$ morphism. Since its projection $\kappa(\rho)$ to $B$ is $0$, 
it follows that $\kappa(\tilde{\rho})$ is an   $L_{\infty}$ derivation from $A$ to $I$. Because $B' \to B$ is an abelian extension, 
it moreover follows that  
 $\kappa(\tilde{\rho})$ is a Poisson $L_{\infty}$ derivation  and that the Poisson $L_{\infty}$ $A$-module structure on $I$ is that induced from $\rho \co A \to B$,  so 
\[
\kappa(\tilde{\rho}) \in \z_{-1}\hatDDer_{P_k^{ac},S,L_{\infty}}(A,f_*I).
\]
Any other lift of $\rho$ is of the form $\tilde{\rho} + \theta$ for $\theta \in \hatDDer _{P_k^{ac},S,L_{\infty}}(A,f_*I)_0$, since $B' \to B$ is an abelian extension. Then $\kappa(\tilde{\rho}+\theta)= \kappa(\tilde{\rho})+ [\pd \pm \delta,\theta]$, the abelian property killing all other terms,  so there exists a lift of $\rho$ to $\Tate_{P_k^{ac},S,L_{\infty}}(A, B')$ if and only if $\kappa(\tilde{\rho})$ is a boundary, i.e. whenever $o_g(\rho):= [\kappa(\tilde{\rho})]$ is $0$  in $\H_{-1}\hatDDer _{P_k^{ac},S,L_{\infty}}(A,f_*I)$.
\end{proof}

\begin{remark}
The cofibrancy condition in Proposition \ref{Tatedefprop} is stronger than the proof uses, since it suffices for  $A$ to be a retract of a quasi-free object. That is an instance  of a general phenomenon that deformation theory works most naturally with derived categories of the second kind.
\end{remark}

\begin{definition} 
 Define the category $\cU_{P_k^{ac}[\hbar^2]/\hbar^{2m}, R}$ (resp. $\cU_{P_k^{ac}[\hbar^2] , R}$) as follows. Objects are non-negatively weighted $\bG_m$-equivariant stacky $P_k^{ac}\ten_{\Q}R [\hbar^2]/\hbar^{2m}$-algebras (resp. $P_k^{ac}\ten_{\Q}R[\hbar^2]$-algebras) which are levelwise flat over $R [\hbar^2]/\hbar^{2m}$ (resp. $R[\hbar^2]$).
 
 Morphisms from $A$ to $B$ then consist of those elements of $\Tate_{P_k^{ac},R[\hbar^2]/\hbar^{2m} ,L_{\infty}}(A,B)$
 (resp. $\Tate_{P_k^{ac},R[\hbar^2],L_{\infty}}(A,B)$) which are strict morphisms modulo $\hbar^{2}$, i.e. morphisms of $\bG_m$-equivariant 
 $P_k^{ac}\ten_{\Q}R$-algebras
 in the strict category of double complexes.
 
 We then say that a morphism $A \to B$ is a weak equivalence if the induced (strict) morphism $A/\hbar^{2} \to B/\hbar^{2}$ is a levelwise quasi-isomorphism, i,e. induces isomorphisms $\H_i(A^j/\hbar^{2}) \to \H^i(B^j/\hbar^{2})$ for all $i,j$. 
\end{definition}
In particular, note that $\cU_{P_k^{ac},R} =\cU_{P_k^{ac}[\hbar^2]/\hbar^{2},R}$ is the strict category of non-negatively weighted $\bG_m$-equivariant stacky  $ P_k^{ac}\ten_{\Q}R$-algebras, i.e. the category with morphisms of cochain degree $0$ and with no higher $L_{\infty}$ terms.

Also observe that since $\cW_i(\hbar^{2m}B)=0$ for all $i<2m$, we have $\hatHHom_{R[\hbar^2]}(\cW_iA,\cW_iB)= \Lim_m \hatHHom_{R[\hbar^2]}(\cW_iA,\cW_i(B/\hbar^{2m}))$, so substitution in Definition \ref{cTdef} gives $\cT^+_{R[\hbar^2],dg}(A,B)\cong\Lim_m \cT^+_{R[\hbar^2]/\hbar^{2m},dg}(A/\hbar^{2m},B/\hbar^{2m})$ and thus 
\[
\cU_{P_k^{ac}[\hbar^2] , R}(A,B) \cong\Lim_m\cU_{P_k^{ac}[\hbar^2]/\hbar^{2m},R}(A/\hbar^{2m},B/\hbar^{2m}).
\]

\begin{definition}
 Letting $\cU$ be either of the categories $ \cU_{P_k^{ac}[\hbar^2]/\hbar^{2m},R}$ or $\cU_{P_k^{ac}[\hbar^2],R}$, and taking objects $A,B \in \cU$, define the simplicial set $\uline{\cU}(A,B)$ to be given by
 \[
 n \mapsto \cU(A, B\ten \Omega(\Delta^n)_{\bt}),
\]
 for the CDGA $\Omega(\Delta^n)_{\bt}$ of polynomial de Rham forms as in Definition \ref{mcPLdef}.
 \end{definition}

\begin{corollary}\label{Tatedefcor} 
Given  $A,B \in \cU_{P_k^{ac}[\hbar^2]/\hbar^{2m+2},R}$, with $A$ cofibrant in the model structure on  $\bG_m$-equivariant $P_k^{ac} \ten R[\hbar^2]/\hbar^{2m+2}$-algebras from Lemma \ref{modellemma},
the map
\[
\theta \co \uline{\cU}_{P_k^{ac}[\hbar^2]/\hbar^{2m+2},R} (A,B) \to \uline{\cU}_{P_k^{ac}[\hbar^2]/\hbar^{2m},R}(A/\hbar^{2m},B/\hbar^{2m})
\]
is a Kan fibration of simplicial sets. 

Moreover, for each element $f \in \uline{\cU}_{P_k^{ac}[\hbar^2]/\hbar^{2m},R}(A/\hbar^{2m},B/\hbar^{2m})$, there is a functorial obstruction $o(f) \in \H_{-1} \hatDDer_{P_k^{ac},R ,L_{\infty}}(A/\hbar^2,f_*\hbar^{2m}(B/\hbar^2)) $ which vanishes if and only if  the fibre $\theta^{-1}(f)$  is non-empty. If $o(f)$ vanishes, then  the homotopy group $\pi_i(\theta^{-1}(f))$  is naturally a torsor for $\H_i \hatDDer_{P_k^{ac}, \cT^+_R, L_{\infty}}(A/\hbar^2,f_*\hbar^{2m}(B/\hbar^2)) $. 
\end{corollary}
\begin{proof}
These are  fairly standard consequences of Proposition \ref{Tatedefprop}.
Since finite limits behave well in $\Tate_{P_k^{ac}[\hbar^2]/\hbar^{2m},R,L_{\infty}}$,   for any finite simplicial set $K$ we have 
\[
\Hom_{s\Set}(K, \uline{\cU}_{P_k^{ac}[\hbar^2]/\hbar^{2m},R}(A,B))\cong \cU_{P_k^{ac}[\hbar^2]/\hbar^{2m},R}(A,B \ten \Omega(K)_{\bt}),
\]
where $\Omega^n(K):= \Hom_{s\Set}(K, \Omega^n(\Delta^{\bt}))$. 

Letting $S:=R[\hbar^2]/\hbar^{2m+2}$,  the relative horn-filling condition for $\theta$ then amounts to surjectivity of 
\begin{align*}
  &\Tate_{P_k^{ac},S,L_{\infty}}(A,B \ten \Omega(\Delta^n)_{\bt})_{F} \to \\
  &\Tate_{P_k^{ac},S,L_{\infty} }(A,(B\ten \Omega(\Delta^n)_{\bt}/\hbar^{2m})\by_{B\ten \Omega(\L^{n,k})_{\bt}/\hbar^{2m}}  B\ten \Omega(\L^{n,k})_{\bt} )_{F},
  \end{align*}
for all $F \in \uline{\cU}_{P_k^{ac},R}(A/\hbar^2,B/\hbar^2)_n \subset \Tate_{ P_k^{ac},S,L_{\infty}}(A,B \ten \Omega(\Delta^n)_{\bt}/\hbar^2 )$, where the subscript denotes the fibre over $F$.  

Now, the map 
\[
B\ten \Omega(\Delta^n)_{\bt} \to  (B\ten \Omega(\Delta^n)_{\bt}/\hbar^{2m})\by_{B\ten \Omega(\L^{n,k})_{\bt}/\hbar^{2m}}  B\ten \Omega(\L^{n,k})_{\bt}
\]
is an abelian extension with kernel $\hbar^{2m}(B/\hbar^2)  \ten \ker( \Omega(\Delta^n)_{\bt} \to \Omega(\L^{n,k})_{\bt} ) $. Since the latter complex is acyclic, the obstruction of Proposition \ref{Tatedefprop} vanishes, giving the required surjectivity for $\theta$ to be a Kan fibration.

The obstruction $o(f)$ is then given by applying Proposition \ref{Tatedefprop} to the abelian extension $B \to B/\hbar^{2m}$ of  $P_k^{ac} \ten R[\hbar^2]/\hbar^{2m+2}$-algebras. Applying Proposition \ref{Tatedefprop}  to the extensions $B \ten \Omega(\Delta^n)_{\bt} \to B\ten \Omega(\Delta^n)_{\bt}/\hbar^{2m}$ and writing $V(n) :=  \hatDDer_{P_k^{ac},R,L_{\infty}}(A/\hbar^2,f_* \hbar^{2m}(B/\hbar^2) \ten \Omega(\Delta^n)_{\bt})$  gives us 
a faithful transitive action of   $\z_0 V(n) $
on $\theta^{-1}(f)_n$ for all $n$, and hence of $\pi_i \z_0V(-)$ on $\pi_i\theta^{-1}(f)$.  Since $A$ is cofibrant, the simplicial chain complex $V(-)$ is a Reedy fibrant replacement of  the chain complex $V(0)$, so the simplicial abelian group $\z_0V(-)$ has homotopy groups $\pi_i \z_0V(-)$ isomorphic to $\H_iV(0)$, as required.
\end{proof}

\begin{corollary}\label{Tatehomcor} 
For   $\cU:= \cU_{P_k^{ac}[\hbar^2]/\hbar^{2m}, R}$ (resp. $ \cU:= \cU_{P_k^{ac}[\hbar^2], R}$), with $A, B \in \cU$ such that  $A$ is cofibrant in the model structure on  $\bG_m$-equivariant stacky $P_k^{ac} \ten R[\hbar^2]/\hbar^{2m}$-algebras (resp. $P_k^{ac} \ten R[\hbar^2]$-algebras) from Lemma \ref{modellemmanonneg}, the simplicial set 
$\uline{\cU}(A,B)$
is a model for the derived mapping space
\[
 \oR\map_{\cU}(A,B)
\]
of morphisms from $A$ to $B$ in the simplicial localisation $L^{\cW}\cU$ of $\cU$  
at weak equivalences.
 \end{corollary}
\begin{proof}
We first show by induction on $m$ that the functor $\uline{\cU}(-,-)$ preserves weak equivalences in both factors provided we restrict to cofibrant $P_k^{ac}[\hbar^2]/\hbar^{2m}$-algebras in the first factor.  For $m=1$, this follows from  \cite[\S 5]{hovey}, since it is a right function complex with respect to the model structure of Lemma \ref{modellemmanonneg}. For the natural maps
 \[
  \uline{\cU}_{P_k^{ac}[\hbar^2]/\hbar^{2m+2}, R}(A,B) \to \uline{\cU}_{P_k^{ac}[\hbar^2]/\hbar^{2m}, R }(A/\hbar^{2m} ,B/\hbar^{2m}),
 \]
the description of Corollary \ref{Tatedefcor} implies that the homotopy fibres are invariant under weak equivalences in $A$ and $B$, 
which gives the inductive step. The case  $ \cU:= \cU_{P_k^{ac}[\hbar^2], R}$ then follows by taking the limit over $m$; since all maps are fibrations, the limit is a homotopy limit.

In the model category of Lemma \ref{modellemmanonneg}, there exists a Reedy cofibrant cosimplicial frame  $\tilde{A}^{\bt}$ of $A$ \cite[\S 5.4]{hovey}, leading to a simplicial set $n \mapsto \cU(\tilde{A}^n,B)$. The invariance above, together with an identical argument that $\cU(\tilde{A}^{\bt},-)$ preserves weak equivalences, then gives us weak equivalences
\[
 \cU(\tilde{A}^{\bt},B) \to \diag \uline{\cU}(\tilde{A}^{\bt},B) \la \uline{\cU}(A,B)
\]
of simplicial sets. In the  model category of restricted diagrams from \cite[\S 2.3.2]{hag1}, the natural morphisms $ \cU(A,-)\to  \cU(\tilde{A}^n, -) $ are all weak equivalences, since the maps $\tilde{A}^n \to A$ are weak equivalences in $ \cU$. Because $ \cU(\tilde{A}^{\bt},B)$ is a model for
 $ \ho\LLim_{n \in \Delta^{\op}} \cU(\tilde{A}^n, -) $, the natural map  $\cU(A, -) \to  \cU(\tilde{A}^{\bt}, -)$ is also a weak equivalence in the category of restricted diagrams. The result now follows by \cite{DKEquivsHtpyDiagrams}, as interpreted in  \cite[Theorem  2.3.5]{hag1}, since we have shown that $\cU(\tilde{A}^{\bt}, -)\simeq \uline{\cU}(A,-)$ preserves weak equivalences. 
\end{proof}

\subsection{Uniqueness results for deformations}

\subsubsection{Uniqueness for stacky $P_k^{ac}[\hbar^2]$-algebras}

\begin{corollary}\label{defPprop2cor} 
Take $A, B \in \cU_{ P_k^{ac}[\hbar^2], R}$ such that  $A$ is cofibrant in the model structure on  $\bG_m$-equivariant stacky  $P_k^{ac} \ten_{\Q} R[\hbar^2]$-algebras from Lemma \ref{modellemmanonneg}. If 
 \[
\H_i \hatDDer_{P_k^{ac},R,L_{\infty}}(A/\hbar^2A, \hbar^{2n}f_*(B/\hbar^2B) )\cong 0  
 \]
for all $i \ge -1$ and $n\ge 1$,  and all  morphisms $f \co A/\hbar^2A \to B/\hbar^2B$ in $\cU_{P_k^{ac},R}$, then   the natural map
\[
 \uline{\cU}_{P_k^{ac}[\hbar^2], R}(A,B)   \to \uline{\cU}_{P_k^{ac},R}(A/\hbar^2A, B/\hbar^2B)
\]
is a trivial fibration.  
 \end{corollary} 
\begin{proof}
Under these hypotheses, Corollary \ref{Tatedefcor} gives  fibrations
\[
\uline{\cU}_{P_k^{ac}[\hbar^2]/\hbar^{2n+2}, R}(A/\hbar^{2n+2},B/\hbar^{2n+2}) \to \uline{\cU}_{P_k^{ac}[\hbar^2]/\hbar^{2n},R}(A/\hbar^{2n},B/\hbar^{2n})
\]
for all $n \ge 1$, and shows that their fibres are contractible, making them trivial fibrations. The result then follows by taking the limit over all $n$.
\end{proof}

\begin{definition}\label{hatTotbGmdef}
Define the lax monoidal dg functor 
 $\hatTot_{\bG_m}$ from  $\bG_m$-equivariant double complexes to $\bG_m$-equivariant chain complexes by $\hatTot_{\bG_m}A := \bigoplus_i\hatTot\cW_iA$.
 \end{definition}

\begin{lemma}\label{hatTotUlemma}
There is an $\infty$-functor $\hatTot_{\bG_m}$ from 
the $\infty$-category  of non-negatively weighted $\bG_m$-equivariant stacky $P_k^{ac}\ten_{\Q}R[\hbar^2]$-algebras to the  $\infty$-category  of $\bG_m$-equivariant $P_k^{ac}\ten_{\Q}R[\hbar^2]$-algebras, and this extends to an $\infty$-functor on $L^{\cW}\cU_{P_k^{ac}[\hbar^2], R}$.
\end{lemma} 
\begin{proof}
Since $\hatTot$ is lax monoidal,  $\hatTot_{\bG_m}$ defines a functor from $\bG_m$-equivariant stacky $\cP$-algebras to  $\bG_m$-equivariant $\cP$-algebras for all chain operads $\cP$, and this automatically yields an $\infty$-functor because it preserves weak equivalences, sending levelwise quasi-isomorphisms to quasi-isomorphisms.

Similarly, $\hatTot_{\bG_m}$ defines a functor from 
$\cU_{P_k^{ac}[\hbar^2], R}$ to the category of $\bG_m$-equivariant Poisson $L_{\infty}$-morphisms between   $\bG_m$-equivariant $P_k^{ac}\ten_{\Q} R [\hbar^2]$-algebras. As in the proof of Lemma \ref{PoissonLinftycatlemma}, these are \emph{a fortiori} $R [\hbar^2]$-linear $(P_k)_{\infty}$-morphisms, i.e. $\infty$-morphisms in the sense of  \cite[\S 10.2.2]{lodayvalletteoperads}.

Composing with the $R[\hbar^2]$-linear analogue of the rectification functor $\Omega_{P_k^{ac}}\b_{P_k^{ac}}$ of \cite[Theorem 11.4.7]{lodayvalletteoperads} thus gives us a functor $F:= \Omega_{P_k^{ac}}\b_{P_k^{ac}}\hatTot_{\bG_m} $ on $\cU_{P_k^{ac}[\hbar^2], R}$,  with the  restriction of $F$ to the category of strict morphisms admitting a natural quasi-isomorphism  from  $\hatTot_{\bG_m}$. 
In order to show that $F$ induces the required $\infty$-functor, it suffices to show that it preserves weak equivalences. By  the $R[\hbar^2]$-linear analogue of \cite[Proposition 11.4.11]{lodayvalletteoperads}, that amounts to showing that the functor $\hatTot_{\bG_m}$ sends weak equivalences to $\infty$-quasi-isomorphisms (i.e. $(P_k)_{\infty}$-morphisms whose first component is a quasi-isomorphism).

 If $f \co A \to B$ is a weak equivalence in $\cU$, then $f_1 \co A/\hbar^2 \to B/\hbar^2$ is a levelwise quasi-isomorphism, so $f_1 \co  \hatTot \cW_i(A/\hbar^2) \to \hatTot \cW_i(B/\hbar^2)$ is a quasi-isomorphism.  Now, the flatness hypotheses on objects  of $\cU_{P_k^{ac}[\hbar^2], R}$ ensure that the maps $\hbar^{2k} \co (A/\hbar^2) \to (\hbar^{2k}A/\hbar^{2k+2}A)$ are isomorphisms and similarly for $B$, so by induction we see that $f_1 \co \hatTot \cW_i(A/\hbar^{2k}) \to \hatTot \cW_i(B/\hbar^{2k})$ is a quasi-isomorphism for all $k$. Taking $k>i/2$ then shows that the map $f_1 \co \hatTot \cW_iA \to \hatTot \cW_iB$ is  a quasi-isomorphism for all $i$. 
\end{proof}

\begin{definition}
 Given a non-negatively weighted $\bG_m$-equivariant   stacky $P_k^{ac}$-algebra $A$ over 
 $R$ and a Beck $A$-module $M$ in the strict category of $\bG_m$-equivariant    double complexes, define  $\oR\hatDDer_{P_k^{ac},R,L_{\infty}}(A,M):= \hatDDer_{P_k^{ac},R,L_{\infty}}(\tilde{A},M)$ for any cofibrant replacement $\tilde{A}$ of $A$ as a $\bG_m$-equivariant   stacky $P_k^{ac}$-algebra; note that by Lemma \ref{polyvectorlemma2} this is well-defined up to quasi-isomorphism because Lemma \ref{modellemmanonneg} implies that $\tilde{A}$ is cofibrant as a stacky CDGA.
 \end{definition}


 
\begin{theorem}\label{defPcor2thm} 
For $\bar{A}:=(A/\hbar^2A)$, the functors 
\[
 A \mapsto \bigoplus_i  \hatTot  \cW_iA\quad \text{ and }\quad 
 A \mapsto  \bigoplus_i  \hatTot  \cW_i( \bar{A}[\hbar^2])
\]
from the $\infty$-category  of non-negatively weighted $\bG_m$-equivariant stacky $P_k^{ac}\ten_{\Q}R[\hbar^2]$-algebras to the  $\infty$-category  of $\bG_m$-equivariant $P_k^{ac}\ten_{\Q}R[\hbar^2]$-algebras become naturally equivalent when restricted to 
objects $A$ which are flat over $R[\hbar^2]$ and satisfy
\begin{enumerate}
      \item[($\dagger$)] $\oR\hatDDer_{ P_k^{ac},R,L_{\infty}}(\bar{A} , M) \simeq 0$ for all Beck $\bar{A}$-modules $M$  of pure $\bG_m$-weight $\ge 2$. 
   \end{enumerate}
  \end{theorem}
  \begin{proof}
  We adapt the proof of Corollary \ref{defPcor}. As a preliminary, note that for all $P,N \in \cT^+_A$, we have $\cT^+_A(P,N) \cong \Lim_r \cT^+_A(P,N/\cW_{>r}N) $  essentially by construction. 
  Since $\bar{A}$ is cofibrant, we can therefore deduce by a filtration argument that $\oR\hatDDer_{ P_k^{ac},R,L_{\infty}}(\bar{A} , N) \simeq 0$ whenever  $\oR\hatDDer_{ P_k^{ac},R,L_{\infty}}(\bar{A} , \cW_rN) \simeq 0$ for all $r$, which the condition ($\dagger$) above ensures  for all Beck $\bar{A}$-modules $N$ with $\cW_iN = 0$ for all $i<2$.

 In particular, whenever $A$ satisfies  the condition  ($\dagger$) and $B$ is non-negatively weighted, we have  $\oR\hatDDer_{P_k^{ac}, R,L_{\infty}}(\bar{A}, \hbar^{2i}\bar{B})\simeq 0$ for all $i\ge 1$, so 
 Corollary \ref{defPprop2cor} implies that the map 
$\uline{\cU}_{P_k^{ac}[\hbar^2], R}(A,B)   \to \uline{\cU}_{P_k^{ac}, R}(A/\hbar^2A, B/\hbar^2B)$ is a weak equivalence. 
By Corollary \ref{Tatehomcor}, this means that the $\infty$-functor $A \mapsto A/\hbar^2A$ from $L^{\cW}\cU_{P_k^{ac}[\hbar^2], R}$ to $ L^{\cW}\cU_{P_k^{ac},R}$ becomes full and faithful when restricted to objects satisfying ($\dagger$) above. Since the functor $C \mapsto C[\hbar^2]$ is right inverse to reduction mod $\hbar^2$, it follows that the identity functor is naturally equivalent to the functor $A \mapsto (A/\hbar^2A)[\hbar^2]$ on this $\infty$-subcategory $(L^{\cW}\cU_{P_k^{ac}[\hbar^2], R})^{\%}$.

Since the $\infty$-functor $\hatTot_{\bG_m}$ factors through $(L^{\cW}\cU_{P_k^{ac}[\hbar^2], R})$ by Lemma \ref{hatTotUlemma}, composition
gives us the required equivalence $\hatTot_{\bG_m}A \simeq \hatTot_{\bG_m}((A/\hbar^2A)[\hbar^2])$, natural in  objects $A$ satisfying ($\dagger$).
     \end{proof}



\subsubsection{Uniqueness for involutive a.c. stacky $P_k$-algebras}
%

Adapting Definitions \ref{involutivelyfiltereddef} and \ref{involutivePacdef}, we have:
\begin{definition}\label{involutivePacdef2}
We say that a cochain chain complex $V_{\bt}^{\bt}$ is quasi-involutively filtered if it is equipped with a filtration $W$ by double subcomplexes and an involution $e$ which preserves $W$ and acts on $\H_*(\gr^W_iV^j)$ as multiplication by $(-1)^i$ for all $j$.

Define a  quasi-involutive a.c. stacky $P_k$-algebra over 
$R$ to be an  $(R \ten P_k^{ac},W,e)$-algebra $A$ in 
quasi-involutively  filtered cochain  chain complexes.
\end{definition}

We have the following immediate analogue of Lemma \ref{reeslemmaP}: 
\begin{lemma}\label{reeslemmaP2}
The Rees functor of Definition \ref{reesdef} gives an equivalence of $\infty$-categories from the category of 
   quasi-involutive a.c. stacky $P_k$-algebras over a chain CDGA $R$,
  localised at 
  filtered levelwise quasi-isomorphisms, to the $\infty$-category of  $\bG_m$-equivariant  $R \ten_{\Q} P_k^{ac} $-algebras in cochain chain complexes of flat  $\Q[\hbar^2]$-modules,   localised at levelwise quasi-isomorphisms. 
  \end{lemma}

Applying Lemmas \ref{reeslemmaP} and \ref{reeslemmaP2} and taking functorial cofibrant replacement, Theorem \ref{defPcor2thm} immediately gives the following:
\begin{corollary}\label{defPcor2} 
The functors $\hatTot$ and $\hatTot\, \gr^W$ from the $\infty$-category of 
 quasi-involutive a.c. stacky $P_k$-algebras in double complexes (localised at filtered levelwise  quasi-isomorphisms) to the $\infty$-category of  quasi-involutive a.c. $P_k$-algebras  become naturally equivalent when  restricted to objects $A$ satisfying the conditions:
   \begin{enumerate}
    \item\label{poswgt2} $W_{-1}A^j = 0$ for all $j$, and 
    \item\label{wgtsunder22}  $\oR\hatDDer_{P_k^{ac}, \cT^+_R, L_{\infty}}(\gr^W A, M) \simeq 0$ for all Beck $W_0A$-modules $M$ of pure $\bG_m$-weight $\ge 2$.
   \end{enumerate}
  \end{corollary}

  The following analogue of Proposition \ref{weightprop} is now an immediate consequence of Lemmas \ref{polyvectorlemma2} and \ref{hatHHomOmegalemma}:
  \begin{proposition}\label{weightprop2} 
 If $B$ is a  non-negatively weighted $\bG_m$-equivariant stacky $P_k^{ac}$-algebra over a CDGA $R$   for which the map  $(\cW_1\oL\Omega^1_{B/\cW_0B})\ten_{\cW_0B}^{\oL}B \to \oL\Omega^1_{B/\cW_0B}$ of commutative cotangent complexes is a quasi-isomorphism, then
\[
 \oR\hatDDer_{P_k^{ac}, \cT^+_R, L_{\infty}}(B, M) \simeq 0
 \]
 for all Beck $\cW_0B$-modules $M$ pure of $\bG_m$-weight $\ge 2$.
 
 Moreover, if $A$ is another such algebra, equipped with a morphism $A \to B$ such that $\cW_0A \to \cW_0B$ is a levelwise quasi-isomorphism, and if the morphisms $\cW_iA \to \cW_iB$ are  $\oR\hatHHom_{\cW_0A}$-homotopy equivalences for all $i \ge 0$, then we also have $\oR\hatDDer_{P_k^{ac}, \cT^+_R, L_{\infty}}(B, M) \simeq 0$. 
\end{proposition}

\subsection{Quantisations on derived Artin stacks} 

We are now in a position to generalise Theorem \ref{fildefhochthm1} to stacky CDGAs, and hence Corollary \ref{DMquantcor} to derived Artin $n$-stacks.

\subsubsection{Quantisation of $0$-shifted Poisson structures} 
 
By  \cite[Lemma \ref{DQnonneg-gradedHH}]{DQnonneg}, based on \cite[\S 3]{voronovHtpyGerstenhaber}, $\C\C_R(A)$ is 
equipped with a stacky brace algebra structure.
Since $A$ is commutative, Lemma \ref{involutiveHH} moreover adapts to make  $(\C\C_{R,\oplus}(A),\tau^{\HH})$ (resp. $\cD_{\oplus}^{\poly}(A)$) into a  quasi-involutive stacky a.c. brace algebra.

For  $w \in \Levi_{\GT}^P(\Q)$,  Definition \ref{pwinvdef}  gives an equivalence $p_w$  between stacky quasi-involutive a.c. brace algebras and stacky quasi-involutive a.c. $P_2$-algebras,  by considering the respective algebras in the 
dg category
of cochain complexes of $R$-modules in chain complexes. However, the proof of Theorem \ref{fildefhochthm1} does not immediately adapt to this setting, because functoriality for stacky Hochschild complexes and stacky polyvectors is much more subtle, which is why  we have had to involve $\hatHHom$. 

\medskip
Writing $\Omega^1_{A,\boxempty}$ for the cotangent module associated to the relevant theory (commutative, $\C^{\infty}$ or EFC) as in Corollary \ref{affquantcor}, we have:
\begin{theorem}\label{fildefhochthm2} 
Take a cofibrant stacky $R$-CDGA (resp. stacky $\C^{\infty}$-DGA or stacky  $K$-EFC-DGA)   $A$ with 
%
\begin{itemize}
 \item[($\ddagger$)] cotangent complex  $(\Omega^1_{A, \boxempty})^{\#}$ perfect as an $A^{\#}$-module.
\end{itemize}

Then the quasi-involutively  filtered DGLA underlying the  complex of polydifferential operators  
\[
 (\hatTot\cD^{\poly}_{\oplus}(A)_{[-1]}, \tau^{\HH})
\]
is filtered quasi-isomorphic to the graded DGLA 
\[
\Pol_R(A,0)_{[-1]}:= \bigoplus_{p \ge 0} \hatHHom_A(\Omega^p_{A,\boxempty},A)_{[p-1]}
\] 
of derived polyvectors on $A$, where  the Lie algebra structure is given by  the Schouten--Nijenhuis bracket. 

This quasi-isomorphism depends only on  a choice of even $1$-associator $w \in \Levi_{\GT}^P$, and is natural with respect to
  homotopy \'etale functoriality induced by \cite[\S \ref{DQnonneg-Artindiagsn}]{DQnonneg}, \cite[\S \ref{poisson-bidescentsn}]{poisson} and its $\C^{\infty}$ and EFC analogues.
 
When  $A$ is a cofibrant stacky $R$-CDGA satisfying $(\ddagger)$, the same statements hold for the Hochschild complex 
$
\C\C_{R,\oplus}(A)
$
in place of $\cD^{\poly}_{\oplus}(A)$.
\end{theorem}
\begin{proof}
We adapt the proof of Theorem \ref{fildefhochthm1}. As explained in \S \ref{stackyHHsn}, the cofibrant hypothesis ensures that
Lemma \ref{involutiveHH} adapts to show that $(D^{\poly}_{\oplus}(A), \tau^{\HH})$
 is a quasi-involutive a.c. stacky brace algebra, with a levelwise graded quasi-isomorphism $\mathrm{HKR} \co \gr^{\tau^{\HH}}\cD^{\poly}_{\oplus}(A) \xra{\sim} \cPol(A,0)$.  Since  $\cD^{\poly}_{\oplus}(A) \to \C\C_{R,\oplus}(A)$ is a filtered levelwise quasi-isomorphism in the stacky $R$-CDGA setting, it suffices to focus on $\cD^{\poly}_{\oplus}(A)$ in all settings.
 
For any even associator $w$, the $\infty$-functor $p_w$ of Definition \ref{pwinvdef}  then gives an involutive a.c. stacky $P_2$-algebra $p_w(\cD^{\poly}_{R, \oplus},\tau^{\HH}) $, with its associated graded algebra having a zigzag of levelwise quasi-isomorphisms of $\bG_m$-equivariant stacky  $P_2^{ac}$-algebras to $\Pol(A,0)$. 

Since $(\Omega^1_{A, \boxempty})^{\#} $ is assumed to be perfect, Lemma \ref{Polcotlemma} adapts verbatim to give a levelwise quasi-isomorphism $\oL\Omega^1_{\cPol(A,0)/A} \simeq \cPol(A,0)\ten^{\oL}_A\cW_1\cPol(A,0)$ of $A$-modules. Thus $\cPol(A,0)$ satisfies the conditions of Proposition \ref{weightprop2}. 

By base change, we know that  $\Omega^1_{A, \boxempty}\ten_AA^0 $ must be a perfect $A^0$-module in double complexes, so all but finitely many of the $A^0$-modules $ (\Omega^1_{A, \boxempty}\ten_AA^0)^r$ must be acyclic.
Thus if we take a diagram $D=(A(0) \to A(1) \to \ldots A(m))$ of homotopy formally \'etale surjections between   $(m+1)$ such stacky CDGAs, the hypotheses of  Lemma \ref{hatHHomPollemma} are satisfied, so 
the maps $\cW_i\cPol(D,0) \to \cW_i\cPol(A(0),0)$ are 
 $\oR\hatHHom_{A(0)}$-homotopy equivalences. 
 
 Thus $\cPol(D,0)$ also satisfies the conditions of Proposition \ref{weightprop2},
so   $p_w(\cD^{\poly}_{R, \oplus}(D),\tau^{\HH}) $ satisfies the conditions   of Corollary \ref{defPcor2}, giving a  zigzag of  filtered involutive quasi-isomorphisms  
\[
 \alpha_{w,D}\co p_w \bigcup_i \hatTot(\tau^{\HH}_i \cD^{\poly}_{\oplus}(D)) \simeq \bigoplus_i \hatTot\cW_i\cPol(D,0), 
 \]
natural with respect to all morphisms $(\cD^{\poly}_{\oplus}(D), \tau^{\HH}) \to (\cD^{\poly}_{\oplus}(D'), \tau^{\HH})$ in the $\infty$-category of quasi-involutive a.c. stacky brace algebras, for all such diagrams $D,D'$. 

If we write $\C$ for the category 
of stacky $R$-CDGAs (resp. stacky $\C^{\infty}$-DGAs or stacky  $K$-EFC-DGAs), with $\C_{c, \onto}$ the subcategory of cofibrant objects and surjective morphisms, then as in  \S \ref{hatHHomsn} we have functors $D \mapsto (\cD^{\poly}_{\oplus}(D), \tau^{\HH})$ and $D \mapsto \cPol(D,0)$ on the Grothendieck construction $\int B \C_{c, \onto}$  of the nerve. If we further restrict to   the full subcategory $\C_{c, \onto}^{lfp} \subset \C_{c, \onto}$ on objects satisfying ($\ddagger$), then the natural transformation $\alpha_w$ above gives a natural equivalence between the respective functors from  $\int B\C_{c, \onto}^{lfp}$ to the category of quasi-involutive a.c. $P_2$-algebras.  

 Consider the $\infty$-category $\oL \C^{lfp,\et}$ given by the localisation at levelwise quasi-isomorphisms of the category of homotopy formally \'etale morphisms between stacky $R$-CDGAs (resp. stacky $\C^{\infty}$-DGAs or stacky  $K$-EFC-DGAs) satisfying ($\ddagger$).
By Lemma \ref{intBClemma}, $\oL \C^{lfp,\et}$
arises as a simplicial localisation of $\int B(\C_{c, \onto}^{lfp,\et})$, for the subcategory $\C_{c, \onto}^{lfp, \et} \subset \C_{c, \onto}^{lfp}$ of homotopy formally \'etale morphisms. 
Since 
the maps $\hatTot\cW_i\cPol(D,0) \to \hatTot\cW_i\cPol(A(0),0)$ above are quasi-isomorphisms, 
our totalised functors descend to that localisation, so we also have an equivalence 
\[
\alpha_w \co \bigcup_i \hatTot (\tau^{\HH}_i \cD^{\poly}_{\oplus}(-)) \simeq \bigoplus_i\hatTot\cW_i\cPol(-,0)
\]
of $\infty$-functors from $\oL \C^{lfp,\et}$ to the $\infty$-category of quasi-involutive a.c. $P_2$-algebras.
\end{proof}

\begin{definition}\label{qpoldef2} 
 Given a cofibrant stacky $R$-CDGA, stacky $\C^{\infty}$-DGA or stacky $K$-EFC-DGA $A$, adapting \cite[Definition \ref{DQnonneg-qpoldef}]{DQnonneg}  as  in \cite{DQDG,DStein}, we define the filtered DGLA $ (Q\widehat{\Pol}(A,0)_{[-1]},\tilde{F})$ of quantised polyvectors by setting
 \[
 \tilde{F}^i Q\widehat{\Pol}(A,0):= \prod_{p \ge i} \hatTot \tau^{\HH}_p \cD^{\poly}_{\oplus}(A)\hbar^{p-1};
 \]
 observe the Gerstenhaber bracket satisfies $[\tau^{\HH}_p,\tau^{\HH}_q] \subset \tau^{\HH}_{p+q-1}$, so $[\tilde{F}^i,\tilde{F}^j] \subset \tilde{F}^{i+j-1}$, making $\tilde{F}^2Q\widehat{\Pol}(A,0)_{[-1]}$  into a pro-nilpotent filtered DGLA.

 The space $Q\cP(A,0)$ of  $0$-shifted quantisations of $A$ is then defined (adapting \cite[Definition \ref{DQnonneg-Qpoissdef}]{DQnonneg}) to be 
\[
\Lim _i \mmc(\tilde{F}^2 Q\widehat{\Pol}(A,0)_{[-1]}/\tilde{F}^i).
\]
The subspace $Q\cP(A,0)^{sd} \subset Q\cP(A,0)$ of self-dual quantisations then consists of fixed points for the involution $(-)^* $ given by $\Delta^*(\hbar):= i(\Delta)(-\hbar)$, for the involution  $i$ of Lemma \ref{involutiveHH}. 

These definitions all extend to diagrams $(A(0) \to \ldots \to A(m))$ in place of $A$, giving  homotopy \'etale functoriality as in the proof of Theorem \ref{fildefhochthm2}. 
 \end{definition}

For a strongly quasi-compact derived Artin $n$-stack $\fX$, the space $Q\cP(\fX,0)$  and its variants are defined in \cite[\S \ref{DQnonneg-Artindiagsn}]{DQnonneg} by first taking an Artin $(n+2)$-hypergroupoid resolution $X_{\bt}$ of $\fX$, then applying the left adjoint $D^*$ of the  denormalisation functor $D$, forming
a cosimplicial stacky CDGA $j \mapsto D^* O(X^{\Delta^j})$ with homotopy formally \'etale structure morphisms. This can be thought of as giving a formally \'etale simplicial resolution of $\fX$ by derived Lie algebroids. We then  set
\[
 \cP(\fX,0) := \ho\Lim_{j \in \Delta} \cP(D^* O(X^{\Delta^j}),0), \quad Q\cP(\fX,0) := \ho\Lim_{j \in \Delta} Q\cP(D^* O(X^{\Delta^j}),0);
\]
this is shown to be  well-defined because the constructions are invariant under smooth hypercovers.
Similar constructions work verbatim for EFC and $\C^{\infty}$ analogues. 

These constructions extend beyond the strongly quasi-compact setting, either by allowing the hypercover to  involve disjoint unions of derived affine schemes, or as in \cite{smallet2} by working directly with a functor on  $D_*\fX$ on stacky CDGAs which admits formally \'etale affine hypercovers. Indeed, analogously to \cite[\S \ref{NCpoisson-derNCprestacksn}]{NCpoisson}, this approach allows the definitions to extend to any homogeneous (a.k.a. infinitesimally cohesive on one factor) derived stack with a bounded below cotangent complex.

 In the following lemma note that, as with any Fermat theory, modules over dg $\C^{\infty}$-rings and dg EFC-rings are just modules over the underlying CDGAs.
\begin{lemma}\label{hgpdstackyCDGAlemma}
 If $Y$ is a derived Artin $m$-hypergroupoid, then the stacky CDGA $A:= D^*O(Y)$ satisfies  $(\Omega^1_{A, \boxempty}\ten_{A}A^0)^r \simeq 0$ for all $r>m$. If the associated derived Artin $m$-stack $\fY:=Y^{\sharp}$ has perfect cotangent complex, then $(\Omega^1_A)^{\#}$ is perfect as an $A^{\#}$-module.
 \end{lemma}
\begin{proof}
First, note that 
the double complex $\Omega^1_{A, \boxempty}\ten_{A}A^0$ is just the Dold--Kan normalisation $N_c$ of the cosimplicial chain complex $\Omega^1_{Y, \boxempty}\ten_{O(Y)}A^0$, since both constructions have the same right adjoint. The term $N_c^r$ in cochain degree $r$ is thus isomorphic to the cokernel $\Omega^1_{Y_r/M_{\L^{r,0}}Y}\ten_{O(Y_r)}O(Y_0)$ of the $(r,0)$th partial matching map, which by hypothesis is a trivial cofibration in degrees $r>m$, so  $(\Omega^1_{A, \boxempty}\ten_{A}A^0)^r \simeq 0$.

The map $\cHom_A(\Omega^1_A,N)^0 \to \cHom_A(\Omega^1_A,\sigma^{\le m}N)^0$ to the brutal cotruncation is thus a quasi-isomorphism for all  $A^0$-modules $N$ in double complexes. That quasi-isomorphism extends to all  $A$-modules $N$ in double complexes which are concentrated in non-negative cochain degrees, since $N\cong \Lim_r \sigma^{\ge r}N$ with the quotients being $A^0$-modules.

All the partial matching maps are smooth, so the argument of the first paragraph also shows that the chain complexes $ (\Omega^1_{A, \boxempty}\ten_{A}A^0)^r$ are projective $A^0$-modules in chain complexes, and in particular perfect complexes,  for all $r>0$. Since $Y_0 \to \fY$ is an Artin $m$-atlas and $\fY$ has perfect cotangent complex, so does $Y_0$, meaning that the $A^0$-module  $\Omega^1_{A^0, \boxempty} =(\Omega^1_{A, \boxempty}\ten_{A}A^0)^0$ is also a perfect complex. Since the vanishing result above  implies that  $(\Omega^1_A\ten_AA^0)^{\#}\simeq \bigoplus_{r=0}^m ((\Omega^1_{A, \boxempty}\ten_{A}A^0)^r)^{[-r]}$, we deduce that it is perfect as an $A^0$-module.

The functor $\cHom_A(\Omega^1_A,-)^0$ thus commutes with filtered homotopy colimits of $A^0$-modules in double complexes, and hence  (via quotients of the brutal truncation filtration) with filtered colimits of $A$-modules concentrated in cochain degrees $[0,m]$.  Since cotruncation commutes with filtered colimits,  it follows from the cotruncation property above that $\cHom_A(\Omega^1_A,-)^0$ commutes with all filtered homotopy colimits of $A$-modules, which is equivalent to saying that $(\Omega^1_A)^{\#}$ is perfect as an $A^{\#}$-module.
\end{proof}




\begin{corollary}\label{Artinquantcor} 
Given a  
derived Artin $n$-stack $\fX$ over $R$  with  perfect cotangent complex, 
any even associator $w \in \Levi_{\GT}^P$ gives rise to a map
\[
 \cP(\fX,0) \to Q\cP^{sd}(\fX,0)
\]
from the space of $0$-shifted Poisson structures on $\fX$ to 
the space of self-dual  $E_1$ quantisations of $\fX$ in the sense of  \cite[Definitions \ref{DQnonneg-Qpoissdef}, \ref{DQnonneg-QpoissdefX}]{DQnonneg}.
When $\fX$ is strongly quasi-compact,
these quantisations give rise to curved $A_{\infty}$ deformations $(\per_{dg}(\fX)\llbracket \hbar \rrbracket, \{m^{(i)}\}_{i\ge 0})$ of the dg category $\per_{dg}(\fX) $ of perfect $\O_{\fX}$-complexes,  $\hbar$-semilinearly anti-involutive with respect to the dg endofunctor $\hom_{\sO_{\fX}}(-,\sO_{\fX})$ on $\per_{dg}(\fX) $.

The analogous statements for derived $\C^{\infty}$ and derived analytic Artin $n$-stacks (in the sense of \cite{DStein}) with perfect cotangent complexes also hold.
\end{corollary}
\begin{proof} 
The first statement and its $\C^{\infty}$ and analytic analogues follow immediately by substituting Theorem \ref{fildefhochthm2} in the definitions above, via Lemma \ref{hgpdstackyCDGAlemma}, and passing to homotopy limits.

By \cite[Proposition \ref{DQnonneg-Perprop2}]{DQnonneg},  $E_1$ quantisations of $\fX$ give rise to  curved $A_{\infty}$ deformations of $\per_{dg}(\fX)$, as a consequence of the corresponding statement \cite[Proposition \ref{DQnonneg-Perprop2}]{DQnonneg} for stacky CDGAs. The latter follows by establishing that the natural  restriction map $\CCC_R(\cPer(A)) \to \CCC_R(A)$ is a filtered quasi-isomorphism, for a bi-dg category $\cPer(A)$ of perfect modules associated to $A$  \cite[Definition \ref{DQnonneg-Perdef}]{DQnonneg}. The identity $A \cong A^{\op}$ extends to the dualisation functor $(-)^{\vee}$  on $\cPer(A)$, which is a homotopy involution in the sense that it extends to an action of a simplicial resolution of the group $C_2$. The restriction map above is then equivariant for this homotopy $C_2$-action, so induces equivalences on homotopy fixed points of the respective spaces of deformations.

Explicitly, for strongly self-dual $A$-modules $P_j$, the self-dual condition on a curved  $A_{\infty}$ deformation is that 
the following diagram commutes for each of the $A_{\infty}$ operations $m^{(i)}(\hbar)$:
\[
 \begin{CD}
  \cHom(P_0,P_1)\ten \ldots \ten \cHom(P_{i-1},P_i) @>{m^{(i)}_{P_0, \ldots, P_i}(-\hbar)}>>\cHom(P_0, P_i)\brh \\
  @V{\cong}VV @VV{\cong}V \\
  \cHom(P_i^{\vee},P_{i-1}^{\vee})\ten \ldots \ten \cHom(P_{1}^{\vee},P_0^{\vee}) @>{\pm m^{(i)}_{P_i^{\vee}, \ldots, P_0^{\vee}}}>> \cHom(P_i^{\vee}, P_0^{\vee})\brh. \qedhere
  \end{CD} 
\]
\end{proof}

\begin{remark}
The hypotheses of 
Corollary \ref{Artinquantcor} 
are satisfied by any derived Artin stack locally of finite presentation over the CDGA $R$. When $R=\H_0R$, this includes those underived Artin $n$-stacks $\fX$ which admit Artin $n$-atlases by 
 schemes which are local complete intersections over $R$, in which case the cotangent complex $\oL\Omega^1_{\fX/R}$ is concentrated  in homological degrees $[-n,1]$.  
\end{remark}


 \subsubsection{Quantisation of $1$-shifted co-isotropic structures}\label{coisosn2}

In \cite[Definition 5.14]{MelaniSafronovII}, shifted co-isotropic structures and their quantisations  are defined for stacky CDGAs, and hence derived Artin stacks, in terms of $P_{[n+1,n]}$-algebras and  $BD_{[n+1,n]}$-algebras in the unweighted Tate dg category $\cT^0_{R,dg} \subset \cT^+_{R,dg}$, the full dg subcategory of objects concentrated in weight $0$. Unwinding our definition of a $0$-shifted quantisation  gives a curved strong homotopy $BD_1$-algebra in $\cT^0_{R,dg}$. 

Since  
we have to allow for curvature, we need a more general definition than \cite{MelaniSafronovII}, similar to Definition \ref{qcoisodef}. We then encounter the subtlety that  $\cT^0_{R,dg}$ does not have infinite limits, so we work with the pro-category $\pro(\cT^0_{R,dg})$ in order to accommodate the algebra of quantised polyvectors. 
The monoidal functor $\cT^+_{R,dg} \to \pro(\cT^0_{R,dg})$ sending $V$ to the pro-object $\{\prod_{0 \le i \le j} \cW_jV\}_j$,  allows us to make the following definition.



 
\begin{definition}\label{qcoisodef2} 
 
 Define a $1$-shifted  quantised co-isotropic structure on a morphism $f \co A  \to B$ of stacky CDGAs to consist of:
 \begin{enumerate}
\item  a quantised $0$-shifted Poisson structure $\Delta$ on $B$, 
\item a quantised $1$-shifted Poisson structure on $A$ in the form of
a system  of flat strong homotopy  $BD_2/\hbar^k$-algebras $\tilde{A}/\hbar^k$ in $\cT^0_{R[\hbar]/\hbar^k,dg}$, equipped with a weak equivalence $\tilde{A}/\hbar \to A$  of  commutative algebras in $\cT^0_{R,dg}$, and
\item a strong homotopy morphism  of 
 $BD_2$-algebras in  $\pro(\cT^0_{R,dg})$:
 \[
F \co   \tilde{A} \to T_{\Delta}Q\widehat{\Pol}(B,0):= \{(\prod_{0 \le i \le j} \tau^{\HH}_jD^{\poly}_{\oplus}(B  )\hbar^j, \delta +\{\Delta,-\}, \smile +\{\Delta\}\{\}_2, \{\Delta\}\{\}_{i \ge 3} )\}_j.
 \]
 \end{enumerate}
 
 Define a $1$-shifted  quantised co-isotropic structure on a morphism $f \co \fX \to \fY$ of  derived Artin $n$-stacks   to consist of compatible $1$-shifted  quantised co-isotropic structures on all morphisms $f \co A \to B$ of stacky CDGAs with $A$ and $B$ homotopy formally \'etale over $\fY$ and $\fX$ respectively, using the homotopy formally \'etale functoriality of \S  \ref{hatHHomsn}.
 
 We say that this quantisation is self-dual if $\Delta$ is self-dual, the $BD_2$-algebras  $\tilde{A}$ are given compatible quasi-involutive structures, and $F$ intertwines the involutions.
  \end{definition}

Corollary \ref{defPprop2cor}, and hence the proof of Theorem \ref{defPcor2thm}, gives equivalences in $L^{\cW}\cU_{P_k^{ac}[\hbar^2], R}$, hence \emph{a fortiori} as $P_k^{ac}[\hbar^2]$-algebras in $\cT^+_{R,dg}$. Substituting in these results, the proof of Corollary \ref{coisocor1} adapts immediately to give:

\begin{corollary}\label{coisocor2}
Given a morphism $f \co \fX \to \fY$ of  derived Artin $n$-stacks    such that $\fX$ has perfect cotangent complex,
 every $1$-shifted co-isotropic structure on $f$   admits  self-dual quantisations in the sense of Definition \ref{qcoisodef2}.
 \end{corollary}

For analytic and $\C^{\infty}$ analogues,  the considerations of Remark \ref{EFCcoisormk} apply equally to derived Artin stacks, so the proof of Corollary \ref{coisocor2} immediately gives abstract $BD_2$-algebra quantisations in those settings, which we can strengthen to  genuine quantisations  of the analytic or $\C^{\infty}$  co-isotropic structures by replacing $\cT^0_{dg}$ with the multicategory of polydifferential operators. 
 
%

\bibliographystyle{alphanum}
\bibliography{references.bib}
\end{document}